\documentclass[12pt, twoside]{amsart} 

\usepackage[all]{xy} \usepackage{amsthm} \usepackage{mathrsfs} \usepackage{amssymb} \usepackage{amsmath} 
\usepackage{array, stmaryrd}

\let\oldmarginpar\marginpar
\renewcommand\marginpar[1]{\-\oldmarginpar[\raggedleft\footnotesize #1]%
{\raggedright\footnotesize #1}}

\theoremstyle{definition} \newtheorem{dfn}{Definition}[section]

\theoremstyle{definition} \newtheorem{prop}[dfn]{Proposition}
\theoremstyle{definition} \newtheorem{thm}[dfn]{Theorem}
\theoremstyle{definition} \newtheorem{cor}[dfn]{Corollary}
\theoremstyle{definition} \newtheorem{claim}[dfn]{Claim}
\theoremstyle{definition} \newtheorem{rmk}[dfn]{Remark}
\theoremstyle{definition} \newtheorem{note}[dfn]{Note}
\theoremstyle{definition} \newtheorem{lm}[dfn]{Lemma}
\theoremstyle{definition} 
\theoremstyle{definition} \newtheorem{defprop}[dfn]{Definition-Proposition}
\theoremstyle{definition} 
\theoremstyle{definition} \newtheorem{notation}[dfn]{Notation}
\theoremstyle{definition} \newtheorem{unmarked}[dfn]{}
\theoremstyle{definition} 
\theoremstyle{definition} 
\theoremstyle{definition} 
\theoremstyle{definition} 
\theoremstyle{definition} 
\theoremstyle{definition} \newtheorem{cor-def}[dfn]{Corollary-Definition}

\theoremstyle{definition} \newtheorem*{thm398}{Theorem \ref{.398}}
\theoremstyle{definition} \newtheorem*{prop3902}{Proposition \ref{.3902}}
\theoremstyle{definition} \newtheorem*{prop3901}{Proposition \ref{.3901}}
\theoremstyle{definition} \newtheorem*{thm50149}{Theorem \ref{.50149}}
\theoremstyle{definition} \newtheorem*{cor46}{Corollary \ref{.46}}
\theoremstyle{definition} \newtheorem*{thm557}{Theorem \ref{.557}}

\newcommand{\Unmarked}[2]{\begin{unmarked} \label{#1} 
{#2} \end{unmarked}}

\newcommand{\Definition}[2]{\begin{dfn} \label{#1} 
{#2} \end{dfn}}
\newcommand{\Proposition}[2]{\begin{prop} \label{#1} 
{#2} \end{prop}}

\newcommand{\Corollary}[2]{\begin{cor} \label{#1} 
{#2} \end{cor}}

\newcommand{\Remark}[2]{\begin{rmk} \label{#1} 
{#2} \end{rmk}}

\newcommand{\Lemma}[2]{\begin{lm} \label{#1} 
{#2} \end{lm}}

\newcommand{\spec}{\operatorname{Spec}}

\newcommand{\End}[1]{\operatorname{End} {#1}}
\newcommand{\m}[1]{\mathrm{#1}}
\newcommand{\fk}[1]{\mathfrak{#1}}
\newcommand{\iEnd}{\mathscr{E}\hspace{-.2 ex}nd \,}
\newcommand{\s}[1]{\mathscr{#1}}
\newcommand{\fAut}[1]{\mathbb{A}\mathrm{ut} \, {#1}}
\newcommand{\bb}[1]{\mathbb{#1}}

\newcommand{\la}{\lambda}
\newcommand{\ka}{\kappa}

\newcommand{\sir}{\sigma^{\m{Fil}^r}}
\newcommand{\sid}{\sigma^\m{dcs}}

\newcommand{\nd}{^\m{nd}}
\newcommand{\fl}{^\m{fl}}
\newcommand{\nnd}{^\m{nd}_n}
\newcommand{\nfl}{^\m{fl}_n}

\newcommand{\nn}{\fk{n}_n}
\newcommand{\nnt}{{\nn}_T}
\newcommand{\g}{\fk{g}}

\newcommand{\M}{\cl{M}}
\newcommand{\mnnd}{M_n^\m{nd}}
\newcommand{\Mnnd}{\cl{M}_n^\m{nd}}
\newcommand{\mnndg}{M_n^\m{nd}(\g)}
\newcommand{\Mnndg}{\cl{M}_n^\m{nd}(\g)}
\newcommand{\Mnfl}{\cl{M}_n^\m{fl}}
\newcommand{\Mnflg}{\cl{M}_n^\m{fl}(\fk{g})}
\newcommand{\Mnffl}{\cl{M}_n^\m{ffl}}
\newcommand{\Mnfflg}{\cl{M}_n^\m{ffl}(\fk{g})}

\newcommand{\Mnfnd}{\cl{M}_n^\m{fnd}}
\newcommand{\Mnfndg}{\cl{M}_n^\m{fnd}(\fk{g})}
\newcommand{\mnfnd}{M_n^\m{fnd}}
\newcommand{\mnfndg}{M_n^\m{fnd}(\fk{g})}

\newcommand{\gtoend}{\fk{g} \to \m{End}}
\newcommand{\rgtoende}{r:\gtoend \, E}

\newcommand{\gttonnt}{\g_T\to \nnt}

\newcommand{\gttoiende}{\fk{g}_T \to \iEnd \cl{E}}
\newcommand{\rgttoiende}{r: \gttoiende}
\newcommand{\rgttonnt}{r:\g_T \to \nnt}

\newcommand{\fHom}{\mathbb{H}\mathrm{om}}

\newcommand{\V}{\bb{V}}
\newcommand{\cl}[1]{\mathcal{#1}}
\newcommand{\bld}[1]{\textbf{#1}}
\newcommand{\fEnd}{\mathbb{E}\mathrm{nd}}

\newcommand{\ga}{\mathbb{G}_a}
\newcommand{\cok}{\operatorname{cok}}
\newcommand{\ssf}{\mathcal{O}}

\newcommand{\gr}{\operatorname{gr}}

\newcommand{\gam}{\gamma}
\newcommand{\wt}{\widetilde}
\newcommand{\wtg}{\wt W _\gam}

\numberwithin{equation}{dfn}

\SelectTips{cm}{11}

\begin{document}

\title[Moduli of unipotent representations]{Moduli of nondegenerate unipotent representations in characteristic zero}
\author{Ishai Dan-Cohen}
\date{\today}

\begin{abstract}
With this work we initiate a study of the representations of a unipotent group over a field of characteristic zero from the modular point of view. Let $G$ be such a group. The stack of all representations of a fixed finite dimension $n$ is badly behaved. We introduce an invariant, $w$, of $G$, its \textit{width}, as well as a certain nondegeneracy condition on representations, and we prove that nondegenrate representations of dimension $n \le w+1$ form a quasi-projective variety. Our definition of the width is opaque; as a first attempt to elucidate its behavior, we prove that it is bounded by the length of  a composition series. Finally, we study the problem of gluing a pair of nondegenerate representations along a common subquotient.

\end{abstract}

\maketitle

\section*{Introduction}

The purpose of this paper is to develop an approach to the problem of moduli of representations of a unipotent group over a field of characteristic zero. Fix such a field $k$, a unipotent group $G$ over $k$ and a positive integer $n$. The stack $\cl{M}_n(G)$ of all representations of a fixed dimension $n$ is badly behaved. It is typically not algebraic\footnote{Thanks are due to Anton Geraschenko for helping me understand this fact.}, and its diagonal, albeit representable, is a positive dimensional group whose fiber dimensions can jump in families. We define a nondegeneracy condition which cuts out an immersed substack $\cl{M}^\m{nd}_n(G) \subset \cl{M}_n(G)$. This substack is large in the sense that it contains an open substack, it is algebraic, and its diagonal is flat. It then follows from a higher quotient construction known as \textquotedblleft rigidification" that the fppf sheaf $M^\m{nd}_n(G)$ associated to $\cl{M}\nnd(G)$ is an algebraic space.

We initiate a study of $M^\m{nd}_n(G)$. There is an invariant $w$ of (the Lie algebra of) $G$ which we call the \textit{width}, which singles out a best-case-scenario for constructing moduli. We prove that for $n \le w+1$, $M\nnd(G)$ is quasi-projective. We also give concrete descriptions of $M^\m{nd}_n(G)$ for $n = 2$ and $3$. Finally, we study a tower formed by our moduli spaces.

The problem of constructing a coarse space of a moduli stack has a long and rich history, which may help put our problem in context. However, our problem does not seem to fit the rubric suggested by the methods which have emerged from this history. The functor 
$$\fHom(G, \bb{GL}_n)$$
of homomorphisms to $\bb{GL}_n$ is typically not representable, so a direct application of geometric invariant theory to the action of $\bb{GL}_n$ by conjugation is not possible; the prospects of a more creative application are unclear. Mumford's theory requires a reductive group as part of the input. But in our context, it is more natural to consider instead the action of the group $\bb{B}_n$ of invertible upper triangular matrices on the space
$$X_n^\m{fl}(G) \subset \fHom(G,\bb{U}_n)$$
of upper triangular representations whose canonically associated filtration is a full flag. Even if we insist on an action by $\bb{GL}_n$ on an appropriate space, the stabilizer subgroups will not be reductive. Partial analogs of Mumford's theory for groups which are not reductive are currently under development in works by A. Asok, B. Doran, and F. Kirwan.

Another well known tool is the Keel-Mori theorem. This theorem applies to algebraic stacks with finite diagonal and produces an algebraic space. In our context, plagued by a unipotent action with positive dimensional stabilizers, it appears that the only readily available tool is rigidification. This requires that we restrict attention to representations whose automorphisms are flat, but it has the advantage of producing a sheaf quotient: if we let
$$X\nnd(G) \subset X\nfl(G)$$
denote the locus of representations which satisfy our nondegeneracy condition, then
$$\mnnd(G) = X\nnd(G)/_\m{fppf}\bb{B}_n$$
is a quotient of flat sheaves.


In the case $n\le w+1$ our quasi-projective variety $M\nnd$ is in fact a sheaf quotient in the Zariski topology. This means, in particular, that its field-vlaued points parametrize isomorphism classes of nondegenerate representations defined over the given field. In this regard our nondegeneracy condition is similar to Mumford's stability (as opposed to semi-stability), and, indeed, our notion is related to his. In this case, our construction is quite explicit, relying essentially on the construction of Zariski-local sections for the projection
$$\Mnnd(G) \to \mnnd(G) \,. $$

I now give an outline of the paper together with statements of the main theorems and sketches of proofs. If $\g$ is the Lie algebra of $G$, then finite dimensional representations of $G$ correspond to finite dimensional nilpotent representations of $\g$. With this in mind, we begin by studying the problem of moduli of nilpotent representations of a fixed Lie algebra $\g$ over $k$. In section \ref{Nilp} we recall the definition and first properties of nilpotent representations. In section \ref{Flag} we study \textit{flag}\footnote{In earlier versions of this paper, such representations were called \textit{regular}. } representations: those whose associated filtration is a full flag. We give a criterion for a nilpotent representation to be flag in coordinates and we define various structures associated canonically to a given flag representation $r$. Among these is an ordered set of $n-1$ points of the projectivization $\bb{P}\fk{g}^\m{ab}$ of the abelianization of $\fk{g}$ which we call the \textit{constellation} of $r$.

In section \ref{Aut} we develop a technical tool: the scheme-theoretic Lie algebra $n(r)$ of the unipotent part $U(r)$ of the automorphism group of a flag representation $r: \fk{g}_T \to \iEnd(\cl{E})$ over a general base $T$ over $k$. We observe that $\fAut r = \bb{G}_{m,T} \times U(r)$, that $U(r)$ is isomorphic to $n(r)$ as a scheme and that $n(r)$ is the total space of a module.

In section \ref{Nondeg} we turn to the definition and an initial study of nondegeneracy. Roughly, a nilpotent representation $r: \fk{g} \to \m{End}(E)$ on a vector space $E$ is \textit{nondegenerate} if it is flag, if every subquotient is (by recursion on the dimension of $E$) already nondegenerate, and if among representations satisfying the above two conditions, the dimension of the automorphism group is minimal. The dimension of the automorphism group of a nondegenerate nilpotent representation of dimension $n$ is independent of the choice of nondegenerate nilpotent representation; it thus defines an invariant of $\g$ and $n$ which we denote by $A(\g,n)$. We illustrate the behavior of $A(\g,n)$ for low $n$ with several examples, showing in particular that the possible triples $(A(\g,2), A(\g,3), A(\g,4))$ are $(2,2,2), (2,2,3)$ and $(2,3,4)$. 

In section \ref{Width} we define the \textit{width}, $w$, an invariant of the pronilpotent completion of $\fk{g}$. We prove that nondegenerate representations of dimension not more than $w+1$ satisfy a particularly strong relative of the flag condition (Theorem \ref{.434}). As a corollary, we obtain:
\begin{cor46}
If $\fk{g}$ has descending central series of length $d$ and has width $w$ then $w \le d$.
\end{cor46}
\noindent In particular, a commutative Lie algebra has width one. At the opposite extreme, a free Lie algebra on two or more generators has width $\infty$. 

We let $\cl{M}^\m{nd}_n(\fk{g})$ denote the stack of $n$-dimensional nondegenerate nilpotent representations of $\fk{g}$. The main goal of section \ref{Modnd} is to prove:
\begin{thm398}
The fppf sheaf $\pi_0^\m{fppf} \cl{M}^\m{nd}_n(\fk{g})$ associated to $\Mnndg$ is an algebraic space.
\end{thm398}
\noindent For the proof, we let $\fk{n}_n$ be the Lie algebra of strictly upper triangular $n \times n$ matrices, we observe that the locus $X^\m{fl}_n$ of flag representations of $\g$ is an open subscheme of $\fHom(\fk{g}, \fk{n}_n)$ and that its stack quotient by the action of the Borel is equal to the stack of flag representation. Next we let $X_n^\m{nd} \subset X_n^\m{fl}$ denote the locus of nondegenerate nilpotent representations and we prove that $X^\m{nd}_n \hookrightarrow X^\m{fl}_n$ is an immersion. Here we use $n(r)$ to help in showing that $X^\m{nd}_n$ is compatible with taking infinite unions of rings and we use $n(r)$ again to produce a flattening stratification for the automorphism group of the universal family. From this it follows that $X_n^\m{nd}$ is representable and hence that $\cl{M}^\m{nd}_n(\fk{g})$ is algebraic. Since its inertia is by construction flat, rigidification applies to produce the theorem.

We let
$$M^\m{nd}_n(\fk{g}) := \pi_0^\m{fppf} \cl{M}^\m{nd}_n(\fk{g})$$
and call it \textit{the moduli space of $n$-dimensional nondegenerate nilpotent representations}. We end the section with a brief discussion of the functoriality of our moduli spaces. We show that if $\fk{f}, \g$ are Lie algebras such that for $i=1, \dots, n$, $A(\g,i)=A(\fk{f},i)$, then any surjection
$$\fk{f} \twoheadrightarrow \g$$
gives rise to a closed immersion
$$\mnndg \hookrightarrow \mnnd(\fk{f})\, .$$

Section \ref{Framed} is a variation on section \ref{Modnd}. We define a \textit{framed nondegenerate nilpotent representation} to be a nondegenerate nilpotent representation equipped with a grading-compatible basis for the associated graded vector space. We let $\Mnfndg$ denote the stack of $n$-dimensional framed nondegenerate nilpotent representations of $\g$, we prove that $\pi_0^\m{fppf}\Mnfndg$ is algebraic and we define
$$\mnfndg := \pi_0^\m{fppf}\Mnfndg\,.$$
This gives a modular interpretation of the sheaf quotient of $X\nnd$ by the action of the unipotent radical of the standard Borel.

Section \ref{Low} is devoted to low dimensional examples. We prove:
\begin{prop3902}
The moduli space $M\nd_2(\fk{g})$ of two dimensional nondegenerate nilpotent representations is canonically isomorphic to $\bb{P}\fk{g}^\m{ab}$
\end{prop3902}
\noindent and we give a simple construction of $M_3^\m{nd}$ whose main features may be summarized as follows:
\begin{prop3901}
Let $m$ denote the dimension of $\g$ and let $w$ denote the width of $\g$. If $w=1$ then the moduli space ${M}^\m{nd}_3(\fk{g})$ of three dimensional nondegenerate nilpotent representations is a closed subscheme of a vector bundle of rank $m-1$ over $\bb{P}\fk{g}^\m{ab}$. If $w \ge 2$ then ${M}^\m{nd}_3(\fk{g})$ is a closed subscheme of a vector bundle of rank $m-2$ over the complement of the diagonal of $\bb{P}\fk{g}^\m{ab} \times \bb{P}\fk{g}^\m{ab}$.
\end{prop3901}
\noindent Roughly, the unipotent part of the Borel action produces the vector bundle while the torus projectivizes the product of abelianizations. If $w=1$ then the compatibility condition of \S12 (see below) forces constellations to degenerate, hence restricts all flag representations to the diagonal; otherwise the nondegeneracy condition excises the diagonal.

Section \ref{ModW} is devoted to a proof of
\begin{thm50149}
Let $w$ denote the width of $\fk{g}$. If $n \le w+1$ then $\pi_0 ^\m{ZAR} \Mnndg$ is quasi-projective. Thus
$$\mnndg = \pi_0 ^\m{ZAR} \Mnndg$$
and in particular $\mnndg$ is quasi projective.
\end{thm50149}
\noindent We begin with a concrete criterion for a flag representation $r$ of dimension $n \le w+1$ to be nondegenerate: $r$ is nondegenerate if and only if the points of its constellation are distinct. We use our study of the functoriality of our moduli spaces to reduce to the case that $\g$ is free. Finally, constellations give rise to a map
$$\mnnd(\g) \to (\bb{P}\g^\m{ab})^{n-1}$$
and a concrete construction shows that locally $\mnndg$ has the structure of a vector bundle over the complement of the big diagonal. In carrying out this construction, we first consider $\mnfnd$ together with a framed analog of the constellation and then study the action of the diagonal torus on $\mnfnd$.

It is well known that the functor \textbf{Lie} induces an equivalence of categories from the category of unipotent groups over $k$ to the category of nilpotent Lie algebras over $k$, and that given a unipotent group $G$ with Lie algebra $\g$, \textbf{Lie} also induces an isomorphism from the category of finite dimensional representations of $G$ to the category of finite dimensional nilpotent representations of $\g$. In section \ref{Unip} we generalize the latter statement to include families (as well as infinite dimensions) as follows. Let $\bb{REP}(G)$ denote the fibered category of quasi-coherent representations of $G$ and let $\bb{REP}^\m{nil}(\fk{g})$ denote the fibered category of locally nilpotent quasi-coherent representations of $\fk{g}$. Then we have
\begin{thm557}
The functor
$$\m{\bf{Lie}}: \bb{REP}(G) \to \bb{REP}^\m{nil}(\fk{g})$$
sending a representation to its derivative at the identity is an isomorphism of fibered categories.
\end{thm557}
\noindent This may be known to experts but does not, as far as I can tell, appear in the literature. The main obstacle is that the functor of automorphisms of a quasi-coherent sheaf may not be representable; it is overcome by a careful analysis of the exponential map. This theorem reduces the problem of moduli of representations of $G$ to the problem of moduli of nilpotent representations of $\fk{g}$, and hence to the context of the previous section. We apply all definitions introduced in sections 1--9 to $G$ through its Lie algebra; in particular we define \textit{the moduli space of} $n$-\textit{dimensional nondegenerate representations of} $G$ by
$$\mnnd(G):= \mnnd({\bf Lie}\,G)\,.$$

Since a flag representation has two canonically defined subquotients, the moduli spaces of nondegenerate representations form a tower
\[\xymatrix{
\vdots \ar@<.7ex>[d] \ar@<-.7ex>[d] \\
M\nd_{n+1} \ar@<.7ex>[d]^{p^n_2} \ar@<-.7ex>[d]_{p^n_1}\\
M\nd_n \ar@<.7ex>[d]^{p^{n-1}_2} \ar@<-.7ex>[d]_{p^{n-1}_1}\\
M\nd_{n-1} \ar@<.7ex>[d] \ar@<-.7ex>[d]\\
\vdots
}\]
with
$$p^{n-1}_2 \circ p^n_1 = p^{n-1}_1 \circ p^n_2 \, .$$ In section \ref{Com} we discuss this tower, focussing on a modular answer to the following question: when can two $n$-dimensional nondegenerate representations be glued along a codimension-one subquotient to produce an $n+1$-dimensional flag representation? The answer is given in the form of a closed algebraic subspace $M_n^\m{cnd}$ of $M\nnd {\times}_{p_2, M\nd_{n-1}, p_1} M\nnd$ through which the map
$$M\nd_{n+1} \rightarrow M\nnd \underset{p_2, M_{n-1}\nd, p_1}{\times} M\nd_n$$
factors. This generalizes the role played by the diagonal of $\bb{P}\g^\m{ab} \times \bb{P}\g^\m{ab}$ in the case $w=1$ of Proposition \ref{.3901}.

Finally, in an appendix we give a direct proof of the fact, implied by its quasi-projectivity, that for $n \le w+1$, $\mnnd$ is separated. Of interest here is the reappearance of the our invariant, the width.

This project leads naturally in several directions which I now indicate briefly. Computations performed jointly with Anton Geraschenko reveal that typically $\mnnd(G)$ has multiple components, many singularities not explained by the multiplicity of components and sometimes even generically nonreduced components. These geometric features endow $\mnnd(G)$ with a natural stratification which provides unipotent representations with an intricate discrete invariant and suggests a classification program in the same spirit as classical representation theory. This may lead to a study of representations of an arbitrary algebraic group which mixes the classical theory with a theory of unipotent representations.

On the other hand, our initial interest in this problem came from the hope to formulate a story somewhat similar to that of \cite{simpson} for the unipotent fundamental group in a $p$-adic context. Given a variety over $\bb{F}_p$ satisfying certain hypotheses, the theory of the $p$-adic unipotent fundamental group gives rise to a pair of prounipotent groups $U_\m{cris}$ and $U_\m{\acute et}$ over $\bb{Q}_p$. These groups carry various extra structures, as well as a comparison isomorphism over $B_\m{cris}$ which together reflect arithmetic properties of the variety and which should in turn be reflected in the structure of the moduli space of representations. For instance, there should be an automorphism whose fixed points single out those unipotent isocrystals which support an $F$-structure. For the applications we have in mind in this direction, we would have to compactify the moduli space; the search for a good compactification presents another natural next step for our study.

\subsection*{Aknowledgements}

First and foremost, I would like to thank my two advisors at Berkeley: Arthur Ogus and Martin Olsson. This project was suggested and guided by Olsson; many of the ideas below (and above) arose in conversations with him. During my time at Rice I often met with Brendan Hassett, who kindly treated me as one of his own students. I would like to thank Anton Geraschenko and Matthew Satriano for many helpful conversations. Finally, I would like to thank the math department at Rice for its hospitality during my time there.
 
 \subsection*{Notations and conventions}
 
\begin{unmarked} \label{.005}
We denote by $\fk{n}_n$ the Lie subalgebra of $\fk{gl}_n$ whose elements are the strictly upper triangular $n\times n$ matrices.
\end{unmarked}
 
 \begin{unmarked} \label{.01}
 If $\fk{g}$ is a Lie algebra over a field, we denote by $\fk{g}^{(i)}$ the $i^\m{th}$ term in the descending central series: $\fk{g}^{(1)} = \fk{g}$, $\fk{g}^{(i+1)} = [\fk{g}, \fk{g}^{(i)}]$, and by $\fk{g}^\m{ab}$ the abelianization: $\fk{g}^\m{ab} = \fk{g}^{(1)} / \fk{g}^{(2)}$. The \textit{pronilpotent completion} of $\fk{g}$ is the inverse limit $\varprojlim \fk{g}/\fk{g}^{(n)}$. We say that $\fk{g}$ is \textit{nilpotent} if there exits an $n$ such that $\fk{g}^{(n)} = 0$ and \textit{pronilpotent} if $\fk{g}$ is equal to its pronilpotent completion. If $F$ is a vector space, $\fk{n}(F)$ denotes the free pronilpotent Lie algebra on $F$. It is characterized by the mapping property $\m{Hom}_\m{\bld{Lie}}(\fk{n}(F), \fk{n}) = \m{Hom}_\m{\bld{Vect}}(F, \fk{n})$ for every nilpotent Lie algebra $\fk{n}$, and may be constructed as the pronilpotent completion of the free Lie algebra on $F$. 
\end{unmarked}

\begin{unmarked}\label{.012}
The word $\textit{filtration}$ will always refer to an increasing filtration indexed by the natural numbers.
\end{unmarked}

\begin{unmarked} \label{.015}
Let $A$ be a ring, $E$ an $A$-module, $\m{Fil}$ a filtration of $E$ by submodules and $\phi$ an endomorphism of $E$. We say that $\phi$ \textit{is nilpotent with respect to} $\m{Fil}$ if for each $i \ge 1$, $\phi (\m{Fil}_i{E}) \subset \m{Fil}_{i-1}{E}$, and we write $\fk{n}_\m{Fil}{E}$ for the space of all such endomorphisms.

More generally, if $(T, \ssf_T)$ is a ringed space, $\cl{E}$ an $\ssf_T$-module and $\m{Fil}$ a filtration of $\cl{E}$ by submodules, we write $\fk{n}_\m{Fil}\cl{E}$ for the sheaf of endomorphism of $\cl{E}$ nilpotent with respect to $\m{Fil}$.
\end{unmarked}

\begin{unmarked} \label{.016}
Let $(T, \ssf_T)$ be a ringed space, $g$ a sheaf of Lie algebras over $\ssf_T$ and $r: g \to \iEnd \cl{E}$ a representation on an $\ssf_T$-module $\cl{E}$. We define the $0$\textit{-eigenspace of} $r$, an $\ssf_T$-submodule $\cl{E}^0$ of $\cl{E}$, by
\[ \cl{E}^0(U) := \{ e \in \cl{E}(U) \; | \; (rv)e=0 \; \forall v\in g(U) \} \]
\end{unmarked}

\begin{unmarked} \label{.02}
If $T$ is a scheme and $\cl{F}$ is a quasi-coherent sheaf, then the \textit{contravariant total space}, denoted $\V\cl{F}$, of $\cl{F}$ is defined by
$$\V\cl{F}(f: T' \to T) = \m{Hom}_{\ssf_{T'}}(f^*\cl{F}, \ssf_{T'})$$
(\cite[1.7.8]{egaii}) and if $\phi:\cl{E} \rightarrow \cl{F}$ is a map of quasi-coherent sheaves then $\V \phi$ denotes the induced map $\V \cl{F} \rightarrow \V \cl{E}$. If $\cl{E}, \cl{F}$ are locally free of finite rank, then the kernel of $\phi$ \textit{regarded as a map of vector groups} is defined by the Cartesian square
\[\xymatrix{
\ker \V\phi^\lor \ar[r] \ar[d]	&
\V \cl{E}^\lor \ar[d]^{\V \phi^\lor}	\\
T \ar[r]	&
\V \cl{F}^\lor	}
\]
where the arrow at the bottom is the zero section. Although $\ker \V \phi^\lor$ may not be a vector group, it is the contravariant total space of a module. Indeed, it suffices to check this under the assumption that $T = \spec A$ is affine in which case we write $E$ and $F$ for the modules of global sections, we let $Q = \cok (\phi^\lor)$ and we observe that for any $A$-algebra $B$ we have
$$\m{Hom}_A (Q,B) = \ker (\m{id}_B \otimes \phi)$$
as in the following diagram.
\[\xymatrix{
0 \ar[r]	&
\m{Hom}(Q,B) \ar[r]	&
\m{Hom}(E^\lor, B) \ar[r] \ar@{=}[d]	&
\m{Hom}(F^\lor, B) \ar@{=}[d]	\\
	&
	&
B \otimes E \ar[r]_{\m{id}_B \otimes \phi}	&
B \otimes F	}
\]
\end{unmarked}

\begin{unmarked} \label{.03}
In general, when working over a scheme $T$, we use black board bold symbols to denote presheaves on the category of affine $T$-schemes, and calligraphic symbols to denote presheaves on the small Zariski site of $T$. So, for example, if $r:\fk{g} \rightarrow \iEnd(\cl{F})$ is a representation of a Lie algebra on a quasi-coherent sheaf over $T$, then $\fEnd(r)$ denotes the functor
$$(f:T' \rightarrow T) \mapsto \m{End}(f^*r)$$
and $\iEnd(r)$ denotes its restriction to $X_\m{zar}$. The latter is quasi-coherent but the former may not be.
\end{unmarked}

\begin{unmarked} \label{.04}
When working over an affine scheme $T= \spec A$ we use a plain font to denote the module of global sections of a quasi-coherent sheaf; thus $E = \Gamma(T, \cl{E})$. On the other hand, when $\fk{g}$ is a Lie algebra over a ring $A$ and $T=\spec A$, we use $\fk{g}$ again for the sheaf of Lie $\ssf_T$-algebras associated to $\fk{g}$ when conflating the two imposes no danger.
\end{unmarked}

\Unmarked{.042}{ Let $T$ be a scheme, $\cl{F}$ an $\ssf_T$-module and $\cl{E} \subset \cl{F}$ a submodule. Following Lang, we say that $\cl{F}$ is a \textit{vector sheaf} if it is locally free of finite rank. Assuming this to be the case, we say that $\cl{E}$ is a \textit{vector subsheaf} if the quotient module $\cl{F}/\cl{E}$ is a vector sheaf. }

\begin{unmarked} \label{.043}
We employ the convention that when no topology is mentioned, the indiscrete topology (that is, the topology whose only coverings are the isomorphisms) is assumed. Notationally, this means the following.

Let $\mathbf{C}$ be a category. If $X$ is a presheaf on $\mathbf{C}$ and $G$ is a group presheaf acting on $X$, we write $X/G$ for the presheaf quotient and $[X/G]$ for the associated fibered category. Thus for $T \in \mathbf{C}$,
$$(X/G)(T)=X(T)/G(T)$$
and $[X/G](T)$ is the groupoid whose objects are the elements of $X(T)$ and whose morphisms $x\rightarrow y$ are those elements of $G(T)$ such that $gx=y$.

Now let $\tau$ be a topology on $\mathbf{C}$ and suppose $X$ is a $\tau$ sheaf and $G$ is a group $\tau$ sheaf acting on $X$. Then we write $X/_\tau G$ for the sheaf quotient with respect to $\tau$ and $[X/_\tau G]$ for the stack quotient with respect to $\tau$. Thus $X/_\tau G$ is the $\tau$ sheaf associated to $X/G$ and $[X/_\tau G]$ is the $\tau$ stack associated to $[X/G]$.
\end{unmarked}

\Unmarked{.0431}{Continuing with the situation of \ref{.043} suppose $\cl{X}$ is a fibered category over $\mathbf{C}$. Then $\pi_0\cl{X}$ denotes the presheaf associated to $\cl{X}$: $(\pi_0\cl{X})(T) = \pi_0(\cl{X}(T))$ is the set of isomorphism classes of objects of $\cl{X}(T)$. If $\mathbf{C}$ is again endowed with a topology $\tau$ then we write $\pi_0^\tau\cl{X}$ for the $\tau$ sheaf associated to $\pi_0\cl{X}$.}

\begin{unmarked} \label{.044}
We denote by $\m{\textbf{Set}}$ the category of sets, and by $\m{\textbf{Ring}}$ the category of rings. If $T$ is a scheme we denote by $\m{\textbf{Aff}}(T)$ the catetory of affine schemes over $T$. If $T=\spec A$ is affine, we sometimes write $\m{\textbf{Aff}}(A)$ instead of $\m{\textbf{Aff}}(T)$.
\end{unmarked}

\Unmarked{.045}{We remind the reader that a quasi-coherent sheaf on the small Zariski site of a scheme $T$ extends uniquely to a quasi-coherent sheaf on the big Zariski site of $T$, so there is usually no danger in conflating the two in our notation. Nevertheless, we find it useful to reserve a special notation for the big structure sheaf $o_T: \mathbf{Aff}(T) \to \mathbf{Ring}$ which sends $T' \mapsto \Gamma(T', \ssf_{T'})$. }

\begin{unmarked} \label{.05}
We denote the group of upper triangular invertible $n\times n$ matrices by $\bb{B}_n$, its subgroup of $n\times n$ upper triangular matrices with $1$'s on the diagonal by $\bb{U}_n$ and the $n$-dimensional torus by $\bb{T}_n$. Thus 
$$\bb{B}_n=\bb{T}_n\ltimes\bb{U}_n$$
\end{unmarked}

\begin{unmarked} \label{.06}
If $(T,\ssf_T)$ is a ringed space, $\cl{E}$ is an $\ssf_T$-module, $\m{Fil}$ is a filtration by $\ssf_T$-submodules and $\psi$ is an automorphism of $\cl{E}$, we say that $\phi$ is \textit{unipotent with respect to} $\m{Fil}$ if $\psi$ respects $\m{Fil}$ and $\gr^\m{Fil} \psi = \m{id}_{\gr^\m{Fil} \cl{E}}$.
\end{unmarked}

\section{Nilpotence}
\label{Nilp}

Fix a field $k$ and a Lie algebra $\g$ over $k$, assumed to be finitely generated. We begin by discussing the definition of a nilpotent representation. A more standard definition, which we do not need here, is equivalent to ours through Engel's theorem (\cite[1.35]{knapp}). 
We limit ourselves to making statements at the level of generality needed in the sequel.

\begin{dfn}\label{.11}
A representation $r:\fk{g}\rightarrow\End{E}$ on a vector space $E$ is \textbf{nilpotent} if $E$ is finite dimensional and if either $E=0$ or $E^0$ is nonzero and, by recursion on $\dim E$, $\m{E}/\m{E}^0$ is nilpotent.
\end{dfn}

\begin{dfn} \label{.111}
Let $T$ be a $k$-scheme and let $\rgttoiende$ be a representation of $\g$ on a quasi-coherent sheaf $\cl{E}$ over $T$. We define the \textbf{associated filtration}, denoted $\m{Fil}^r$, of $\cl{E}$ by
$$\m{Fil}^r_0 \cl{E} =0$$
and
$$\m{Fil}^r_{n+1}\cl{E}=\tau^{-1}_n((\cl{E}/\m{Fil}^r_n\cl{E})^0) \, ,$$
where
$$\tau_n:\cl{E} \to \cl{E}/\m{Fil}_n^r\cl{E}$$
is the projection and the $0$-eigenspace $(\cl{E}/\m{Fil}^r_n\cl{E})^0$ is defined as in \ref{.016}. When there is no risk of confusion, we drop the $r$ from the notation, and we sometimes write $\cl{E}_n$ instead of $\m{Fil}_n\cl{E}$. We also write $r_i$ for the subrepresentation
$$\fk{g}_T \to \iEnd(\m{Fil}^r_i\cl{E})$$
and $r^j$ or $r/r_j$ for the quotient representation
$$\fk{g}_T \to \iEnd(\cl{E}/\m{Fil}^r_j\cl{E})$$
and finally, for $j \le i$ we write $r^j_i$ for the subquotient
$$\fk{g}_T \to \iEnd(\m{Fil}^r_i\cl{E}/\m{Fil}^r_j\cl{E}) \, .$$
\end{dfn}

\begin{rmk} \label{.1113}
Let $T$ be a $k$-scheme and let $\rgttoiende$ be a representation of $\fk{g}$ on a quasi-coherent sheaf $\cl{E}$ over $T$. We note the following formula for $0\le i\le j$ and $0\le k\le j-i$:
\[ \m{Fil}_k^{r^i_j}(\cl{E}_j/\cl{E}_i)=\cl{E}_{i+k}/\cl{E}_i \; .\]
\end{rmk}

\begin{lm} \label{.1114}
Let $r:\fk{g} \to \End E$ be a representation of $\g$ on a vector space $E$ over $k$. Then $r$ factors through $\fk{n}_{\m{Fil}^r}E$ (\ref{.015}).
\end{lm}

\begin{proof}
Fixing $v \in \fk{g}$, $i \ge1$ and $e \in \m{Fil}^r_iE$, we are to show that $(rv)e \in \m{Fil}^r_{i-1}E$; since, in the notation of \ref{.111}, $\tau_{i-1}$ is a morphism of representations, we have
$$\tau_{i-1}((rv)e) = (r^{i-1}v)(\tau_{i-1}e) =0$$
whence $(rv)e \in \m{Fil}^r_{i-1}E$.
\end{proof}

\begin{dfn} \label{.112}
Given a ring $B$, a module $F$ over $B$ and a filtration $\m{Fil}$, we say that $\m{Fil}$ is \textit{exhaustive} if
$$\bigcup_i \m{Fil}_iF =F \, .$$
\end{dfn}

\begin{prop} \label{.113}
Let $r:\fk{g} \to \End{E}$ be a representation of $\fk{g}$ on a finite dimensional vector space $E$. The following conditions are equivalent:
\begin{itemize}
\item[(i)] $r$ is nilpotent.
\item[(ii)] $\m{Fil}^r$ is exhaustive.
\item[(iii)] $r$ factors through $\fk{n}_\m{Fil} E$ for some exhaustive filtration $\m{Fil}$ of $E$.
\end{itemize}
\end{prop}

\begin{proof}
(i $\Rightarrow$ ii) Suppose $r$ is nilpotent. Then $r^i$ is nilpotent for all $i$; indeed, if we assume for an induction on $i$ that $r^i$ is nilpotent, since
$$E/E_{i+1}=(E/E_i)/((E/E_i)^0) \, ,$$
it follows that $E/E_{i+1}$ is nilpotent. This implies that for each $i$, either $E_i = E$ or $E_i \neq E_{i+1}$. Since $E$ is finite dimensional, there exists an $i$ such that $E_i=E$, from which the conclusion follows.

(ii $\Rightarrow$ iii) This follows from \ref{.1114} by setting $\m{Fil}:=\m{Fil}^r$.

(iii $\Rightarrow$ i) If $\dim E=0$ then $r$ is nilpotent, indeed. Fix a positive integer n and assume for an induction on $n$ that (iii) implies (i) whenever $\dim E=n$. Let $\m{Fil}$ be an exhaustive filtration on $E$, suppose $r$ factors through $\fk{n}_\m{Fil}E$ and suppose $\dim E =n+1$. If $i$ is the smallest number such that $\m{Fil}_iE\neq0$, then $\m{Fil}_iE \subset E^0$ so $E^0 \neq0$. It remains to show that the inductive hypothesis may be applied to the quotient representation $r^1:\fk{g} \to \End (E/E^0)$. Denote by $\tau_1:E\to E/E^0$ the projection and for each $j$, let $_{\tau_1}\m{Fil}_j(E/E^0)$ be the image of $\m{Fil}_jE$ in $E/E^0$. Given $j$ an arbitrary natural number, $e \in \m{Fil}_{j+1}E$ and $v\in\fk{g}$ we have
$$(r^1v)(\tau_1e) = \tau_1((rv)e) \, ;$$
since $(rv)e \in \m{Fil}_jE$, it follows that $(r^1e)(\tau_1e) \in {_{\tau_1}\m{Fil}_j}(E/E^0)$. Thus $r^1$ factors through $\fk{n}_{_{\tau_1}\m{Fil}}(E/E^0)$; since ${_{\tau_1}\m{Fil}}$ is exhaustive, the induction hypothesis applies as hoped to conclude that $r^1$ is nilpotent.
\end{proof}

\begin{cor} \label{.114}
Let $\rgtoende$ be a representation of $\g$ on a finite-dimensional vector space $E$. If $r$ is nilpotent then so is every subquotient.
\end{cor}

\begin{proof}
It is sufficient to consider subrepresentations and quotient representations separately. Suppose $r:\fk{g} \to \End E$ is nilpotent and let $\m{Fil}$ be an exhaustive filtration such that $r$ factors through $\fk{n}_\m{Fil}E$ as in \ref{.113} (iii). If $\iota: E' \hookrightarrow E$ is a subrepresentation, define $\m{Fil}'$ by
$$\m{Fil}'_iE' := \iota^{-1}\m{Fil}_iE$$
and if $\pi:E \twoheadrightarrow E'$ is a quotient representation, define $\m{Fil}'$ by
$$\m{Fil}'_iE' := \pi(\m{Fil}_iE)$$
as in the proof of (iii $\Rightarrow$ ii) in \ref{.113}. Either way, $r':\fk{g} \to \End E'$ factors through $\fk{n}_{\m{Fil}'}E'$: in the case of a quotient representation this was verified in the proof of (iii $\Rightarrow$ ii) in \ref{.113}; the case of a subrepresentation is similar.
\end{proof}

\begin{dfn} \label{.115}
If $r:\fk{g} \to \End E$ is a nilpotent representation and $\m{Fil}$ is an exhaustive filtration such that $r$ factors through $\fk{n}_\m{Fil}E$, we say that $r$ is \textbf{nilpotent with respect to} $\m{Fil}$. Note that $r$ is nilpotent with respect to $\m{Fil}$ if and only if $\m{Fil}^r$ is subordinate to $\m{Fil}$.
\end{dfn}


\section{Flag representations}
\label{Flag}

We continue to work with a finitely generated Lie algebra $\g$ over a field $k$.

\begin{dfn} \label{.116}
A nilpotent representation $r:\fk{g} \rightarrow \End{E}$ is \textbf{flag} if $\m{Fil}^r$ is a (full) flag. 
\end{dfn}

We begin our study of flag representations with a characterization in coordinates.

\begin{dfn} \label{.11633}
Given a representation $r:\fk{g} \to \fk{gl}_n$ of $\fk{g}$ on $k^n$, we denote the composite
$$\fk{g} \to \fk{gl}_n \to k$$
of $r$ with the $(i,j)^\m{th}$ standard projection of $\fk{gl}_n$ by $\la^r_{i,j}$ and call it \textbf{the full} $(i,j)^\m{th}$ \textbf{matrix entry of} $r$. When there is no risk of confusion, we drop the $r$ from the notation.
\end{dfn}

\begin{prop} \label{.11634}
If $r:\fk{g}\to\fk{n}_n$ is a representation on $k^n$, nilpotent with respect to the standard flag, then $r$ is flag if and only if in the notation of \ref{.11633} $\la_{i,i+1}$ is surjective for each $i=1,\dots,n-1$.
\end{prop}

The proof follows (\ref{.116341} -- \ref{.11636}).

\begin{unmarked} \label{.116341}
Let $s: \fk{g} \to \fk{n}_m$ be a representation on $k^m$, nilpotent with respect to the standard flag, suppose $\la^s_{i,i+1}$ is surjective for $i=1,\dots,m-1$ and let $x=(x_1,\dots,x_n)$ be an element of the $0$-eigenspace of $s$. We claim that $x_l=0$ for $l=2,\dots,m$. For $v\in\fk{g}$ and $i=1,\dots,m$ we have 
\begin{equation}
((rv)x)_i=\sum_{i<j\le m}(\la^s_{i,j}v)x_j \label{.116342}
\end{equation}
This family of equations specializes to
\begin{equation}
(\la^s_{m-1,m}v)x_m=0 \label{.116343}
\end{equation}
when $i=m-1$. Since $\la^s_{m-1,m}$ is surjective there exists a $v\in\fk{g}$ such that $\la^s_{m-1,m}v=1$; plugging in to \ref{.116343} produces $x_m=0$, which is the base case for a descending induction on $l$. Suppose $x_j=0$ for $j=l+1,\dots,m$. Setting $i$ equal to $l-1$ in \ref{.116342} then produces $(\la^s_{l-1,l}v)x_l=0$; plugging in a $v\in\fk{g}$ such that $\la^s_{l-1,l}v=1$ yields $x_l=0$ and the claim follows. We conclude that the $0$-eigenspace of $s$ is equal to the first step in the standard flag.
\end{unmarked}

\begin{unmarked} \label{.11635}
\setcounter{equation}{0}
Returning to the notation of the proposition, suppose $\la^r_{i,i+1}$ is surjective for $i=1,\dots,n-1$. We note first that $\m{Fil}^r_0 k^n =0$ agrees with step zero of the standard flag $\m{Fil}^\m{st}$. Fix a positive integer $l$ and assume for an induction on $l$ that $\m{Fil}^r_l k^n =\m{Fil}^\m{st}_l k^n$. Then for $i=1, \dots, n-l$, $\la^{r^l}_{i,i+1}$ is surjective. Hence paragraph \ref{.116341} applies with $s:=r^l$ to conclude that $(k^{n-l})^0 = \m{Fil}^\m{st}_1 k^{n-l}$ from which it follows that $\m{Fil}^r_{l+1} k^n =\m{Fil}^\m{st}_{l+1} k^n$, as hoped.
\end{unmarked}

\begin{unmarked} \label{.11636}
\setcounter{equation}{0}
For the converse suppose $\la_{i,i+1}=0$ and assume for a contradiction that $r$ is nevertheless flag. Since $r$ is nilpotent with respect to the standard flag, $\m{Fil}^r$ is subordinate to the standard flag. Since $\m{Fil}^r$ is a full flag, it follows that $\m{Fil}^r$ is equal to the standard flag. Then $k^n/\m{Fil}^r_{i-1}k^n = k^n/k^{i-1}$ is the quotient of $k^n$ by the $(i-1)^\m{st}$ term of the standard flag. The resulting representation $r^{i-1}:\fk{g}\to\fk{n}_{n-i+1}$ has $\la_{12}^{r^{i-1}}=0$. It follows that the $0$-eigenspace of $r^{i-1}$ contains step two of the standard flag, hence that $\m{Fil}^r_{i-1}k^n$ has codimension no less than two in $\m{Fil}^r_ik^n$, a contradiction, which concludes the proof.
\end{unmarked}

\begin{prop} \label{.431}
Let $r:\fk{g}\to\End E$ be a nilpotent representation of $\g$ on a vector space $E$. If $r$ is flag then every subquotient is of the form $r^i_j$ for some $0 \le i \le j$. Moreover, every such subquotient is itself flag.
\end{prop}

The proof follows (\ref{.43} -- \ref{.4312}).

\begin{lm} \label{.43}
If $r: \fk{g} \rightarrow \m{End} \; E$ is a flag representation and $x$ is an element of $E_l$ not in $E_{l-1}$ then there exist a $v \in \fk{g}$ such that $(rv)x$ is an element of $E_{l-1}$ not in $E_{l-2}$.
\end{lm}

\begin{proof}
Indeed, every $v \in \fk{g}$ satisfies $(rv)x \in E_{l-1}$ by definition; if, moreover, every $v\in\fk{g}$ satisfies $(rv)x \in E_{l-2}$  then $x \in E_{l-1}$ by definition.
\end{proof}

\begin{unmarked} \label{.4311}
We consider subrepresentations first. Fixing a subrepresentation $E'$, and an element $x\in E'$, it is enough to assume $x \in E_l \setminus E_{l-1}$ and show $E' \supset E_l$. By \ref{.43}, there exists an element $x_i \in (E_i \setminus E_{i-1}) \cap E'$ for all $i \le l$. The sequence of elements $\{x_i\}_i$ then forms a basis for $E_l$.
\end{unmarked}

\begin{unmarked} \label{.4312}
Since every subrepresentation of a flag representation $r$ is of the form $r_j$, every quotient representation is of the form $r^i$. It follows from \ref{.1113} that every subrepresentation is itself flag and hence that every subquotient is of the form $r^i_j$. Finally, it follows from the same paragraph that each such subquotient is again flag.
\end{unmarked} 

\begin{rmk} \label{.116333}
Let $r:\fk{g} \to \fk{n}_n$ be a flag representation of $\fk{g}$ on $k^n$, nilpotent with respect to the standard flag. We note the formula
\[ \la^{ r^l_m }_{i,j} = \la^r_{l+1,l+j} \]
for any $0 \le l \le m \le n$ and $1 \le i < j \le m-l$.
\end{rmk}

\begin{dfn} \label{.42}
We denote the $s^\m{th}$ graded piece $\fk{g}^{(s)}/\fk{g}^{(s+1)}$ of the descending central series of $\fk{g}$ by $\fk{h}_s$. Let $r: \fk{g} \rightarrow \m{End}(E)$ be a flag representation of $\fk{g}$ on a vector space $E$ of dimension $n$. We denote by $L^r_i$ the $i^\m{th}$ line $\m{gr}_i^{\m{Fil}^r}(E)$ associated to $r$. When there is no risk of confusion, we drop the $r$ from the notation.

For each $1\le i < j \le n $, $r$ defines a map
$$\ka_{i,j}^r:\fk{h}_{j-i} \rightarrow \m{Hom}(L_j, L_i)$$
which we call \textbf{the canonical} $(i,j)^\m{th}$ \textbf{matrix entry of} $r$. When there is no risk of confusion, we drop the $r$ from the notation.
\end{dfn}

\begin{unmarked} \label{.4222}
Continuing with the situation and the notation of \ref{.42}, we note that the canonical $(i,j)^\m{th}$ matrix entry of $r$ is related to the full $(i,j)^\m{th}$ matrix entry of $r$ as follows: the choice of a filtered isomorphism $\phi: E \rightarrow k^n$ gives rise to a trivialization of the line $\m{Hom}(L_j,L_i)$ as well as an extension of the canonical $(i,j)^\m{th}$ matrix entry from $\fk{g}^{(j-i)}$ to all of $\fk{g}$; together, these produce the full $(i,j)^\m{th}$ matrix entry, as shown in the following diagram.
\[\xymatrix{
\fk{g} \ar[r] \ar `u[r] `[rr]+(7,0) `[rrddd]^{\la_{i,j}}   [rrddd] 							& 
\fk{n}_\m{Fil}(E) \ar[r]^\cong 								& 
\fk{n}_n 												\\
\fk{g}^{(s)} \ar[r] \ar@{^{(}->}[u] \ar@{->>}[d] 					& 
\fk{n}_\m{Fil}^{(s)}(E) \ar[r]^\cong \ar@{^{(}->}[u] \ar@{->>}[d] 		& 
\fk{n}^{(s)}_n \ar@{^{(}->}[u] \ar@{->>}[d]						\\
\fk{h}_s \ar[r] \ar[dr]_{\ka_{i,j}} 						& 
\fk{n}^{(s)}_\m{Fil}(E) / \fk{n}^{(s+1)}_\m{Fil}(E) \ar[r]^-\cong \ar[d] 	& 
k^{n-s} \ar[d] 											\\
													& 
\m{Hom}(L_j, L_i) \ar[r]^-\cong								& 
k													}
\]
\end{unmarked}

\begin{unmarked} \label{.4231}
Continuing with the situation and the notation of \ref{.42}, we note that there is a canonical isomorphism
\[ L^{ r^l_m }_k = L^r_{l+k}\]
and a corresponding equality (through the above isomorphism)
\[ \ka^{ r^l_m }_{i,j} = \ka^r_{l+i,l+j} \]
for any $0 \le l \le m \le n$ and $1 \le i < j \le m-l$.
\end{unmarked}

We now explain how the canonical matrix entries of a flag representation interact with the Lie bracket (\ref{.4233} -- \ref{.44}).

\begin{unmarked} \label{.4233}
For each $s\ge1$, we denote the projection $\fk{g}^{(s)}\twoheadrightarrow\fk{h}_s$ by $\sigma_s^\m{dcs}$. For $s,t \ge1$, the bracket in $\fk{g}$ restricts to a bilinear map
$$\fk{g}^{(s)}\times\fk{g}^{(t)}\to\fk{g}^{(s+t)} \, ;$$
its composite with $\sid_{s+t}$ factors through a bilinear map
$$\fk{h}_s \times \fk{h}_t \to \fk{h}_{s+t}$$
which we denote by $(u,v) \mapsto \{u,v\}$. (Although we have no use for this in the sequel, we remark that thus defined, $\{\cdot,\cdot\}$ endows $\bigoplus_s\fk{h}_s$ with the structure of a split-nilpotent Lie algebra, that is, a graded nilpotent Lie algebra whose grading splits the descending central series.)
\end{unmarked}

\begin{unmarked} \label{.44}
Given a flag representation $r:\fk{g}\to\End E$ and integers $1 \le i < j < k \le n$ there are three natural maps
$$\fk{h}_{j-i} \otimes \fk{h}_{k-j} \rightarrow \m{Hom}(L_k, L_i)$$
defined by the bracket in $\fk{g}$ and by composing the canonical matrix entries of $r$ in either order:
\[\small\xymatrix @C=-16mm @R=16mm {
&
\fk{h}_{j-i} \otimes \fk{h}_{k-j}= \fk{h}_{k-j} \otimes \fk{h}_{j-i}  \POS p+(18,-3) \ar[dr] | {\ka_{i, k-j+i} \otimes \ka_{k-j+i, k}} \POS p+(-18,-3) \ar   [dl] | {\ka_{i,j} \otimes \ka_{j,k}} \POS p+(-9, -3) \ar@{->>} [d] | { \{ \cdot , \cdot \} }   &
 \\
\m{Hom}(L_j, L_i) \otimes \m{Hom}(L_k, L_j) \ar[dr]_\circ &
\fk{h}_{k-i} \ar[d]|{\ka_{i,k}} &
\m{Hom}(L_{k-j+i}, L_i) \otimes \m{Hom}(L_k, L_{k-j+i}) \ar[dl]^{\circ'} \\
 &
\m{Hom}(L_k, L_i) &
\\
}\]
They are related as follows:
\[\ka_{i,k} \cdot \{\cdot,\cdot\}=\circ\cdot\ka_{i,j}\otimes\ka_{j,k}-\circ'\cdot\ka_{i,k-j+i}\otimes\ka_{k-j+i,k}\]

This is purely formal. For $0\le i \le n = \dim E$, we denote by $\sir_i$ the projection $E_i \twoheadrightarrow L_i$. In order to simplify notation, for $v\in\fk{g}$ and $x\in E$ we abbreviate $(rv)x$ by $vx$. To verify the equality, we fix $u\in\fk{g}^{(j-i)}, v\in \fk{g}^{(k-j)}$ and $x\in E_k$, we evaluate both side of the equation on the generator $(\sid_{j-i}u)\otimes(\sid_{k-j}v)$ of $\fk{h}_{j-i}\otimes\fk{h}_{k-j}$, we evaluate the resulting maps $L_k \to L_i$ on the element $\sir_k x$ of $L_k$ and we compute:
\begin{align*}
\big( \ka_{i,k}\{\sid_{j-i}u, &\sid_{k-j}v\} \big)(\sir_k x )\\
&= \big( \ka_{i,k} \sid_{k-i} [u,v] \big)(\sir_k x)  \\
&= \sir_i (uvx-vux) \\
&= \sir_i(uvx) - \sir_i(vux)  \\
&= \big( \ka_{i,j}(\sid_{j-i} u) \big)(\sir_j vx) - \big(\ka_{i, k-j+i}(\sid_{k-j} v) \big)(\sir_{k-j+i} ux)  \\
&= \Big( \big( \ka_{i,j}(\sid_{j-i} u) \big)  \big( \ka_{j,k} ( \sid_{k-j} v ) \big)\\
& \qquad\qquad - \big(\ka_{i, k-j+i}(\sid_{k-j} v) \big) \big( ( \ka_{k-j+i,k} ( \sid_{j-i} u ) \big) \Big)  ( \sir_k x ) 
\end{align*}
concluding that both sides of the equation do indeed produce the same result.
\end{unmarked}

\begin{dfn} \label{.4433}
If $r: \fk{g} \rightarrow \m{End}(E)$ is a flag representation on a vector space $E$ of dimension $n$, then (in the notation of \ref{.42}) each $\ka_{i,j}$ is a surjective map
$$\fk{g}^\m{ab} \rightarrow \m{Hom}(L_{i+1}, L_i) \rightarrow 0 \, ,$$
hence corresponds to a $k$-rational point of $\bb{P}\fk{g}^\m{ab}$. In this way, the vector $(\ka_{1,2},\dots,\ka_{n-1,n})$ may be regarded as a $k$-point of $(\bb{P}\fk{g}^\m{ab})^J$ where $J$ is the index set of pairs $(i,j)$, $1 \le i<j \le n$, such that $j-i =1$; we call it \textbf{the constellation of} $r$ and denote it by $\m{\it{const}}(r)$.
\end{dfn}

\begin{rmk} \label{.4434}
Continuing with the situation and the notation of \ref{.4433}, we note the following consequence of \ref{.4231}: for any $0 \le l \le m \le n$,
\[ \m{ \it{const} }(r^l_m) = (\ka^r_{l+1,l+2},\dots,\ka^r_{m-1,m})  \]
is equal to $\m{\it{const}}(r)$ truncated $l$ times on the left and $n-m$ times on the right.
\end{rmk}

We now discuss how to generalize the above constructions to include families of representations. For this purpose we fix an affine $k$-scheme $T=\spec A$. We write $\fk{g}_A$ for the Lie $A$-algebra $A \otimes \fk{g}$ and $\fk{g}_T$ for the associated Lie $\ssf_T$-algebra.

\begin{dfn} \label{.225}
A representation $r:\fk{g}_T \rightarrow \iEnd{\mathcal{E}}$ is \textbf{flag} if
\begin{itemize}
\item[(i)] $\mathcal{E}$ is a vector sheaf (\ref{.042}),
\item[(ii)] for each $i \ge 0$, $\m{Fil}^r_i\mathcal{E}$ is a vector subsheaf (loc. cit.), and
\item[(iii)] formation of $\m{Fil}^r$ is compatible with base-change.
\end{itemize}
\end{dfn}

\subsection*{Counterexamples (\ref{.34}-\ref{.36})}

We discuss briefly what can go wrong when definition \ref{.225} is weakened. Even when the canonical filtration is a full flag by vector subsheaves, the flag can hide a degeneration (\ref{.34}); and infinitesimally, the filtration can take on a horizontal flavor (\ref{.35}).

\begin{unmarked}\label{.34}
Let $T=\spec k[x]$, $E=k[x]^2$, $\fk{g}=k$, and define $r:k[x]\rightarrow \mathrm{Mat}_{2\times2}(k[x])$ by
$$1\longmapsto \left( \begin{array}{cc} 
0	& x	 \\ 
0  	& 0	 \\ 
\end{array} \right)$$
Then $E^0=\ker \left( \begin{array}{cc} 
0	& x	 \\ 
0  	& 0	 \\ 
\end{array} \right)$, so $\m{Fil}^r$ is the standard flag $0\subset k[x] \subset k[x]\oplus k[x]$. But $r_0$, the fiber of $r$ above the origin, is the trivial representation.
\end{unmarked}

\begin{unmarked} \label{.35}
Let $T=\spec k[\epsilon]/(\epsilon^2)$, $E=k[\epsilon]/(\epsilon^2)$, $\fk{g}=k$, and define $r:k[\epsilon]/(\epsilon^2) \rightarrow k[\epsilon]/(\epsilon^2)$ by
$$1\longmapsto \epsilon$$
Then $\m{Fil}^r$ is given by $0\subset k\epsilon \subset k\oplus k\epsilon = k[\epsilon]/(\epsilon^2)$.
\end{unmarked}

\begin{unmarked}\label{.36}
In \ref{.35}, $\m{Fil}_1$ is not locally free; in \ref{.34}, $\m{Fil}_1$ is locally free and co-locally-free, but its formation is not compatible with base change.
\end{unmarked}

We generalize definition \ref{.11633} in the obvious way:

\begin{dfn} \label{.37}
\setcounter{equation}{0}
Given a representation $r: \fk{g}_T \to \fk{gl}_{n,T}$ of $\fk{g}_T$ on $\ssf_T^n$, we denote the composite
$$\fk{g}_T \to \fk{gl}_{n,T} \to \ssf_T$$
of $r$ with the $(i,j)^\m{th}$ standard projection of $\fk{gl}_{n,T}$ by $\la^r_{i,j}$ and call it \textbf{the full} $(i,j)^\m{th}$ \textbf{matrix entry of} $r$.
\end{dfn}

\begin{prop}\label{.38}
In the notation of \ref{.37}, a representation of the form $r:\fk{g}_T\rightarrow \fk{n}_{n,T}$ is flag if and only if each $\la_{i,i+1}$, $i= 1, \dots, n-1$, is surjective.
\end{prop}

\begin{proof}
Suppose each $\la_{i,i+1}$, $i=1,\dots, n-1$, is surjective. Then with the obvious notational modifications, paragraphs \ref{.116341} -- \ref{.11635} apply to show that $\m{Fil}^r$ is the standard flag. Since, moreover, surjectivity is preserved by arbitrary change of base, it follows that $r$ is flag.

For the converse, suppose $\la_{i,i+1}$ not surjective. Then its image is an ideal $\cl{I}$, and for any $t \in Z(\cl{I})$, $\la_{i,i+1}(t)=0$. Then by \ref{.11636} the filtration associated to $r(t)$ is not a full flag, whence $r$ was not flag.
\end{proof}

\begin{rmk} \label{.381}
We note that every flag representation of $\fk{g}_T$ over $T$ is Zariski locally on $T$ isomorphic to a representation on $\ssf_T ^n$, with associated flag equal to the standard flag.
\end{rmk}

We generalize definition \ref{.42} by applying the usual typographical transformations:

\begin{dfn} \label{.3813}
Let $r:\fk{g}_T \to \iEnd \cl{E}$ be a flag representation of $\fk{g}$ on a vector sheaf $\cl{E}$ of rank $n$. We denote by $\cl{L}^r_i$ the $i^\m{th}$ line sheaf $\gr_i ^{\m{Fil}^r} \cl{E}$ associated to $r$. When there is no risk of confusion, we drop the $r$ from the notation.

For each $1 \le i < j \le n$, $r$ defines a map
$$\ka^r_{i,j} : (\fk{h}_{j-i})_T \to \cl{H}om_{\ssf_T} (\cl{L}_j, \cl{L}_i)$$
which we call \textbf{the canonical} $(i,j)^\m{th}$ \textbf{matrix entry of} $r$. When there is no risk of confusion, we drop the $r$ from the notation.
\end{dfn}

In view of definition \ref{.3813}, definition \ref{.4433} may be generalized as follows. 

\begin{dfn} \label{.495}
Let $T$ be a $k$-scheme and let $r: \fk{g}_T \to \iEnd \cl{E}$ be a flag representation of $\fk{g}$ on a vector sheaf $\cl{E}$ of rank $n$. Then the \textbf{constellation of} $r$, denoted $\m{\mathit{const}}(r)$, is the vector $(\ka^r_{1,2},\dots,\ka^r_{n-1,n})$ regarded as a $T$-valued point of $(\bb{P}\fk{g}^\m{ab})^J$ where $J$ is the index set of pairs $(i,j)$, $1 \le i<j \le n$ such that $j-i =1$.
\end{dfn}

\section{Automorphisms of flag representations}
\label{Aut}

We continue to work with a finitely generated Lie algebra $\g$ over a field $k$.

\begin{dfn} \label{.383}
If $\cl{E}$ is a vector sheaf on $T$ and $\m{Fil}$ is a filtration by vector subsheaves, we call \textbf{the group of unipotent automorphism of} $(\cl{E},\m{Fil})$ and denote by $\bb{U}_{\m{Fil}}(\cl{E})$ the $T$-group of automorphisms of $\cl{E}$ which respect $\m{Fil}$ and induce the identity on the associated graded. The group of unipotent automorphisms of $(\cl{E},\m{Fil})$ is a closed subgroup of $\fAut {\cl{E}}$. Its functor of points is defined as follows. Since $\m{Fil}$ is locally split, given $f:T' = \spec A' \rightarrow T$, $\m{Fil}$ pulls back to a filtration $f^*\m{Fil}$ on $f^*\cl{E}$; $\bb{U}_\m{Fil}(\cl{E})(T')$ is the set of automorphisms of $f^*\cl{E}$ which respect $f^*\m{Fil}$ and induce the identity on the associated graded.
\end{dfn}

\Definition{.38303}{ Let $\rgttoiende$ be a flag representation of $\g$ on a vector sheaf $\cl{E}$ over $T$ (\ref{.225}). We define \textbf{the group of unipotent automorphisms of} $r$ and denote by $U(r)$ the subgroup of $\fAut r$ consisting of those automorphisms which are unipotent with respect to $\m{Fil}^r$. }

\begin{prop} \label{.3831}
Let $r:\fk{g}_T \rightarrow \iEnd(\cl{E})$ be a flag representation of $\fk{g}_T$ on a vector sheaf $\cl{E}$ over $T$. Then 
\[\fAut r = {\bb{G}_m}_T \times U(r)\]
\end{prop}

\begin{proof}
Let $n$ denote the rank of $\cl{E}$. We consider a point $\phi'$ of $\fAut r$ with values in an arbitrary $T$-scheme $T'$, adding primes to denote pullback to $T'$. Then for $i = 1, \dots, n-1$, $\ka'_{i,i+1}$ (\ref{.3813}) is surjective. Thus, for any open $U' \subset T'$ and any morphism
$$\cl L'_{i+1}|_{U'} \rightarrow \cl L'_i|_{U'} \, ,$$
$\gr \phi'$ induces a commuting square 
\[\xymatrix		@ C= 2 cm	{
\cl L'_i |_{U'} \ar[r]^-{(\gr \phi')_i |_{U'}}	&
\cl L'_i |_{U'}	\\
\cl L'_{i+1} |_{U'} \ar[r]_-{(\gr \phi')_{i+1} |_{U'}} \ar[u]	&
\cl L'_{i+1} |_{U'} \ar[u]	}
\]
Applying this in particular to a family of local isomorphisms, it follows that $(\gr \phi')_i = (\gr \phi')_{i+1}$. This gives us a short exact sequence
\[ 1 \rightarrow U(r) \rightarrow \bb{U}_\m{Fil}(\cl{E}) \rightarrow {\bb{G}_m}_T \rightarrow 1 \]
Finally, there is a natural splitting making ${\bb{G}_m}_T$ central, whence the product decomposition.
\end{proof}

\begin{dfn} \label{.3832}
We let $n(r)$ denote the Lie $o_T$-algebra (\ref{.045}) defined by
\[ n(r)(T') = \{ \phi \in \fk{n}_{\m{Fil}_{T'}}(E_{T'}) \; | \; r_{T'}(v) \circ \phi = \phi \circ r_{T'}(v) \; \m{for \, all \,} v \in \fk{g}_{T'} \}\]
that is, the set of endomorphisms of $E_{T'}$ nilpotent with respect to $\m{Fil}_{T'}$ and equivariant with the action of $r_{T'}$. We call $n(r)$ \textbf{the Lie algebra of nilpotent infinitesimal automorphisms of} $r$.
\end{dfn}

\begin{claim} \label{.3833}
The functor $n(r)$ is the scheme-theoretic Lie algebra of $U(r)$. (In the notation of \cite{dg}, $n(r) = \fk{Lie} \, U(r)$.)
\end{claim}

\begin{proof}
Fix $T' = \spec B'$, write $T'[\epsilon] = \spec B'[\epsilon]/[\epsilon^2]$ and write $E'$ for $\Gamma(T', \cl{E}_{T'})$. We are to show that $n(r)(T')$ is the kernel of
$$\alpha: U(r)(T'[\epsilon]) \rightarrow U(r)(T') \, .$$
Consider the (split) short exact sequence of abstract groups
\[\xymatrix @1 {
1 \ar[r]	&
\fk{n}_{\m{Fil}'}(E') \ar[r]^-{e^{\epsilon \cdot}}	&
\bb{U}_\m{Fil} (\cl{E}) (T'[\epsilon]) \ar[r]^-\beta \POS c+(12.5,-1.4)="t"	&
\bb{U}_\m{Fil} (\cl{E}) (T') \ar[r] \ar[r] \POS c+(-10.5,-1.7)="s"	&
1
\ar @/^3pt/ "s";"t"	}
\]
where for $\phi \in \fk{n}_{\m{Fil}'}(E')$, $e^{\epsilon \phi} := 1 + \epsilon \phi$. Since $\alpha$ is just the restriction of $\beta$ to the set of automorphisms equivariant with the action, it suffices to check that $\phi \in \fk{n}_{\m{Fil}'}(E')$ is equivariant with the action of $\fk{g}_{T'}$ if and only if $e^{\epsilon \phi}$ is equivariant with the action of $\fk{g}_{T'[\epsilon]}$. This is formal: suppose $\phi$ is equivariant with the action of $\fk{g}_{T'}$, fix an arbitrary $v \in \fk{g}_{T'[\epsilon]}$, write $v = v_0 + \epsilon v_1$ with $v_0, v_1 \in \fk{g}_{T'}$ and compute
\begin{align*}
(1+ \epsilon \phi)(v_0 + \epsilon v_1) 	&= v_0 + \epsilon(v_1 + \phi v_0) \\ 
									&= v_0 + \epsilon (v_1 + v_0 \phi) \\
									&= (v_0 + \epsilon v_1)(1 + \epsilon \phi);
\end{align*}
conversely, suppose $e^{\epsilon \phi}$ is equivariant with the action of $\fk{g}_{T'[\epsilon]}$, fix an element $v \in \fk{g}_{T'}$, regard it as an element of $\fk{g}_{T'[\epsilon]}$ and note that
\begin{align*}
v + \epsilon v \phi 	&= v (1+\epsilon \phi) \\
					&= (1+ \epsilon \phi)v \\
					&= v + \epsilon \phi v
\end{align*}
implies $v \phi = \phi v$.
\end{proof}

The Lie algebra of nilpotent infinitesimal automorphisms of $r$ is useful because it is, on the one hand, isomorphic to $U(r)$ (as a scheme) and, on the other hand, occurs as the kernel of a map of vector groups, hence as the contravariant total space of a module, as I now explain.

\begin{claim} \label{.3834}
The exponential power series induces an isomorphism of functors $n(r) \rightarrow U(r)$.
\end{claim}

\begin{proof}
Given a $T$-scheme $T' = \spec B'$ as above, the exponential power series defines a map
\[\xymatrix @ R=1pt {
	&	\fk{n}_{\m{Fil}'}(E') \ar[r]^\exp 		& \bb{U}_\m{Fil}(\cl{E})(T') \\
\text{given by} \\
	& 	v \ar@{|->}[r] 					& 1+ v + \frac{v^2}{2} + \frac{v^3}{3!} + \cdots 
}\]
while the logarithmic power series defines a map
\[\xymatrix @ R=1pt {
	& 	\fk{n}_{\m{Fil}'}(E') 	& \bb{U}_\m{Fil}(\cl{E})(T') \ar[l] \\
\text{given by} \\
	& 	(u-1) - \frac{(u-1)^2}{2} + \frac{(u-1)^3}{3} - \cdots 	& u \ar@{|->}[l] \, .
}\]
These are inverse to one another. So it suffices to check that for $v \in \fk{n}_{\m{Fil}'}(E')$, $v$ is equivariant with the action of $\fk{g}_{T'}$ if and only if $\exp v$ is: if $v \in \fk{n}_{\m{Fil}'}(E')$ is equivariant with the action of $\g_{T'}$ and $w \in \g_{T'}$ is an arbitrary element then $r_{T'}w$ commutes with the terms of the exponential power series in $v$ and hence with their sum, $\exp v$; conversely, if $\exp v \in \bb{U}_\m{Fil}(\cl{E})(T')$ is equivariant with the action of $\g_{T'}$ then $r_{T'}w$ commutes with the terms of the logarithmic power series in $\exp v$ hence with their sum which equals $v$.
\end{proof}

\begin{unmarked} \label{.3835}
Fix a basis $v_1, \dots, v_m$ for $\fk{g}$, and continuing with our flag representation $r: \fk{g}_T \to \iEnd(\cl{E})$, define
$$\Psi: \fk{n}_\m{Fil}\cl{E} \to (\fk{n}_\m{Fil}\cl{E})^{\oplus m}$$
by
$$\phi \mapsto ([\phi, rv_1], \dots, [\phi, rv_m]) \, .$$
Then $n(r) = \ker \V \Psi^\lor$ (\ref{.02}).
\end{unmarked}

\begin{prop}\label{.3}
Suppose $r:\fk{g}\rightarrow\End E$ is a flag representation of $\fk{g}$ on a vector space $E$ of dimension $n$. Then
$$2\le \dim \fAut{r} \le n \, .$$ 
\end{prop}

Only the lower bound is needed in the sequel.

\begin{proof}
Fixing a filtered isomorphism $E \cong k^n$, we may replace $r$ with a homomorphism of Lie algebras $r:\fk{g} \to \fk{n}_n$. In the notation of \ref{.11633}, 
$$U(r) = \left\{ b\in \bb{U}_n \Bigg\bracevert \sum_{i<k<j} b_{ik} \la_{kj} - b_{kj}\la_{ik} = 0, \; j-i\ge2\right\}$$
For each $i$, fix $v_i\in \fk{g}$ such that $\la_{i,i+1}v_i = 1$. Then by applying the above equation to $v_i$ we solve for $b_{i+1,j}$ in terms of entries of the form $b_{i'+1, j'}$ with either $j'-i' <j-i$ or both of: $j'-i' = j-i$ and $i'<i$. Iterating (and renaming), we can solve for each $b_{i, i+s}$ in terms of entries of the form $b_{1, 1+s'}$ with $s'\le s$. This provides the upper bound. 

For the lower bound, note that $\bb{U}_n$ contains a copy of $\bb{G}_a$ whose entries commute with all strictly upper triangular matrices (observe simply that $b_{1n}$ does not intervene in the above equations). Thus
$$\fAut{r} \supset \bb{G}_m \times \bb{G}_a \, .$$
\end{proof}


\Remark{.3003}{Finally, we note that if $T$ is connected and $r: \g_T \to \iEnd \cl{E}$ is any representation on a vector bundle, then $\fAut r$ is connected. Indeed, $\fEnd \, r$ is the total space of a module, hence connected with irreducible fibers; and $\fAut r$ is an open subscheme containing a global section. }

\section{Nondegeneracy}
\label{Nondeg}

We continue to work with a finitely generated Lie algebra $\g$ over a field $k$.

\begin{dfn} \label{.12}
A nilpotent representation $r:\fk{g} \rightarrow \End E$ is \textbf{nondegenerate} if it is flag and satisfies the following condition. If $\dim E =1 $ then $r=0$ is the trivial representation. If, on the other hand, the rank of $E$ is $n\ge 2$ then 
\begin{itemize}
\item[(i)] $E_{n-1}$ and $E/E_1$ are both nondegenerate and 
\item[(ii)] dim $\fAut r$ is minimal among representations defined over fields containing $k$ and satisfying (i).
\end{itemize} 
\end{dfn}

\begin{notation} \label{.1222}
 We denote the dimension of the automorphism group of some (hence any) nondegenerate $n$-dimensional representation of $\fk{g}$ by $A(\fk{g}, n)$.
\end{notation}

\subsection*{Examples and counterexamples in low dimensions (\ref{.15} -- \ref{.269})}

We consider representations of our fixed Lie algebra $\fk{g}$ on $k^n$ for small $n$, nilpotent with respect to the standard flag. We begin with the case $n=2$.

\begin{unmarked}\label{.15}
Since $\fk{n}_2=k$ is abelian, every representation $r:\fk{g} \to \fk{n}_2$ factors through the abelianization $\fk{g}^\m{ab}$ of $\fk{g}$. Conversely, any linear map $\fk{g}\to \fk{n}_2$ which factors through $\fk{g}^\m{ab}$ is a homomorphism of Lie algebras and hence a representation. Thus representations on $k^2$, nilpotent with respect to the standard flag, correspond canonically to linear functionals $\lambda:\fk{g}^\m{ab}\rightarrow k$ on the abelianization of $\fk{g}$. 

The trivial representation is clearly not flag; conversely, every nontrivial representation is flag (\ref{.11634}). A calculation shows that for $r$ nonzero,
$$\fAut{r}= 
\left\{ \left( \begin{array}{cc} 
b & c \\ 
0 & b
\end{array} \right) \right\}
=\bb{G}_m\times\bb{G}_a \, .$$
In particular, $\dim \fAut r = 2$. Thus every flag representation is nondegenerate and $A(\fk{g},2)=2$ independently of $\fk{g}$.
\end{unmarked}

We now consider three dimensional representations (\ref{.16} -- \ref{.18}).

\begin{unmarked}\label{.16}
\setcounter{equation}{0}
Fix a three dimensional nilpotent representation $r:\fk{g}\rightarrow\fk{n}_3$. We use single indices in order to simplify notation: for $v\in\fk{g}$ write 
$$rv = 
\left( \begin{array}{ccc} 
0 & 	\lambda_1v 	& \lambda_3v \\ 
0 & 	0 			& \lambda_2v \\ 
0 &	0			& 0
\end{array} \right) $$
For each $i$, $\la_i$ is the composite of $r$ regarded as a map $\fk{g} \to \fk{n}_3$ with one of the three standard projections $\fk{n}_3 \to k$. Thus for each $i$, $\la_i$ is a linear functional on $\fk{g}$. The equation $r[ \cdot, \cdot] =[r\cdot,r\cdot]$ may be rewritten in terms of these linear functionals:
\begin{eqnarray}
\la_1[\cdot,\cdot] &=& 0 					\label{.163}	\\
\la_2[\cdot,\cdot] &=& 0 					\label{.164}	\\
\la_3[\cdot,\cdot] &=& \la_1 \wedge \la_2	\,.	\label{.165}
\end{eqnarray} 
According to \ref{.11634}, $r$ is flag if and only if $\lambda_1$, $\lambda_2$ are nonzero. For $r$ flag, the automorphism group is given by 
\begin{align}
\fAut{r} = 
\left\{ \left( \begin{array}{ccc} 
a & 	b_1 	& c \\ 
0 & 	a 	& b_2 \\ 
0 &	0	& a	\label{.166}
\end{array} 
\right)
\Bigg{\bracevert} \; b_1\lambda_2=b_2\lambda_1 \; \mathrm{and} \; a\neq0 
\right\}
\end{align}
Thus
\begin{align}
A(\fk{g},3)= \left\{ 
\begin{array}{ll}
3 & \mbox{if equations \ref{.163} -- \ref{.165} imply } \la_1 \wedge \la_2 = 0 \\
2 & \m{otherwise} .	\label{.167}
\end{array}
\right.
\end{align}

\end{unmarked}

\begin{unmarked} \label{.17}
Continuing with the notation of \ref{.16}, consider, for a first example in three dimensions, the case $\fk{g} = k^2$. Then the system of equations \ref{.163} -- \ref{.165} becomes simply
\setcounter{equation}{0}
\begin{equation}
\la_1 \wedge \la_2 = 0 \, .
\end{equation}
Hence every flag representation is nondegenerate and we have
$$\mathrm{A}(k^2, 3) = 3 \, .$$
\end{unmarked}

\begin{unmarked} \label{.18}
Again in the notation of \ref{.16}, consider, for a second example in three dimensions, the case $\fk{g} = \fk{n}_3$. Then flag representations with $\la_1, \la_2$ linearly independent exist: the natural representation provides an example. So nondegenerate representations are precisely those for which $\la_1, \la_2$ are linearly independent and $\mathrm{A}(\fk{n}_3, 3)=2$.
\end{unmarked}

We now consider four dimensional representations (\ref{.2} -- \ref{.217}).

\begin{unmarked} \label{.2}
Given a representation $r:\fk{g}\rightarrow \fk{n}_4$, write
$$rv = 
\left( \begin{array}{cccc} 
0 & 	\lambda_1v 	& \lambda_4v	&\lambda_6v \\ 
0 & 	0 			& \lambda_2v	&\lambda_5v \\ 
0 &	0			& 0			&\lambda_3v\\
0 & 	0			& 0			& 0	
\end{array} \right) $$
with $\lambda_i \in \fk{g}^\lor$. In terms of these linear functionals, the equation $r[\cdot,\cdot]=[r\cdot,r\cdot]$ becomes:
\setcounter{equation}{0}
\begin{equation}
\la_1[\cdot,\cdot] = \la_2[\cdot, \cdot] = \la_3[\cdot, \cdot] = 0 \label{.201} \\
\end{equation}
and
\begin{eqnarray} 
\la_4[\cdot, \cdot] &=& \la_1 \wedge \la_2 				\label{.202}	\\
\la_5[\cdot, \cdot] &=& \la_2 \wedge \la_3 				\label{.203}	\\
\la_6[\cdot,\cdot] &=& \la_1 \wedge \la_5 - \la_3 \wedge \la_4 \, . \label{.204}
\end{eqnarray}
The representation $r$ is flag if and only if $\la_1, \la_2, \la_3 \neq 0$ (\ref{.11634}). For $r$ flag, $\fAut{r}$ is the set of matrices of the form 
$$\left( \begin{array}{cccc} 
a & b_1	 	& b_4	& b_6 \\ 
0 & 	a 			& b_2	& b_5 \\ 
0 &	0			& a			& b_3\\
0 & 	0			& 0			& a	
\end{array} \right) $$
subject to the equations
\begin{eqnarray}
b_1\la_2 &=& b_2\la_1 						\label{.205}	\\
b_2\la_3 &=& b_3\la_2 						\label{.206}	\\
b_1\la_5 + b_4\la_3 &=& b_5\la_1 + b_3\la_4	\, .	\label{.207}
\end{eqnarray}

The possible pairs of numbers $(A(\fk{g}, 3), A(\fk{g}, 4))$ are $(2,2),(2,3)$ and $(3,4)$. Paragraphs \ref{.2153} -- \ref{.216} explain why this is and paragraph \ref{.217} gives examples of each of these cases.
\end{unmarked}

\begin{unmarked} \label{.2153}
\setcounter{equation}{0}
Suppose $A(\fk{g},3) = 3$, fix a representation $r:\fk{g}\to\fk{n}_4$ and suppose $r$ is nondegenerate. Then according to \ref{.167}, in the notation of \ref{.2}, we have 
\begin{align}
\la_1, \la_2, \la_3 \neq 0 					\label{.21531} 	\\
\la_4 [\cdot, \cdot] = \la_1 \wedge \la_2 =0 	\label{.21532}	\\
\la_5 [\cdot, \cdot] = \la_2 \wedge \la_3 =0	\, .	\label{.21533}
\end{align}
Fix elements $a,a' \in k$ such that
\begin{align}
a\la_1=\la_2 							\label{.21534}	\\
a'\la_1=\la_3	 \, .						\label{.21535} 
\end{align}
Combining \ref{.201} and \ref{.204} with \ref{.21532} -- \ref{.21535} we get the following three equations:
\begin{align}
\la_1[\cdot,\cdot] = &0 \\
(\la_5 -a'\la_4)[\cdot , \cdot ] = &0 \\
 \la_1 \wedge (\la_5-a'\la_4) = & \la_6[\cdot,\cdot] \, .
\end{align}
So by \ref{.167} applied with $\la_5-a'\la_4$ in place of $\la_2$ and $\la_6$ in place of $\la_3$, we have
\begin{align}
\la_1 \wedge (\la_5-a'\la_4) = 0 \, .
\end{align}
There is thus an element $a''\in k$ such that
\begin{align}
\la_5 = a''\la_1 + a'\la_4 \, .
\end{align}
In terms of $a,a',a''$ the system of equations \ref{.205} -- \ref{.207} is equivalent to
\begin{align}
ab_1=& b_2 \\
a'b_1 =& b_3 \\
a''b_1 + a'b_4 = & b_5 \, .
\end{align}
This shows that $\dim \fAut r = 4$ and hence that $A(\fk{g},4)=4$.
\end{unmarked}

\begin{unmarked} \label{.216}
\setcounter{equation}{0}
Now suppose instead that $A(\fk{g}, 3) \neq 3$. Fix a nilpotent representation $r:\fk{g} \to \fk{n}_4$, denote its full coordinates by $\la_1,\dots,\la_6$ as in \ref{.2} and suppose $r$ is nondegenerate. Then according to \ref{.167}, the pairs $(\la_1,\la_2),(\la_2,\la_3)$ are both linearly independent, so the system of equations \ref{.205} -- \ref{.207} becomes
\begin{align}
b_1 = b_2 = b_3 = 0 \label{.2399} \\
b_4\la_3 = b_5\la_1 \, . \label{.24}
\end{align}
If $\la_1, \la_3$ are linearly dependent then (\ref{.24}) becomes
\begin{equation}
b_5= a'b_4
\end{equation}
and $A(\fk{g}, 4)=3$.
If, on the other hand, $\la_1, \la_3$ are linearly independent, then the system of equations \ref{.2399} -- \ref{.24} becomes 
\begin{equation}
b_1=\cdots=b_5=0
\end{equation}
and $A(\fk{g},4)=2$.
\end{unmarked}

\setcounter{equation}{0}
\begin{unmarked} \label{.217}
Here, then, are the promised examples. The $4\times 4$ upper triangular Lie algebra $\fk{n}_4$ is of type $(2,2)$, as witnessed by the natural representation and its subquotients. The $m$-dimensional abelian Lie algebra $k^m$ is of type $(3, 4)$: indeed, the system \ref{.201} -- \ref{.204} becomes precisely
\begin{eqnarray}
\la_1 \wedge \la_2=0\\
\la_2 \wedge \la_3=0
\end{eqnarray}
and, in terms of $a,a'$ such that $a\la_1=\la_2$ and $a'\la_1=\la_3$,
\begin{equation}
\la_1 \wedge (\la_5-a'\la_4)=0 \, .
\end{equation}
Finally, $\fk{n}_3$ is of type $(2,3)$; this is not hard to check directly but is better understood as a consequence of the fact that the width is bounded by the depth (\ref{.46}).
\end{unmarked}

\begin{unmarked} \label{.269}
Finally, we give an example of a representation which is flag, has minimal automorphism group, and a degenerate subquotient. Let $F$ be a vector space of dimension 2 and let $\fk{g}=\fk{n}(F)$ (\ref{.01}). Then
$$\m{Hom}_\m{\bld{Lie}}(\fk{g},\fk{n}_4)=\m{Hom}_\m{\bld{Vect}}(F,\fk{n}_4)=F^{\lor 6} \, .$$
Denoting the full coordinates of $r$ by $\la_1,\dots,\la_6$ as in \ref{.2}, suppose $\la_1,\la_2$ are linearly independent and $\la_2,\la_3$ are linearly dependent but nonzero. Then $\dim \fAut r^1 = 3$ is not minimal; but $\la_1, \la_3$ are linearly independent and equations \ref{.205} -- \ref{.207} show that $\dim\fAut{r}=2$ is minimal. This example shows that the recursive aspect of the definition of nondegeneracy is not redundant.
\end{unmarked}

We conclude this section with a definition of nondegeneracy in families.

\begin{dfn} \label{.3836}
A representation $r: \fk{g}_T \rightarrow \iEnd(\cl{E})$ is \textbf{nondegenerate nilpotent} if it is flag (we recall that this implies in particular that $\cl{E}$ is a vector sheaf) and if it satisfies the following conditions, recursive on the rank of $\cl{E}$: 
\begin{itemize}
\item[(i)] $r_{n-1}$, $r^1$ are both nondegenerate 
\item[(ii)] the fibers of $r$ are nondegenerate in the sense of \ref{.12}
\item[(iii)] $\fAut r$ is flat.
\end{itemize} 
\end{dfn}

\section{The nilpotent width of a Lie algebra}

\label{Width}

We continue to work with a finitely generated Lie algebra $\g$ over a field $k$.

\begin{dfn} \label{.26933}
Recall that the \textbf{nilpotent depth} of $\fk{g}$ is the largest number $d$ such that $\fk{g}^{(d)}/\fk{g}^{(d+1)} \neq 0$. We denote it by $d(\fk{g})$, or simply by $d$ when there is no risk of confusion. A nilpotent representation is \textbf{wide} if it is flag and if the points of its constellation (\ref{.4433}) are distinct. Define the \textbf{nilpotent width} of $\fk{g}$ to be the largest number $w$ such that there exists a field $k'$ containing $k$ and a wide nilpotent representation $r: \fk{g}_{k'} \rightarrow \m{End}(E)$ of $\fk{g}_{k'}$ on a vector space $E$ of dimension $w+1$ over $k'$. We denote it by $w(\fk{g})$ or by $w$ when there is no risk of confusion.
\end{dfn} 

\Remark{}{We show in \S9 that a representation of dimension not more than $w+1$ is nondegenerate if and only if it is wide (\ref{.5}).}

\Remark{}{It follows from Proposition \ref{.39} below that in defining the width we may equivalently consider only fields $k'$ which are finite over $k$.}

\begin{rmk} \label{.26934}
If $r:\fk{g} \to \fk{n}_n$ is a flag representation of $\fk{g}$ on $k^n$, nilpotent with respect to the standard flag, then $r$ is wide if and only if the full matrix entries $\la^r_{i,i+1}$ for $i=1,\dots,n-1$ are pairwise linearly independent.
\end{rmk}

\begin{prop} \label{.433}
Any subquotient of a wide nilpotent representation is again wide.
\end{prop}

\begin{proof}
Fix a nilpotent representation $r:\fk{g} \to \End E$ of $\fk{g}$ on a vector space $E$ of dimension $n$ and suppose $r$ is wide. Since, in particular, $r$ is flag, according to \ref{.431} any subquotient is of the form $r^l_m$ for some $l,m$, $0\le l \le m \le n$. By \ref{.4434}, the constellation of $r^l_m$ is a truncation of the constellation of $r$; hence in particular its points are distinct.
\end{proof}

\begin{thm} \label{.434}
If $r: \fk{g} \rightarrow \m{End} \; E$ is a wide nilpotent representation then every canonical matrix entry of $r$ (\ref{.42}) is surjective.
\end{thm}

\begin{proof}
Suppose for an induction on $s$ that $\ka_{i,j}$ is surjective for $j-i \le s$, and fix $i,j$ with $j-i = s$. Then there exists a $u \in \fk{h}_s$ such that $\ka_{i,j}(u) \neq 0$. On the other hand, linear independence of $\ka_{i, i+1}, \ka_{j, j+1}$ implies that there exists a $v \in \fk{h}_1$ such that $ \ka_{i, i+1}(v) = 0$ while $\ka_{j, j+1}(v) \neq 0$. Then by the conclusion of (and in the notation of) \ref{.44}, we have
\[
\ka_{i, j+1} \{ v, u \} = \ka_{i, i+1}(v) \circ \ka_{i+1, j+1}(u) - \ka_{i,j}(u) \circ \ka_{j, j+1}(v) \neq 0
\]
This shows that $\ka_{i,j+1}$ is nonzero, hence surjective, and completes the induction.
\end{proof}

\begin{cor} \label{.46}
If $\fk{g}$ is a Lie algebra of width $w$ and depth $d$ then $w \le d$.
\end{cor}

\begin{proof}
Fix a wide representation of dimension $w+1$ and consider its $(1, w+1)^\m{st}$ canonical matrix entry, $\ka_{1,w+1}$. This is a map from $\fk{h}_w$ (possibly tensored with a field extension) to a line; by the theorem, it is surjective. Hence its source, $\fk{h}_w$, is nonzero.
\end{proof}

\subsection*{Examples (\ref{.47} -- \ref{.49})}

\begin{unmarked} \label{.47}
By \ref{.46}, $w(\fk{n}_n) = n-1 = d(\fk{n}_n)$, as witnessed by the natural representation.
\end{unmarked}

\begin{unmarked} \label{.48}
Let $\fk{g}$ be the $4$-dimensional Lie algebra over $\bb{Q}$ on basis elements $v_1, \dots, v_4$ with $[v_1, v_2] = v_3$ and $[v_1, v_3] = v_4$ the only nonzero brackets among the generators. Then $w=d=3$, as witnessed by the representation $\fk{g} \rightarrow \fk{n}_{4, \bb{Q}(\sqrt{-1})}$ sending $v_1, v_2$ to the matrices with entries $(1,1,1+2\sqrt{-1}), (1,\sqrt{-1},-1)$ along the first superdiagonal and zero elsewhere.
\end{unmarked}

We prepare for a final example with a small lemma.

\begin{lm} \label{.429}
If $\fk{g}$ does not admit an $n$-dimensional wide representation then $w(\fk{g}) < n-1$.
\end{lm}

\begin{proof}
We prove the contrapositive: If $w(\fk{g}) \ge n-1$ then $\fk{g}$ admits a wide representation $r$ of dimension $\ge n$; by \ref{.433} $r_n$ is wide, so, in particular, $\fk{g}$ does admit a wide $n$-dimensional representation.
\end{proof}

\begin{unmarked} \label{.49}
Let $\phi$ be the map
$$k \rightarrow \wedge^2 k^4$$
defined by
$$1 \mapsto e_1 \wedge e_2 + e_3 \wedge e_4 \, .$$
The formula
$$[v+a,w+b]=\phi^\lor(v\wedge w)$$
for $v,w\in k^4, a,b\in k$ endows $k^4 \oplus k$ with the structure of a Lie algebra, nilpotent of depth 2. We claim that this Lie algebra has width 1. Suppose given linear functionals $\la_1, \la_2, \la_3 \in (k^4 \oplus k)^\lor$ defining a representation on $k^3$ nilpotent with respect to the standard flag as in \ref{.16}. Since $(k^4 \oplus k)^\lor = (k^4)^\lor \oplus k^\lor$ we may write
$$\la_i= \la'_i + \la''_i$$
with $\la'_i \in (k^4)^\lor$ and $\la''_i \in k^\lor$. Then, on the one hand, for $u,v \in k^4$ we have
\begin{align}
\la''_3 \phi^\lor(u\wedge v) 	& =\la_3[u,v] 	\nonumber	\\
	& =\la_1 u \la_2 v -\la_1 v \la_2 u		\nonumber 	\\
	& =\la'_1 u \la'_2 v -\la'_1 v \la'_2 u		\nonumber 	\\
	& = (\la'_1 \wedge \la'_2)(u \wedge v)	\nonumber
\end{align}
which implies that
\begin{align}
 \la'_1 \wedge \la'_2 \nonumber = \phi(\la''_3)
\end{align}
is a multiple of $e_1 \wedge e_2 + e_3 \wedge e_4$. Since the latter cannot be written as a wedge product, it follows that
\begin{align}
\la'_1 \wedge \la'_2 =0 \nonumber
\end{align}
On the other hand, $\la''_1 = \la''_2 =0$. Thus $\la_1 \wedge \la_2 = 0$. This shows that the Lie algebra we've constructed does not admit a wide representation of dimension $3$, hence by \ref{.429} has width 1 as claimed. 
\end{unmarked}

\section{Moduli of nondegenerate nilpotent representations}
\label{Modnd}

We continue to work with a finitely generated Lie algebra $\g$ over a field $k$.

\begin{unmarked} \label{.38372}
Fix an affine $k$-scheme $T= \spec A$, a vector sheaf $\cl{E}$ with module of global sections $E$, and a full flag by vector subsheaves $\m{Fil}$. We denote by $\fHom_{\m{\bld{Mod}}(o_T)}(\fk{g}_T, \fk{n}_\m{Fil}(\cl{E}) )$ (or by $\fHom_{\m{\bld{Vect}}(k)}(\fk{g}, \fk{n}_\m{Fil}( \cl{E} ) )$ when $T= \spec k$) the functor $\mathbf{Aff}(A) \to \mathbf{Set}$ sending
\[T' \mapsto \m{Hom}_{\mathbf{Mod}(\ssf_{T'})}(f^*\fk{g}, f^*\fk{n}_\m{Fil}(\cl{E}))\]
and we denote by 
\[\fHom_{\m{\bld{Lie}}(o_T)}(\fk{g}_T, \fk{n}_\m{Fil}(\cl{E}) ) = X_{(A, E, \m{Fil})} \supset X_{(A, E, \m{Fil})}^\m{fl} \supset X_{(A, E, \m{Fil})}^\m{nd}\] 
the successively smaller subfunctors whose points are representations, flag representations (\ref{.225}) and nondegenerate nilpotent representations (\ref{.3836}), respectively. 

When $A = k, E = k^n$ and $\m{Fil}$ is the standard flag, we write simply $X_n \supset X_n^\m{fl} \supset X_n ^\m{nd}$.
\end{unmarked}

We begin by studying the functors $X_{(A,E,\m{Fil})}, X_{(A,E,\m{Fil})}^\m{fl}$ and $X_{(A,E,\m{Fil})}^\m{nd}$ (\ref{.383721} -- \ref{.3962}).

\begin{prop} \label{.383721}
In the situation and the notation of \ref{.38372}
\begin{itemize}
\item[(1)] the inclusion $X_{(A, E, \m{Fil})} \hookrightarrow \fHom_{\m{\bld{Mod}}(o_T)}(\fk{g}_T, \fk{n}_\m{Fil}(\cl{E}) )$ is a closed immersion;
\item[(2)] the inclusion $X_{(A,E,\m{Fil})}^\m{fl} \hookrightarrow X_{( A, E, \m{Fil} )}$ is an open immersion.
\end{itemize}
\end{prop}

\begin{proof}
(1) Preservation of bracket is a closed condition defined by equations depending on the choice of a basis for $\fk{g}$. (2) Let
$$r_{ (A, E, \m{Fil} ) }:\fk{g}_{X_{( A, E, \m{Fil} )} }\rightarrow \fk{n}_\m{Fil}(\cl{E}_{X_{ (A, E, \m{Fil} ) } })$$
be the universal family and suppose $\cl{E}$ has rank $n$. Then $X^\m{fl}_{ ( A, E, \m{Fil} ) }$ is the open locus defined by the nonvanishing of each $\ka_{i,i+1}^{ r_{( A,E, \m{Fil} )} }$, $i=1,\dots, n-1$.
\end{proof}

\begin{cor} \label{.383722}
The functor $X^\m{fl}_{(A,E, \m{Fil})}$ is representable by a quasi-projective scheme; in particular, it is locally of finite presentation.
\end{cor}

\begin{prop} \label{.38373}
Let $A$ be a Noetherian $k$-algebra, $(E, \m{Fil})$ a vector sheaf of rank $n$ filtered by a full flag of vector subsheafs, let $B$ be an $A$-algebra and let $\s{B}$ denote the directed system of finite type $A$-subalgebras of $B$. Then the map
\[ \lim_{\substack{ \longrightarrow \\ B' \in \s{B} }} X^\m{nd}_{(A,E, \m{Fil} )} (B') \to X^\m{nd}_{(A,E,\m{Fil})} (B)\]
is an isomorphism.
\end{prop}

The proof follows (\ref{.38374} -- \ref{.38378}).

\begin{unmarked} \label{.38374}
For injectivity, given representations $r': \fk{g}_{B'} \to \fk{n}_{\m{Fil}}(\cl{E}_{B'})$, $r'': \fk{g}_{B''} \to \fk{n}_\m{Fil}(\cl{E}_{B''})$ such that
$$\m{id}_B \otimes_{B'} r' = \m{id}_B \otimes_{B''} r'' \, ,$$
let $B'''$ be the subalgebra generated by $B'$ and $B''$; then
$$\m{id}_{B'''} \otimes_{B'} r' = \m{id}_{B'''} \otimes_{B''} r'' \, .$$ 

We turn to surjectivity. Let $r : \fk{g}_B \to B \otimes \fk{n}_\m{Fil} \cl{E} = \fk{n}_{\m{Fil}_B}\cl{E}_B$ be a nondegenerate nilpotent representation. Then by \ref{.383722} there exists a $B' \in \s{B}$ and a flag representation $r': \fk{g}_{B'} \to \fk{n}_{\m{Fil}_{B'}}\cl{E}_{B'}$ such that $r = \m{id}_B \otimes_{B'} r'$. Assume for an induction on $n$ that after possibly replacing $B'$ by a finite type subalgebra of $B$ containing $B'$, $r'_{n-1}, r'^1$ are nondegenerate. Fix a basis $v_1, \dots, v_m$ for $\fk{g}$ and define 
$$\Psi: \fk{n}_{\m{Fil}_B}\cl{E}_B \to (\fk{n}_{\m{Fil}_B} \cl{E}_B)^{\oplus m}$$
by
$$\phi \mapsto ([\phi, rv_1], \dots, [\phi, rv_m])$$
and define $\Psi'$ similarly for $r'$ so that writing
\[ \xymatrix{
0										&
Q  \ar[l] \POS c+(3,1.6)="s"					&
( \fk{n}_{\m{Fil}_B} \cl{E}_B )^\lor \ar[l]^-\chi	 \POS p+(-10,1.5) = "t"					&
( \fk{n}_{\m{Fil}_B} \cl{E}_B )^{\oplus m \lor} \ar[l]_-{\Psi^\lor}			\\
0	&
Q' \ar[u]^-\alpha \ar[l]	&
( \fk{n}_{\m{Fil}_{B'}} \cl{E}_{B'} ) ^\lor \ar[l]^-{\chi'} \ar[u]	&
( \fk{n}_{\m{Fil}_{B'}} \cl{E}_{B'} ) ^{\oplus m \lor} \ar[l]^-{\Psi'^\lor} \ar[u]
\ar @/^3pt/ "s";"t" ^-\sigma	}
\]  
we have $n(r) = \V Q$, $n(r') = \V Q'$ (\ref{.3835}). The $B$-module $Q$ is flat and of finite presentation, hence projective, so that $\chi$ splits; fix a splitting $\sigma$ as in the diagram. Our goal is to show that after possibly replacing $B'$ by a finite type $B'$-subalgebra of $B$, $\chi'$ splits.
\end{unmarked}

\setcounter{equation}{0}
\begin{lm} \label{.38375}
Disengaging briefly from the notation of the proposition, let $B'$ be a Noetherian ring, let $B$ be a $B'$-algebra, let $N'$ be a finite $B'$-module, and consider an element $n' \in N'$. If $1_B \otimes_{B'} n' = 0$, then there exists a finite type subalgebra $B'' \subset B$ such that $1_{B''} \otimes_{B'} n' =0$.
\end{lm}

\begin{proof}
Fix a finite family $\{ n'_i \}$ of generators for $N'$, and write
\[ 1_B \otimes_{B'} n' = \sum_i b_i \otimes_{B'} n'_i \]
Then by \cite[6.4]{eisenbud}, there are elements $a_{i,j} \in B'$ and $c_j \in B$ such that
\begin{eqnarray}
\sum_j a_{ij} c_j &=& b_i \hspace{11ex} \m{for\,all} \; i \\
\m{and} \hspace{12ex} \sum_i a_{ij} n'_i &=& 0 \hspace{12ex} \m{for \, all} \; j \, .
\end{eqnarray}
Let $B''$ be the subalgebra generated over $B'$ by the (finitely many) $c_j$. Then again by \cite[6.4]{eisenbud}, $1_{B''} \otimes_{B'} n' =0$ as claimed.
\end{proof}

\begin{lm} \label{.38376}
In the notation of \ref{.38375}, let $Q'$ be a finitely presented $B'$-module, and let $\sigma$ be a morphism
$$Q := B \otimes_{B'} Q' \to N := B \otimes_{B'} N' \, .$$
Then after possibly replacing $B'$ by a finite type subalgebra of $B$ there exists a morphism
$$\sigma': Q' \to N'$$
such that $\sigma = \m{id}_B \otimes_{B'} \sigma'$.
\end{lm}

\begin{proof}
Fix a finite presentation
$$F'_1 \to F'_0 \to Q' \to 0 \, ,$$
and drop the primes to denote base-change to $B$:
\[\xymatrix@!0{
 0 & & 0\\ \\
 Q' \ar[uu] \ar[rr] & & Q \ar[uu] \ar[dr]\\
 & N'\ar[rr]|\hole & & N\\
 F_0' \ar[uu] \ar[rr] \ar@{.>}[ur] & & F_0 \ar[uu]\\ \\
 F_1'\ar[uu]\ar[rr] & & F_1\ar[uu]\\ \\
 B'\ar[rr] & & B
}\]
Since $F'_0$ is free and finite, after possibly replacing $B'$ by a finite type subalgebra of $B$, there is a map $\beta':F'_0 \to N'$ commuting with $\sigma \epsilon$ as in the diagram. Now
$$(F'_1 \to F'_0 \to N' \to N) =0 \, ,$$
so by \ref{.38375}, after possibly replacing $B'$ by a finite type subalgebra of $B$,
$$(F'_1, \to F'_0 \to N') = 0 \, .$$
Subsequently, $\beta'$ factors through $Q'$ to produce the desired morphism. 
\end{proof}

\begin{unmarked} \label{.38377}
Returning to the situation of the proposition, after possibly replacing $B'$ by a finite type subalgebra of $B$ containing $B'$, we obtain a candidate $\sigma': Q' \to (\fk{n}_{\m{Fil}_{B'}} \cl{E}_{B'})^\lor$ for our desired splitting. Now since 
\[ ( Q' \to (\fk{n}_{\m{Fil}_{B'}} \cl{E}_{B'})^\lor \to Q' \to Q ) - ( q' \mapsto 1_B \otimes q') =0 \]
and since $Q'$ is of finite type, by \ref{.38375}, after possibly replacing $B'$ by a finite type subalgebra of $B$ containing $B'$, we have $\chi' \sigma' = \m{id}_{Q'}$, giving us our desired splitting.
\end{unmarked}

\begin{unmarked} \label{.38378}
Finally, we have that $\fAut r'$ is flat, and also that $\dim (\fAut r'(t)) = \dim Q'(t) +1$ is locally constant on $T'= \spec B'$. Since the image of $T' \leftarrow T$ is dense, it follows that $\dim (\fAut r'(t)) = A(\fk{g}, n)$ for all $t \in T'$ which completes the proof.
\end{unmarked}

\begin{prop} \label{.39}
The inclusion $X_n^\m{nd} \hookrightarrow X_n^\m{fl}$ is an immersion. In particular, $X^\m{nd}_n$ is representable by a quasi-projective scheme.
\end{prop}

The proof follows in paragraphs \ref{.3904} -- \ref{.396}. We begin by recalling the theory of Fitting ideals.

\Proposition{.3904}{ Let $X$ be a Noetherian scheme and $\cl{F}$ a coherent sheaf. Then there exists a stratification 
$$s:\coprod_i X_i \to X$$
of $X$ by immersed subschemes such that given a morphism
$$g: T \to X \,,$$
\begin{itemize}
\item[(i)] if $T= \spec l$ with $l$ a field, then $g$ factors through $X_i$ if and only if 
$$\dim g^*\cl{F} = i \,,$$ and
\item[(ii)] if $T$ is Noetherian, then $g$ factors through $s$ if and only if $s^*\cl{F}$ is flat.
\end{itemize}
 }
 
 \begin{proof}
Since formation of the hypothetical stratification is compatible with base change, we may assume $X= \spec A$ to be affine. Write $F := \Gamma(X,\cl{F})$ for the associated $A$-module. Let $\m{Fitt}_i(F)$ denote the $i^\m{th}$ Fitting ideal of $F$ (\cite[20.4]{eisenbud}) and let 
$$X_i := Z(\m{Fitt}_i(F))$$
denote its zero locus. Compatibility with base change is immediate from the definition. Thus, fixing $g$ as in the theorem, we may assume $T = \spec B$ is affine and connected. Let $s_i$ denote the immersion
$$X_i \hookrightarrow X \,.$$
Suppose that $g$ factors through $s$. Then $g$ factors through $s_i$ for some $i$. Then (again using compatibility with base change),
\begin{align}
\m{Fitt}_i \, B \otimes F = 0
&&\text{and}&&
\m{Fitt}_{i+1} \, B \otimes F = (1)
\label{.3905}
\end{align}
so by \cite[20.8]{eisenbud}, $B\otimes F$ is projective, hence flat. Conversely, suppose $B\otimes F$ is flat. Since $B \otimes F$ is finitely presented, $B \otimes F$ is also projective; and since $\spec B$ is connected, $B \otimes F$ has constant rank, say $i$. Then by loc. cit., \ref{.3905} holds, from which it follows that $g$ factors through $s_i$, hence through $s$. This shows that the stratification we've defined satisfies (ii). Finally, (i) is immediate from the definition.
\end{proof}

\begin{unmarked} \label{.392}
Assume for an induction on $n$ that for $i \le n$, $X^\m{nd}_i \hookrightarrow X^\m{fl}_i$ is an immersion, and moreover that the two composites in
\[\xymatrix{
X^\m{nd}_i \ar[r] \ar@<-.7ex>@{.>}[d] \ar@<.7ex>@{.>}[d] &  X^\m{fl}_i \ar@<-.7ex>[d] \ar@<.7ex>[d] \\
X^\m{nd}_{i-1} \ar[r] & X^\m{fl}_{i-1} 
}\]
factor as shown. Define $X'_{n+1}$ by the Cartesian square
\[\xymatrix{
X'_{n+1} \ar[r] \ar[d] & X^\m{fl}_{n+1} \ar[d] \\
X^\m{nd}_n \times_{X^\m{nd}_{n-1}} X^\m{nd}_n \ar[r] & X^\m{fl}_n \times_{X^\m{fl}_{n-1}} X^\m{fl}_n
}\]
and let $r'_{n+1} : \fk{g}_{X'} \rightarrow \fk{n}_{n+1, X'}$ be its universal family.
\end{unmarked}

\begin{unmarked} \label{.394}
Write $n(r'_{n+1}) = \V Q$, combining \ref{.3835} with \ref{.02} as above. The Fitting ideals 
of $Q$ define a flattening stratification $X'_{n+1} = \cup (X'_{n+1})_i$ of $\fAut r'_{n+1}$. Each $(X'_{n+1})_i \subset X'_{n+1}$ is an immersed suscheme; if $T = \spec l$ is a field then $g: T \to X'_{n+1}$ lands in $(X'_{n+1})_i$ if and only if $\dim g^* \fAut r'_{n+1} = i$; and if $T$ is Noetherian, then $g: T \to X'_{n+1}$ factors through $\coprod (X'_{n+1})_i$ if and only if $g^* \fAut r'_{n+1}$ is flat. We claim that $X^\m{nd}_{n+1} = (X'_{n+1})_{A(\fk{g}, n+1)}$.
\end{unmarked}

\begin{unmarked} \label{.395}
Suppose first that $T = \spec B$ is an affine Noetherian $k$-scheme, suppose $g: T \rightarrow X^\m{fl}_{n+1}$ factors through $(X'_{n+1})_{A(\fk{g}, n+1)}$ and let $r$ be the corresponding representation over $T$. Then it is clear that $r_n$, $r^1$ are both nondegenerate, hence also in particular fiberwise-nondegenerate in the sense of \ref{.12}. Thus if $t \in T$ then $r(t)_n$, and $r(t)^1$ are nondegenerate in the sense of \ref{.12}, and $\dim \fAut r(t) = A(\fk{g}, n+1)$. Hence $r(t)$ is nondegenerate in the sense of \ref{.12}. Finally, it is clear that $\fAut r$ is flat.

Conversely, suppose $r: \fk{g}_T \rightarrow \fk{n}_{n+1, T}$ is nondegenerate ($T$ Noetherian) and let $g: T \rightarrow X^\m{fl}_{n+1}$ be the corresponding map. It is clear that $g$ factors through $X'_{n+1}$. Flatness of $\fAut r$ implies that $g$ factors through $\coprod_i (X'_{n+1})_i$. Finally, the fiberwise condition implies that set-theoretically $g$ factors though $(X'_{n+1})_{A(\fk{g}, n+1)}$, from which it follows that the previous (scheme-theoretic) factorization was actually a factorization through $(X'_{n+1})_{A(\fk{g}, n+1)}$.
\end{unmarked}

\begin{unmarked} \label{.396}
If $T = \spec B$ is not Noetherian, let $\s{B}$ be the system of finite type subalgebras. Then it follows from \ref{.38373}, from \ref{.395}, and finally from the fact that $(X'_{n+1})_{A(\fk{g},n+1)}$ is finite type over a field, hence in particular locally of finite presentation, that
\begin{align*}
X_{n+1}^\m{nd}(B) &= \lim_{\substack{ \longrightarrow \\ B' \in \s{B} }} X_{n+1}^\m{nd}(B') \\ 
&= \lim_{\substack{ \longrightarrow \\ B' \in \s{B} }} (X'_{n+1})_{A(\fk{g}, n+1)} (B') \\ 
&= (X'_{n+1})_{A(\fk{g}, n+1)} (B)
\end{align*}
(\cite[Ch. 3, Prop. 8.14.2.1]{egaiv}). This completes the proof of \ref{.39}.
\end{unmarked}

\begin{rmk} \label{.3961}
It follows from the construction that although $X^\m{nd} \hookrightarrow X^\m{fl}$ may not be an open immersion, it is close to an open immersion in the sense that it factors as a surjective closed immersion followed by an open immersion. 
\end{rmk}

\begin{rmk} \label{.3962}
Although Proposition \ref{.38373} falls short of stating that $X_n^\m{nd}$ is locally of finite presentation, it follows from the construction of \ref{.39} that it is locally of finite presentation after all.

\end{rmk}

We now discuss moduli stacks and rigidification (\ref{.3837} -- \ref{.3982}), obtaining in particular our main result for this section (\ref{.398}) as a corollary of the work completed above.

\begin{dfn} \label{.3837}
We let $\Mnflg$ denote the category fibered in groupoids over $\mathbf{Aff}(k)$ whose objects are flag representations of rank $n$. Thus an object is a pair $(T, r)$, with $T\in \m{\textbf{Aff}}(k)$ and $r:\fk{g}_T\rightarrow \iEnd{\cl{E}}$ a flag representation. A morphism $(T', r')\rightarrow (T, r)$ is a pair $(f, \phi)$ where $f:T'\rightarrow T$ is a map of affine schemes and $\phi:f^*r\rightarrow r'$ is an isomorphism of representations. We let $\Mnndg$ denote the fibered subcategory of $\Mnflg$ whose objects are nondegenerate nilpotent representations. When there is no risk of confusion we write simply $\Mnfl$ and $\Mnnd$.
\end{dfn}

\begin{prop} \label{.38371}
Both $\Mnflg$ and $\Mnndg$ are stacks for the fppf topology. 
\end{prop}

The proof is in paragraph \ref{.5578} below.

\begin{prop} \label{.397}
Let $\bb{B}_n$ act on $X\nfl(\fk{g})$ and on $X^\m{nd}_n(\g)$ by conjugation. Then (in the notation of \ref{.043})
\begin{align*}
[X^\m{fl}_n/_{ZAR}\bb{B}_n]= \Mnflg && \mbox{and} && [X\nnd/_\m{ZAR} \bb{B}_n] = \Mnndg \,.
\end{align*}
\end{prop}

\begin{proof}
Recall our notational convention (\ref{.043}) by which $[X^\m{fl}_n/\bb{B}_n]$ denotes the fibered category whose objects over $T$ are the elements of $X^\m{fl}_n(T)$ and whose morphisms $x \to y$  over $\m{id}_T$ are those elements $b$ of $\bb{B}_n(T)$ such that $bx=y$. There is an obvious map
$$[X^\m{fl}_n/\bb{B}_n]\rightarrow \mathcal{M}\nfl(\fk{g}) \,.$$
Any isomorphism between flag representations of the form $\rgttonnt$ belongs to $\bb{B}_n(T)$. Indeed, any isomorphism of representations must respect the associated filtrations (\ref{.111}); the filtration associated to a flag representation of the form $\gttonnt$ is equal to the standard flag; an element of $\bb{GL}_n(T)$ which preserves the standard flag is by definition an element of $\bb{B}_n(T)$. This shows that the map is fully faithful. Moreover, a general nondegenerate nilpotent representation $r:\fk{g}_T \rightarrow \iEnd{\mathcal{E}}$ is of the form $\fk{g}_T \rightarrow \nnt$, hence comes from $X^\m{fl}_n(\fk{g})(T)$, after possibly replacing $T$ by a Zariski covering of $T$. That is, every object of the target is Zariski locally in the image. It follows that the map factors through an isomorphism of Zariski stacks as claimed.

The same argument applies to $\Mnnd$.
\end{proof}

\Corollary{.3973}{The stacks $\Mnfl, \Mnnd$ are algebraic.}

\begin{proof} This follows from \ref{.397} in view of \ref{.383721} and \ref{.39}, respectively. \end{proof}

\begin{thm}\label{.398}
The fppf sheaf $\pi_0^\m{fppf} \cl{M}\nd_n(\fk{g})$ (\ref{.0431}) associated to $\Mnndg$ is an algebraic space.
\end{thm}

\begin{proof}
By construction, the inertia stack of $\Mnndg$ is flat. So this follows from \ref{.3973} by rigidification (see, for instance, \cite[\S1.5]{olsson}).
\end{proof}

\begin{dfn}  \label{.3981}
We let
$$M\nd_n(\fk{g}) := \pi_0^\m{fppf}\mathcal{M}_n(\fk{g})$$
and call it \textbf{the moduli space of nondegenerate nilpotent representations}.
\end{dfn}

\Remark{.3982}{ The map $\Mnnd \to \mnnd$ is fppf. Indeed, if $f:T \to \mnnd$ is a map from a $k$-scheme then there is an fppf map $T' \to T$ such that $f': T' \to T \to \mnnd$ admits a section $g: T' \to \Mnnd$. Then by \cite[1.5.6]{olsson} $f'^*\Mnnd$ is the classifying stack of a flat group scheme and is thus fppf over $T'$. This implies that $f^*\Mnnd$ is fppf over $T$ and subsequently, since $T$ and $f$ were arbitrary, that $\Mnnd$ is fppf over $\mnnd$ as claimed. }

We now discuss the functoriality of our moduli spaces (\ref{.3983} -- \ref{.39879}).

\Remark{.3983}{Let $s$ be a surjection of Lie algebras
$$\fk{f} \twoheadrightarrow \g$$
and let $\rgttoiende$ be a representation of $\g$ over a $k$-scheme $T$. Then the zero eigenspace of $rs_T$ is equal to the zero-eigenspace of $r$. Since the functor $r \mapsto rs_T$ is exact, it follows that $rs_T$ is flag if and only if $r$ is flag. In particular, $s$ gives rise to a fully faithful morphism of stacks
$$\Mnflg \hookrightarrow \Mnfl(\fk{f}) \, .$$ }

\Proposition{.3985}{ Let $s$ be a surjection of Lie algebras
$$\fk{f} \twoheadrightarrow \g \, .$$
Then the morphism
$$\Mnflg \hookrightarrow \Mnfl(\fk{f})$$
induced by $s$ as in \ref{.3983} is a closed immersion. }

The proof follows (\ref{.39853} -- \ref{.39875}).

\Lemma{.39853}{ Let $A$ be a ring and $\theta:E \to F$ a morphism of $A$-modules with $F$ locally free of finite rank. Then there exists an ideal $I(\theta)$ of  $A$ such that for all $A$-algebras $B$, $\theta_B := \m{id}_B \otimes \theta =0$ if and only if $I(\theta)B =0$. }

\begin{proof}
The assertion is local on $A$, so we may assume $F$ is free. After possibly precomposing $\theta$ with a surjection $E' \twoheadrightarrow E$, we may assume $E$ is free as well. To complete the proof, we fix bases and let $I(\theta)$ be the ideal generated by the corresponding matrix entries of $\theta$.
\end{proof}

\Unmarked{.39875}{\textit{Completion of proof of} \ref{.3985}. Let $r: \fk{f}_T \to \iEnd \cl{E}$ be a flag representation of $\fk{f}$ over an affine $k$-scheme $T =\spec A$. Let $i: \fk{a} \to \fk{f}$ be the kernel of $s:\fk{f} \twoheadrightarrow \fk{g}$ and let $Z$ be the closed subscheme of $T$ defined by $I(ri_T)$. Then by \ref{.39853}, $g: T' \to T$ factors thought $Z$ if and only if $g^*r$ factors through $\fk{g}$. Regarding $r$ as a map $T \to \Mnfl(\fk{f})$, we have constructed a cartesian square
\[\xymatrix{
Z \ar@{^{(}->}[r] \ar[d] 	&
T \ar[d]				\\
\Mnflg \ar@{^{(}->}[r]	&
\Mnfl(\fk{f})			}
\]
with $Z \hookrightarrow T$ a closed immersion, which concludes the proof.}

\Remark{.39876}{Consider again the situation of \ref{.3983} given by a surjection
$$s: \fk{f} \twoheadrightarrow \g$$
of Lie algebras and a representation $\rgttoiende$ of $\g$ over a $k$-scheme $T$. Then
$$\fAut rs_T = \fAut r \, .$$
Suppose $A(\g,i) = A(\fk{f}, i)$ for $i=1, \dots, n$. Then it follows from \ref{.3983} that $r$ is nondegenerate nilpotent if and only if $rs_T$ is nondegenerate nilpotent. In particular, $s$ gives rise to a fully faithful morphism of stacks
$$\Mnndg \hookrightarrow \Mnnd(\fk{f}) \, .$$  }

\Corollary{.39878}{Continuing with the situation of \ref{.39876}, the map
$$\Mnndg \hookrightarrow \Mnnd(\fk{f})$$
is a closed immersion.}

\begin{proof}
Indeed, the \textquotedblleft if and only if" part of \ref{.39876} implies moreover that the resulting square
\[\xymatrix{
\Mnflg \ar@{^{(}->}[r] 		&
\Mnfl(\fk{f})				\\
\Mnndg \ar@{^{(}->}[u] \ar@{^{(}->}[r]	&
\Mnnd(\fk{f}) \ar@{^{(}->}[u]			}
\]
is Cartesian, so this follows from \ref{.3983}.
\end{proof}

\Corollary{.39879}{Continuing with the situation of \ref{.39876}, $s$ gives rise to a closed immersion
$$\mnndg \hookrightarrow \mnnd(\fk{f}) \, .$$}

\begin{proof}
The square 
\[\xymatrix{
\Mnndg \ar[d] \ar@{^{(}->}[r]		&
\Mnnd(\fk{f}) \ar[d] 				\\
\mnndg \ar@{^{(}->}[r]			&
\mnnd(\fk{f})					}
\]
is automatically Cartesian. Since the vertical arrow on the right is fppf (\ref{.3982}) and the horizontal arrow at the top is a closed immersion (\ref{.39878}), the proposition follows.
\end{proof}

\begin{dfn}
Since an isomorphism of flag representations give rise to an equality of constellations (\ref{.495}), construction of constellations gives rise to a map
$$\Mnflg \to (\bb{P}\fk{g}^\m{ab})^J$$
which we denote by $\m{\mathit{const}}$.
\end{dfn}

\Remark{.496}{Although we have no use for this in the sequel, we remark that a surjection
$$\fk{f} \twoheadrightarrow \g$$
of Lie algebras gives rise to a commuting square 
\[\xymatrix{
\Mnflg \ar@{^{(}->}[r] \ar[d]		&
\Mnfl(\fk{f}) \ar[d]				\\
(\bb{P}\g^\m{ab})^J \ar@{^{(}->}[r]	&
(\bb{P}\fk{f}^\m{ab})^J			}
\] }

\section{Variant: framed representations}
\label{Framed}

We continue to work with a finitely generated Lie algebra $\g$ over a field $k$, and we preserve all notations introduced in the previous section.

Recall that $\bb{B}_n=\bb{T}_n\ltimes\bb{U}_n$ is the semidirect product of a torus and a unipotent group. In studying the quotient of $X\nnd$ by $\bb{B}_n$ it will be convenient to consider the actions of $\bb{U}_n$ on $X_n^\m{nd}$ and of $\bb{T}_n$ on the quotient of $X_n^\m{nd}$ by $\bb{U}_n$ separately. The following variant provides a natural interpretation of the quotient stack $[X_n^\m{fl}/_\m{ZAR}\bb{U}_n]$ in terms of flag representations.

\begin{dfn} \label{3988}
Let $T$ be a $k$-scheme. A \textbf{framed flag representation of} $\g$ \textbf{over} $T$ is a pair $(r,e)$ where $r$ is a flag representation $\fk{g}_T\rightarrow \iEnd{\mathcal{E}}$ of $\g$ over $T$ and $e = (e_1, \dots, e_n)$ is a basis for $\m{gr}^{\m{Fil}^r}\mathcal{E}$ (\ref{.111}) compatible with the grading. If $r,r'$ are framed flag representations, a \textbf{framed isomorphism} $r\rightarrow r'$ is an isomorphism $\phi$ of the underlying flag representations such that $\m{gr}(\phi):\m{gr}\,\mathcal{E}\rightarrow \m{gr}\,\mathcal{E'}$ is the isomorphism determined by the given bases. A \textbf{framed nondegenerate nilpotent representation} is a framed flag representation whose underlying flag representation is nondegenerate.

We let $\mathcal{M}^\m{ffl}_n(\fk{g})$ denote the fibered category of framed flag representations and we let $\Mnfndg$ denote the fibered subcategory consisting of framed nondegenerate nilpotent representations.
\end{dfn}

\begin{prop} \label{.39881}
The fibered categories $\mathcal{M}^\m{ffl}_n(\fk{g})$ and $\Mnfndg$ obey fppf descent.
\end{prop}

\begin{proof}
We start with $\Mnffl$. This is a straightforward verification using the fact that $\Mnfl$ is an fppf stack (\ref{.5578}). Let $T$ be an affine $k$-scheme, $f: T' \rightarrow T$ an fppf covering, let $T'' = T' \times_T T'$ and denote by $p_1, p_2: T'' \rightrightarrows T'$ the two projections. Denote by $\cl{M}^\m{ffl}_n(f)$ the category of descent data relative to $f$. We are to show that the functor
$$f^*:\cl{M}^\m{ffl}_n(T) \rightarrow \cl{M}^\m{ffl}_n(f)$$
is an equivalence. To see that $f^*$ is fully faithful, fix two framed representations
$$(r: \fk{g}_T \rightarrow \iEnd(\cl{E}), e^r)$$
and
$$(s: \fk{g}_T \rightarrow \iEnd(\cl{F}), e^s )$$
over $T$. Let $\alpha$ denote the induced isomorphism
$$p_2^* f^* r \rightarrow p_1^* f^* r \, ,$$
$\beta$ the induced isomorphism
$$p_2^* f^* s \rightarrow p_1^* f^* s \, ,$$
and consider a morphism of descent data
$$ \phi': (f^*r,  f^*e^r , \alpha) \rightarrow (f^*s, f^*e^s , \beta ) \, ;$$
that is, a morphism
$$\phi': (f^*r,  f^* e^r ) \rightarrow (f^*s, f^* e^s )$$
such that the square
\[\xymatrix{
p_2^* f^* r \ar[r]^{p_2^* \phi'} \ar[d]_\alpha	&
p_2^* f^* s \ar[d]^\beta	\\
p_1^* f^* r \ar[r]_{p_1^* \phi'}	&
p_1^*f^*s	}
\]
commutes. Descent for $\cl{M}\fl_n$ implies that there is a (necessarily unique) morphism of representations
$$\phi: r \rightarrow s$$
such that
$$f^* \phi = \phi' \, ;$$
on the other hand
$$f^* \m{gr} \; \phi = \m{gr} \; f^* \phi = \m{gr} \; \phi'$$
sends $f^* e^r$ to $f^* e^s$ which implies that $\m{gr} \; \phi$ sends $e^r$ to $e^s$ since restriction maps along coverings are injective. This establishes the full faithfulness.

To check essential surjectivity, fix a framed flag representation
$$(r': \fk{g}_{T'} \rightarrow \iEnd(\cl{E}'), e')$$
and a framed isomorphism
$$\alpha: (p_2^*r', p_2^*e') \rightarrow (p_1^*r', p_1^* e')$$
obeying the cocycle condition. Descent for $\cl{M}\fl_n$ produces a representation $r: \fk{g}_T \rightarrow \iEnd(\cl{E})$ whose descent data $(f^*r, \alpha^\m{can.})$ relative to $f$ is isomorphic to $(r', \alpha)$. Fixing an isomorphism
$$\phi: (r', \alpha) \rightarrow (f^*r, \alpha^\m{can.})$$
we get a diagram
\[\xymatrix@!0 @C=10ex{
	&
p_2^* \m{gr} \; \cl{E}' \ar[rr]^-{p_2^*(\m{gr} \; \phi)} \ar[dddd]^(.3){\m{gr} \; \alpha}	&&
p_2^*f^* \m{gr} \; \cl{E} \ar[dddd]_(.3){\m{gr} \; \alpha^\m{can.}}	&
	\\\\
**[l] e'_i \in \m{gr} \; \cl{E}'	\ar[uur]^{p_2^*} \ar[ddr]_{p_1^*} \ar'[r]'[rrr]_{\m{gr} \; \phi}[rrrr]	&
	&&
	&
**[r] f^* \m{gr} \; \cl{E} \ni (\m{gr} \; \phi)(e'_i) \ar[uul]_{p_2^*} \ar[ddl]^{p_1^*}	\\\\
	&
p_1^* \m{gr} \; \cl{E}' \ar[rr]_-{p_1^*(\m{gr} \; \phi)}	&&
p_1^* f^* \m{gr} \; \cl{E}	
	&}
\]
in which the (small) square at the front and the two trapezoids at the back commute.  Since
$$( \m{gr} \; \alpha)(p_2^*e'_i) = p_1^* e'_i \, ,$$
it follows that
$$(\m{gr} \; \alpha^\m{can.})p_2^* (\m{gr} \; \phi) (e'_i) = p_1^* (\m{gr} \; \phi)(e'_i) \, ,$$
hence that $\{ (\m{gr} \; \phi) (e'_i) \}$ descends to a basis $\{ e_i \}$ of $\m{gr} \; \cl{E}$ making $(r, e)$ into a framed representation. This shows that $(r', e', \alpha)$ is in the essential image of $f^*$ and completes the verification.

Finally, the same argument using the fact that $\Mnnd$ is an fppf stack shows that $\Mnfnd$ is an fppf stack.
\end{proof}

\begin{prop} \label{.39882}
Let $\bb{U}_n$ act on $X\nfl(\fk{g})$ and on $X^\m{nd}_n(\g)$ by conjugation. Then 
\begin{align*}
[X^\m{fl}_n/_{ZAR}\bb{U}_n]= \Mnfflg && \mbox{and} && [X\nnd/_\m{ZAR} \bb{U}_n] = \Mnfndg \, .
\end{align*}
\end{prop}

\begin{proof}
Consider the map
$$[X^\m{fl}_n/\bb{U}_n]\rightarrow \mathcal{M}^\m{ffl}_n(T)$$
which sends $r:\fk{g}_T\rightarrow \fk{n}_{n,T}$ to
$$(\fk{g}_T\rightarrow \fk{n}_{n,T}\hookrightarrow \iEnd{\mathcal{O}^{(n)}_T}, \text{the standard basis of } \ssf_T^{(n)}) \, .$$
If $r,r'\in X^\m{fl}_n(T)$, then an isomorphism of representations $b:r\rightarrow r'$ is in $\bb{U}_n(T)$ if and only if $\m{gr}\,b=\m{id}_{\mathcal{O}^{(n)}_T}$. The same argument applies to $\Mnfnd$.
\end{proof}


\begin{cor} \label{.39883}
The fppf-sheaf $\pi_0^\m{fppf}\cl{M}_n^\m{fnd}(\fk{g})$ associated to $\cl{M}_n^\m{fnd}(\fk{g})$ is algebraic.
\end{cor}

\begin{dfn} \label{.39884}
We define $M^\m{fnd}_n(\fk{g})$, \textbf{the moduli space of framed nondegenerate nilpotent representations}, by
$$M^\m{fnd}_n(\fk{g}) = \pi_0^\m{fppf}\cl{M}_n^\m{fnd}(\fk{g}) \, .$$
\end{dfn}

\Remark{.3992}{The argument of \ref{.3982} applies to show that the map $\Mnfndg \to \mnfndg$ is fppf.} 

\Remark{.3993}{Remark \ref{.3983} applies without essential change to show that a surjection of Lie algebras
$$s:\fk{f} \twoheadrightarrow \g$$
gives rise to a fully faithful morphism
$$\Mnfflg \hookrightarrow \Mnffl(\fk{f})$$
of stacks, and the proof of \ref{.3985} applies without essential change to show that this map is in fact a closed immersion. 

Suppose, moreover, that for $i=1,\dots,n$, $A(\fk{f},i)=A(\g,i)$. Then paragraphs \ref{.39876} -- \ref{.39879} apply without essential change to show that $s$ gives rise to closed immersions
$$\Mnfndg \hookrightarrow \Mnfnd(\fk{f})$$
and
$$\mnfndg \hookrightarrow \mnfnd(\fk{f})\,.$$}

We conclude this section with a brief discussion of the framed analogs of canonical coordinates and constellations.

\Definition{.399303}{ Fix a $k$-scheme $T$ and a framed flag representation $(r,e)$ on a vector bundle $\cl{E}$ of rank $n$ over $T$. For each pair of integers $i,j$ with $1 \le i < j \le n$ we define the $(i,j)^\m{th}$ \textbf{canonical matrix entry of} $(r,e)$ to be the composite
\[\xymatrix{
(\fk{h}_{j-i})_T \ar[r]_-{\ka_{i,j}^r}  \ar@/^15pt/[rr]^{\ka_{i,j}^{(r,e)}} &
\cl{H}om(\cl{L}^r_j, \cl{L}^r_i) \ar[r]_-\cong &
\ssf_T
}\]
of the $(i,j)^\m{th}$ canonical matrix entry $\ka^r_{i,j}$ of the underlying flag representation (\ref{.3813}) with the isomorphism induced by the basis. In particular, each $\ka_{i,i+1}^{(r,e)}$ is surjective (\ref{.38}) and we define the \textbf{constellation of} $(r,e)$, denoted $\mathit{const}(r,e)$, to be the vector
$$(\ka_{1,2}^{(r,e)}, \dots, \ka_{n-1,n}^{(r,e)}) \in (\V_*\g^\m{ab})^J(T)$$
where
$$J= \{ (1,2), \dots, (n-1,n) \} \,. $$
}

\Remark{.399304}{ Formation of constellations in the framed and unframed cases gives rise to a commuting square
\[\xymatrix{
\mnfnd \ar[r] \ar[d] &  (\V_*\g^\m{ab})^J \ar[d] \\
\mnnd \ar[r] & (\bb{P}\g^\m{ab})^J
}\]
The horizontal arrow at the top is equivariant with respect to the actions of $\bb{T}_{n-1}$ given on the left by conjugation by the diagonal torus and on the right by
$$(a_1, \dots, a_{n-1})(\ka_{1,2},\dots,\ka_{n-1,n}) = (a_1\ka_{1,2}, \dots, a_{n-1}\ka_{n-1,n}) \,.$$
Since the action of $\bb{T}_{n-1}$ on the right is free, it follows that the action on the left is also free. Thus we have
\begin{align*}
\mnnd = [\mnfnd/_\m{ZAR}\bb{T}_{n-1}] &&\text{and} && (\bb{P}\g^\m{ab})^J = [(\V_*\g^\m{ab})^J/_\m{ZAR}\bb{T}_{n-1}] \,.
\end{align*}
It follows that the square is in fact Cartesian and that $\mnfnd$ is a $\bb{T}_{n-1}$-torsor over $\mnnd$.
 }

\section{Low dimensional examples}
\label{Low}

We continue to work with a finitely generated Lie algebra $\g$ over a field $k$.

\subsection*{Two dimensions}

\begin{prop} \label{.3902}
The moduli space $M\nd_2(\fk{g})$ of two dimensional nondegenerate nilpotent representations is canonically isomorphic to $\bb{P}\fk{g}^\m{ab}$.
\end{prop}

\begin{proof}
As observed in \ref{.15}, every flag representation is nondegenerate. So 
$$X_2^\m{nd} = X_2^\m{fl} \hookrightarrow X_2 = \fHom{_\m{\bld{Lie}}(\fk{g}, \fk{n}_2)}=\V\fk{g}^\m{ab}$$
is just the complement of the origin. The action by $\bb{G}_a$ is trivial, and the action of $\bb{T}_2$ factors through $\bb{T}_2/\bb{G}_m = \bb{G}_m$. The resulting action by $\bb{G}_m$ is just the weight-one action defining projective space. 
\end{proof}

\subsection*{Three dimensions}

\begin{prop} \label{.3901}
Let $m$ denote the dimension of $\g$ and let $w$ denote the width of $\g$. If $w=1$ then the moduli space ${M}^\m{nd}_3(\fk{g})$ of three dimensional nondegenerate nilpotent representations is a closed subscheme of a vector bundle of rank $m-1$ over $\bb{P}\fk{g}^\m{ab}$. If $w \ge 2$ then ${M}^\m{nd}_3(\fk{g})$ is a closed subscheme of a vector bundle of rank $m-2$ over the complement of the diagonal of $\bb{P}\fk{g}^\m{ab} \times \bb{P}\fk{g}^\m{ab}$.
\end{prop}

The proof follows in paragraphs \ref{.393}--\ref{.3932}.

\Unmarked{.393}{
We begin by noting that
\begin{align*}
w(\g/\g^{(3)}) = \left\{ 
\begin{array}{ll}
1 & \mbox{if } w(\g)=1 \\
2 & \m{otherwise} .	\label{.167}
\end{array}
\right.
\end{align*}
Replacing $\fk{g}$ by $\fk{g}/\fk{g}^{(3)}$, we may suppose that $w$ is equal to either $1$ or $2$ and, moreover, that $\fk{g}$ is finite dimensional. Let $S_1 \subset \bb{P}\fk{g}^\m{ab}\times\bb{P}\fk{g}^\m{ab}$ denote the diagonal and $S_2$ its complement, and for $i=1,2$ let
$$\tilde{S}_i := S_i \underset{\bb{P}\fk{g}^\m{ab} \times\bb{P} \fk{g}^\m{ab}} \times (\V_*\fk{g}^\m{ab})^2$$
For $j=1,2,3$ let $\la^r_j$ denote the $j^\m{th}$ full matrix entry of $r$, indexed as in \ref{.16}.
Then we have the following diagram:
\[
\xymatrix{
X^\m{nd}_3 \ar@{_{(}->}[d]_{j'} \ar@{^{(}->}[r]			& 
X^\m{fl}_3 \ar@{_{(}->}[d] \ar@{^{(}->}[r] 				& 
X_3 \ar@{_{(}->}[d]^j	& r \ar@{|->}[d] 					 \\
\V\fk{g}_{\tilde{S}_w} \ar[d] \ar@{^{(}->}[r] 				& 
(\V_*\fk{g}^\m{ab})^2\times\V\fk{g} \ar[d] \ar@{^{(}->}[r]		&
(\V\fk{g}^\m{ab})^2 \times \V \fk{g} \ar[d]									& 
(\la^r_1,\la^r_2,\la^r_3) \ar@{|->}[d]						\\
\tilde{S}_w \ar[d] \ar@{^{(}->}[r]						& 
(\V_*\fk{g}^\m{ab})^2 \ar[d]^\pi \ar@{^{(}->}[r]				&
(\V\fk{g}^\m{ab})^2										& 
(\la^r_1,\la^r_2)									\\
S_w \ar@{^{(}->}[r]								&
\bb{P}\fk{g}^\m{ab} \times\bb{P} \fk{g}^\m{ab} 						}
\]
We claim that all squares appearing in the diagram are Cartesian. This is clear below the top row. For the square in the upper right, this follows from \ref{.38}. We discuss the square in the upper left. Note that if $w=1$ then $A(\g,3)=3$ and if $w \ge 2$ then $A(\g, 3) = 2$. Applying this to \ref{.16}, we have: if $w=2$ then nondegeneracy is equivalent to the condition that $\la_1,\la_2$ be linearly independent, which is precisely the condition imposed by cartesianness of the square in that case; if, on the other hand, $w=1$, then every flag representation is nondegenerate, and the upper left square, which now has an equality at the top, just expresses the fact that every (flag) representation has $\la_1, \la_2$ linearly dependent. 

To complete the set up, note that $\bb{B}_3$ acts on $(\V\fk{g}^\m{ab})^2 \times \V \fk{g}$, with $\V\fk{g}_{\tilde{S}_w}$ and $X^\m{nd}_3$ invariant.}

\Unmarked{.3931}{
We claim that $\V\fk{g}_{\tilde{S}_w}/_\m{ZAR}\bb{B}_3$ forms a vector bundle of rank $m-w$ over $S_w$.}

\begin{proof}
We begin by recording the action of $\bb{U}_3$:
$$(b_1,b_2,b_3)(\la_1,\la_2,\la_3)=(\la_1,\la_2,\la_3+b_1\la_2-b_2\la_1)\,.$$
Define the vector sheaf $\tilde{\mathcal{E}}$ on $\tilde{S}_w$ by
\[
\ssf_{\tilde{S}_w}^{(2)} \to  \fk{g}_{\tilde{S}_w}^\lor \xrightarrow{\phi}  \tilde{\mathcal{E}} \to  0 
\]
where the first arrow is given by the canonical pair of sections. Since this system has constant rank $w$, $\tilde{\mathcal{E}}$ has constant rank equal to $m-w$. In particular, $\tilde{\mathcal{E}}$ is locally free. Then
$$\V \phi^\lor : \V \fk{g}_{\tilde{S}_w} \rightarrow \V\tilde{\mathcal{E}}^\lor$$
is the presheaf quotient of $\V \fk{g}_{\tilde{S}_w}$ by $\bb{U}_3$. Indeed, an orbit is determined precisely by $\la_1, \la_2$ and the image of $\la_3$ modulo $\la_1, \la_2$. 
Now $\bb{T}_2=\bb{T}_3/\bb{G}_a$  acts on $\V\tilde{\mathcal{E}}^\lor$. The action is given by
$$(a_1,a_2)(\la_1,\la_2,v)=(a_1\la_1,a_2\la_2, a_1a_2 v)$$
This is the same as descent data for the vector sheaf $\tilde{\mathcal{E}}$ along the faithfully flat morphism $\tilde{S}_w \rightarrow S_w$, hence gives rise to a vector sheaf $\mathcal{E}$ on $S_w$ forming a Cartesian square
\[\xymatrix{
									& \V\fk{g}_{\tilde{S}_w} \ar[dl]_{\V\phi^\lor} \ar[d] 	\\
\V\tilde{\mathcal{E}}^\lor \ar[r] \ar[d]_{\pi''}	& \tilde{S}_w \ar[d]_{\pi'}						\\
\V \mathcal{E}^\lor \ar[r]					& S_w
}\]
with $\pi''$ the Zariski sheaf quotient of $\V\tilde{\mathcal{E}}^\lor$ by $\bb{T}_2$. This establishes the claim, and shows, moreover, that the projection to the quotient is fppf.
\end{proof}

\Unmarked{.3932}{
It is a general fact about sheaf quotients that we get a Cartesian square
\[\xymatrix{
X^\m{nd}_3 \ar@{^{(}->}[r]^{j'} \ar[d] 			& \V\fk{g}_{\tilde{S}_w} \ar[d]^{\psi=\pi'' \circ \V \phi^\lor} 	\\
X^\m{nd}_3/_\m{ZAR}\bb{B}_3 \ar@{^{(}->}[r]		& \V\mathcal{E}^\lor
}\] 
Since $j'$ is a closed immersion and $\psi$ is fppf, it follows by descent that $X^\m{nd}_3/_\m{ZAR}\bb{B}_3$ is representable by a closed subscheme of $\V\mathcal{E}^\lor$.}

\section{Moduli of wide nilpotent representations}

\label{ModW}

In this section we give an explicit construction of $M\nd_n(\fk{g})$ for $n \le w+1$, proving, in particular, that $M\nd_n(\fk{g})$ is quasi-projective. This generalizes Proposition \ref{.3901} in the case $w(\g)\ge 2$, and the construction is similar. We continue to work with a finitely generated Lie algebra $\g$ over a field $k$.

\begin{prop} \label{.5}
Suppose $\fk{g}$ has width $w$ and let $r: \fk{g}_T \rightarrow \iEnd(\cl{E})$ be a flag representation of $\fk{g}$ on a vector sheaf $\cl{E}$ of rank $n \le w+1$ over an affine $k$-scheme $T = \spec B$. Then $r$ is nondegenerate if and only if the points of $\m{\textit{const}} \; r(t)$ are distinct for all $t \in T$.
\end{prop}

\begin{proof}
Suppose first that $T$ is Noetherian and suppose the points of $\m{\textit{const}} \; r(t)$ are distinct for all $t \in T$. Recall our notation $\cl{L}_i = \gr_i\cl{E}$ (\ref{.3813}). We claim that
$$n(r) = \fHom(\cl{L}_n, \cl{L}_1)\,.$$
Indeed, there is in general a functorial injection
$$\fHom(\cl{L}_n, \cl{L}_1) \hookrightarrow n(r)$$
given by
$$\phi \mapsto ( x \mapsto \phi(\bar{x}))$$
on $T'$-valued points ($T'$ an arbitrary $T$-scheme), where $\bar{x}$ denotes the image of $x$ in $\Gamma(T', \cl{L}_n)$. To see that our assumption makes this into an isomorphism it is enough to look Zariski locally on $T$ so that $\cl{E}$, together with its flag becomes trivial, and to evaluate on an affine Noetherian $T$-scheme $\spec A$ (since both functors are locally of finite presentation). Fixing a basis compatible with the flag, we thus have $n(r)(A) \subset {\fk{n}_n}_A$. 

Now fix $(a_{i,j}) \in n(r)(A)$, assume for an induction on $s$ that $a_{i,j} = 0$ for $j-i <s$ and fix $i,j$ with $j-i = s$. Then equivariance of $(a_{i,j})$ with $r$ implies
$$a_{i+1,j+1}\la_{i,i+1} = a_{i,j}\la_{j,j+1}\,.$$
Then the fiberwise linear independence of $\la_{i,i+1},\la_{j,j+1}$ implies
$$a_{i+1,j+1} = a_{i,j} = 0\,.$$
Indeed, consider the map of $A$-modules
$$\psi: A^2 \rightarrow A \otimes \fk{g}^\lor$$
defined by
$$(b,c) \mapsto b \la_{i,i+1} + c \la_{j,j+1}\,.$$
By \cite[20.8]{eisenbud} the cokernel splits as a free submodule of co-rank $2$, hence the image is a free submodule of rank $2$, hence $\psi$ is injective.

This shows that $a_{i,j} = 0$ for $1 \le j-i \le n-2$ and establishes the claimed isomorphism $n(r) = \fHom(\cl{L}_n, \cl{L}_1)$. Thus the automorphism group is flat, its fiberwise dimension is minimal and, assuming for an induction that the subquotients of $r$ are nondegenerate, it follows that $r$ is nondegenerate.

For the converse with $T$ still assumed Noetherian, suppose there is a $t \in T$ such that the points of $\m{\textit{const}} (r)$ are not all distinct. Our goal being to show that $r$ is not nondegenerate, after possibly replacing $T$ by a nonempty Zariski open subset, we may suppose $(\cl{E}, \m{Fil}^r)$ to be trivial. After possibly replacing $r$ by a subrepresentation, we may assume that the $\la_{i,i+1}$ are pairwise linearly independent with the exception of $\la_{1,2}, \la_{n-1,n}$. Then the equation 
\[ a_{2,n} \la_{1,2}(t) = a_{1,n-1} \la_{n-1, n}(t) \]
has a one dimensional space of solutions, so $\dim \fAut r(t) = 3$ is not minimal.

If $B$ is not Noetherian, write $E$ for the global sections of $\cl{E}$ and $\m{Fil}$ for the flag associated to $r$ and let $B_0$ be a finite type subalgebra over which $E, \m{Fil}$ are defined: $E= B\otimes_{B_0} E_0$ with $\m{Fil}$ induced from a filtration $\m{Fil}_0$ on $E_0$. Let $\s{B}$ denote the system of finite type $B_0$-subalgebras of $B$. The functor $X^\m{wide}_{(B_0, E_0, \m{Fil}_0)}$ whose points are wide representations on $E_0$ with filtration $\m{Fil}_0$ is an open subfunctor of $X^\m{fl}_{(B_0, E_0, \m{Fil}_0)}$, so is, in particular, locally of finite presentation. Thus by \ref{.38373} we have an equality of subsets of $X^\m{fl}_{(B_0, E_0, \m{Fil}_0)}(B)$:
 \begin{eqnarray}
 X^\m{nd}_{(B_0, E_0, \m{Fil}_0)} (B) &=& 
 \lim_{\substack{\longrightarrow \\ B' \in \s{B} }} X^\m{nd}_{(B_0, E_0, \m{Fil}_0)} (B')  \nonumber \\
 &=& 
 \lim_{\substack{\longrightarrow \\ B' \in \s{B} }} X^\m{wide}_{(B_0, E_0, \m{Fil}_0)} (B') \nonumber \\
 &=& 
 X^\m{wide}_{(B_0, E_0, \m{Fil}_0)} (B)
\end{eqnarray}
\end{proof}

\begin{thm} \label{.5011}
Let $w$ denote the width of $\fk{g}$. If $n \le w+1$ then $\pi_0 ^\m{ZAR} \Mnfndg$ (\ref{.39881}) is quasi-projective. Thus
$$\mnfndg = \pi_0 ^\m{ZAR} \Mnfndg$$
(\ref{.39884}) and in particular, $\mnfndg$ is quasi-projective.
\end{thm}

The proof follows (\ref{.5012} -- \ref{.50148}).

\begin{unmarked} \label{.5012}
We have
\begin{align*}
\pi_0^\m{ZAR}\Mnfndg &= \pi_0^\m{ZAR} [X \nnd /_\m{ZAR} \bb{U}_n ]\\ 
&= X_n^\m{nd}(\g)/_\m{ZAR}\bb{U}_n\,,
\end{align*}
so our goal is to show that the Zariski sheaf quotient $X_n^\m{nd}(\g)/_\m{ZAR}\bb{U}_n$ is quasi-projective. We first consider the case
$$\fk{g} = \fk{n}(F)$$
free pronilpotent on a finite dimensional vector space $F$ of dimension not less than $2$. Let $K$ be the set of pairs $(i,j)$ of integers $1\le i<j\le n$ and let
$$I = \{(i,j) \in K \; | \; j-i \ge 2 \}\,,$$
$$J= \{ (i,j) \in K \; | \; j-i = 1 \}$$
be its partition into first superdiagonal and higher superdiagonals. Let
$$W \hookrightarrow (\bb{P}F)^J$$
denote the complement of the big diagonal, that is, the open subscheme of $J$-tuples of distinct points of $\bb{P}F$. Write $\V_*F$ for $\V F \setminus \{0\}$ and let
$$\widetilde{W} \hookrightarrow (\V_*F)^J$$
denote the cone above $W$. By \ref{.5},
$$X\nnd(\fk{n}(F)) = \widetilde{W} \times \V F^I = \V F^I_{\widetilde{W}}\,.$$ 
\end{unmarked}

\begin{unmarked} \label{.5013}
Let $B$ be a $k$-algebra and consider an arbitrary element $u \in \bb{U}_n(B)$. Then
$$(u^{-1})_{i,j} = -u_{i,j} + P_{i,j}$$
where $P_{i,j}$ is a polynomial in $u_{i',j'}$ with $i' \ge i, j' \le j$ and $(i',j') \neq (i,j)$, hence, in particular, with $j'-i' < j-i$. Indeed, since formation of the inverse of an automorphism commutes with taking subquotients along a filtration respected by the automorphism, it is sufficient to check this statement for $(i,j)=(1,n)$. Let $u_{\widehat{i,j}}$ denote the matrix $u$ with the $i^\m{th}$ row and $j^\m{th}$ column removed and let $u_{ \widehat{ i,j }, \widehat{ i',j' } }$ denote $u$ with rows $i,i'$ and columns $j,j'$ removed. Then in this case, we have
\begin{align}
(u^{-1})_{1,n} 	& = (-1)^{n+1} \det u_{\widehat {n,1}} \nonumber \\
			& = (-1)^{n+1} \sum_{j=2}^n (-1)^{j} u_{1,j} \det  u_{ \widehat{ n,1 }, \widehat{ 1,j } }   \nonumber \\
			& = \left( (-1)^{n+1} \sum_{j=2}^{n-1} (-1)^{j} u_{1,j} \det  u_{ \widehat{ n,1 }, \widehat{ 1,j } } \right) - u_{1,n} \nonumber
\end{align}
and we need only remark that the parenthetical term is a polynomial in $u_{i',j'}$ with $(i',j') \neq (1,n)$.
\end{unmarked}

\begin{unmarked} \label{.5014}
Continuing with the situation and the notation of \ref{.5013}, let $\la$ be a $B$-valued point of $\V F_{\widetilde{W}}^I$. Then $\la$ corresponds to a family of $B$-linear maps
$$(\la_{i,j}: F_B := B \otimes F \to B)_{(i,j)\in K}$$
indexed by $K$, with $\{\la_{i,j}(t)\}_{(i,j)\in J}$ pairwise linearly independent for every $t\in T:= \spec B$, and corresponds further to a wide nilpotent representation $r:\fk{n}(F)_T \to {\nn}_T$ such that for each $(i,j)\in K$, $\la_{i,j}$ is the restriction of $\la^r_{i,j}$ to $F$.

We denote the action of $u$ on $\la$ by $u\la u^{-1}$. In terms of the linear functionals $\la_{i,j}$, the calculation of \ref{.5013} gives rise to a formula for the component linear functionals of $u\la u^{-1}$. For notational convenience, we allow scalars to act on the right.
\[\setlength{\extrarowheight}{2ex}
\begin{array}{rll}
(u \la u^{-1})_{i,j} & = & \la_{i,j} -\la_{i,i+1}u_{i+1,j} + u_{i,j-1} \la_{j-1,j} \\
&& - \underset{i+1<m<j}{\sum}  \la_{i,m} u_{m,j} + \underset{i<l<j-1}{\sum} u_{i,l}\la_{l,j} -\underset{i<l<m<j}{\sum} u_{i,l} \la_{l,m} u_{m,j}  \\
&& + \la_{i,i+1} P_{i+1,j} + \underset{i+1 <m<j}{\sum} \la_{i,m} P_{m,j} + \underset{i<l<m<j}{\sum} u_{i,j} \la_{l,m} P_{m,j} \\
& = & \la_{i,j} -\la_{i,i+1}u_{i+1,j} + u_{i,j-1} \la_{j-1,j} + \underset{ (l,m) \neq (i,j) }{\underset{i\le l < m \le j}{\sum}} Q_{l,m}\la_{l,m}
\end{array}\]
where each $Q_{l,m}$ is a polynomial in $u_{l',m'}$ with $l',m'$ subject to the conditions $i \le l' < m' \le j$ and $m'-l'\le j-i-2$. We note in particular that the action factors through the quotient $U:= \bb{U}_n/\ga$ of $\bb{U}_n$ by the subgroup defined by $u_{i,j}=0$ for $(i,j)\neq (1,n)$.
\end{unmarked}

\begin{unmarked} \label{.50143}
Let $G$ denote the set of $(n-2)$-ples $(E^\lor_2,\dots,E^\lor_{n-1})$ of subspaces of $F^\lor$ of codimension $2$ and fix arbitrarily an element $\gamma \in G$. Define an open subset $\widetilde{W}_\gamma \subset \widetilde{W}$ as follows. Given a point $\la \in \widetilde{W}$ let $k(\la)$ denote its residue field, let $(\la_{1,2},\dots,\la_{n-1,n})$ denote the corresponding $J$-tuple of lineal functionals $F_{k(\la)} \to k(\la)$ and require that for each $i= 2, \dots, n-1$, the subspace of $F^\lor_{k(\la)}$ generated by the pair $\la_{1,2}, \la_{i,i+1}$ intersect $(E^\lor_i)_{k(\la)}$ trivially:
$$\widetilde{W}_\gamma= \left\{ \la=(\la_{1,2},\dots,\la_{n-1,n}) \in \widetilde{W} \;\big|\; \langle\la_{1,2},\la_{i,i+1}\rangle \cap (E^\lor_{i})_{k(\la)} = 0 \;\m{in}\; \forall i \right\}\,.$$
Since the projection $\V F^I_{\wt{W}} \to \wt{W}$ is $U$-equivariant for the trivial action of $U$ on the target, $U$ acts on the restriction $\V F^I_{\wtg}$ of $\V F^I_{\wt{W}}$ to $\wtg$.

For each $i=1,\dots, n-1$, $F^\lor_{\widetilde{W}}$ has a canonical section $\la^c_{i,i+1}$; denote its restriction to $\widetilde{W}_\gamma$ by $\la^\gamma_{i,i+1}$. Since for $i \neq 1$, the sections $\la^\gamma_{1,2}, \la^\gamma_{i,i+1}$ are at each point of $\widetilde{W}_\gamma$ linearly independent of each other and of $(E^\lor_i)_{\widetilde{W}_\gamma}$, the trivial vector sheaf $F^\lor_{\widetilde{W}_\gamma}$ decomposes as
$$F^\lor_{\wtg} = \ssf_{\wtg}\la^\gam_{1,2} \oplus \ssf_{\wtg} \la^\gam_{i,i+1} \oplus (E^\lor_i)_{\wtg} \, .$$
Define for each $(i,j)\in I$ a vector subsheaf $F_\gam(i,j)^\lor$ of $F^\lor_{\wtg}$ by
\[ F_\gam(i,j)^\lor := \left\{
\begin{array}{ll}
(E^\lor_{j-1})_{\wtg} 	& \mbox{if } i=1 \\
\ssf_{\wtg}\la^\gam_{1,2} \oplus (E^\lor_i)_{\wtg} 	& \mbox{otherwise}
\end{array}
\right.\]
and define $V^\lor_\gam$ to be the vector subsheaf
\[ V_\gam^\lor := \bigoplus_{(i,j)\in I} F_\gam (i,j)^\lor \]
of $F^{\lor I}_{\wtg}$.
\end{unmarked}

\begin{unmarked} \label{.50144}
The construction of \ref{.50143} gives rise to a closed immersion $\V V_\gam \hookrightarrow \V F^I_{\wtg}$ and hence, by linearization with respect to the action of $U$, to a $U$-equivariant map $U\times \V V_\gam \to \V F^I_{\wtg}$ which we denote by $d_\gam$. We claim that $d_\gam$ is an isomorphism. 
\end{unmarked}

The next two paragraphs (\ref{.50145} -- \ref{.50146}) are devoted to a proof of \ref{.50144}.

\begin{unmarked} \label{.50145}
We fix an affine $k$-scheme $T = \spec B$ and a point $\la \in \V F^I_{\wtg}(T)$, and we set out to prove that there exists a unique point $(u,\la') \in U(T) \times \V V_\gam(T)$ mapping to $\la$ under $d_\gam$. Note first that the restriction $\la|_J$ of $\la$ to $J$ corresponds to a map $T\to \wtg$ and that for any $u \in U(T)$, $u \la u^{-1}$ corresponds to an element of the pullback $(\la|_J)^*F^{\lor I}_{\wtg}$ of $F^{\lor I}_{\wtg}$ along this map; in concrete terms, our goal is to show that there exists a unique element $u \in U(T)$ such that $u \la u^{-1}$ is actually contained in the submodule $(\la|_J)^*V^\lor_\gam$ of $(\la|_J)^*F^{\lor I}_{\wtg}$. We impose a total ordering on $I$: $(i,j) < (i',j')$ if $j-i < j'-i'$ or if both of the following hold  $j-i = j'-i'$ and $i < i'$. 
\end{unmarked}

The following lemma is a strengthening of \ref{.50144} suited for an induction on $I$.

\begin{lm} \label{.50146}
For each $(i,j)$ there exist elements $b_{i',j'} \in B$, $(i',j') \le (i+1,j)$ such that
$$(u\la u^{-1})_{i'',j''} \in (\la|_J)^*F_\gam(i'',j'')^\lor \text{ for all } (i'',j'') \le (i,j)$$
if and only if
$$u_{i',j'} = b_{i',j'} \text{ for all } (i',j') \le (i+1,j) \,.$$
\end{lm}

\begin{proof}
The base case is in fact a special case of the inductive step with vacuous inductive hypothesis. For the inductive step, fix arbitrarily $(i_0,j_0) \in I$, suppose the lemma holds for $i=i_0,j=j_0$ and let $(i_1,j_1)$ be the immediate successor of $(i_0,j_0)$. Consider first the case given by $j_0 \neq n$. In this case, $(i_0,j_0) = (i_1 -1, j_1 -1)$. Note the decomposition
\[ (\la |_J)^*F^\lor_{\wtg} = B\la_{1,2} \oplus B\la_{i_1,i_1 +1} \oplus (E^\lor_{i_1})_B \,. \]
For $(i',j') \le (i_1, j_1-1)$ set $u_{i',j'}$ equal to the element $b_{i',j'} \in B$ determined by the inductive hypothesis. Then
\[ (u\la u^{-1})_{i_1,j_1} = \la_{i_1,j_1} + b_{i_1,j_1-1}\la_{j_1 -1, j_1} - u_{i_1 +1, j_1}\la_{i_1, i_1 +1} + \sum Q_{l,m}\la_{l,m} \]
with each $Q_{l,m}$ appearing in the sum a polynomial in the elements $b_{i',j'}$ of $B$ determined by the inductive hypothesis, and, in particular, independent of $u_{i_1 +1, j_1}$. Thus, there exists an element $b_{i_1 +1, j_1} \in B$ such that
$$(u \la u^{-1})_{i_1,j_1} \in B \la_{1,2} \oplus (E^\lor_{i_1})_B = (\la |_J)^* F_\gam(i_1, j_1)$$
if and only if $u_{i_1 +1, j_1} = b_{i_1 +1, j_1}$. The case $j_0 = n$ is similar.
\end{proof}

\begin{unmarked} \label{.50147}
Consequently, the composite
$$\V F^I_{\wtg} \rightarrow U\times \V V_\gam \rightarrow \V V_\gam$$
of the isomorphism $d_\gam^{-1}$ with the second projection of $U \times \V V_\gam$ is the presheaf quotient of $\V F^I_{\wtg}$ by the action of $U$. Since $\bigcup_{\gam \in G} \wtg = \wt{W}$ and hence $\bigcup_{\gam \in G} \V F^I_{\wtg} = \V F^I_{\wt{W}}$, we have a family of local presheaf quotients defined over a covering. These glue automatically to produce a global Zariski sheaf quotient, as I now explain.

Given elements $\gam, \beta \in G$, we denote the intersection $\wtg \cap \wt{W}_{\beta}$ by $\wt{W}_{\beta \gam}$. Then the square 
\[
\xymatrix{\V F^I_{\wt{W}_{\beta \gam}} \ar[d] \ar@{^{(}->}[r] & \V F^I_{\wtg} \ar[d] \\ \V F^I_{\wt{W}_{\beta \gam}}/U \ar@{^{(}->}[r] & \V V_\gam}
\]
is automatically Cartesian (a general fact about (pre)sheaf quotients). Since the top arrow is an open immersion and the vertical arrow on the right is fppf, it follows that the horizontal arrow on the bottom is an open immersion. The cocycle condition follows from the fact that the intersection $\V F^I_{\wt{W}_{\alpha\beta}}/U \cap \V F^I_{\wt{W}_{\beta\gam}}/U$ in $\V V_\beta$ is canonically identified with $\V F^I_{\wt{W}_{\alpha \beta \gam}}/U$.
\end{unmarked}

\begin{unmarked} \label{.50148}
This completes the proof of the first statement in the special case $\g = \fk{n}(F)$. It follows that $\pi_0^\m{ZAR}\Mnfndg$ is representable by a scheme, hence is an fppf sheaf. Thus,
\begin{align*}
\pi_0^\m{ZAR}\Mnfndg &=  \pi_0^\m{fppf}\Mnfndg \\ 
&= \mnfndg
\end{align*}
which completes the proof of the theorem in the special case. For the general case, fix a surjection
$$s:\fk{f} \twoheadrightarrow \g$$
from a free Lie algebra on two or more generators. Then $n \le w(\fk{f}) +1$ so by \ref{.3993}, $s$ induces a fully faithful morphism
$$\Mnfndg \hookrightarrow \Mnfnd(\fk{f}) \,. $$
This gives rise to a square
\[\xymatrix{
\Mnfndg \ar@{^{(}->}[r] \ar[d]		&
\Mnfnd(\fk{f}) \ar[d]				\\
\pi_0^\m{ZAR}\Mnfndg \ar@{^{(}->}[r]	&
\pi_0^\m{ZAR}\Mnfnd(\fk{f})			}
\]
concerning which we have the following facts: the square is Cartesian; the bottom right corner is algebraic; the vertical map on the right is fppf (\ref{.3992}); the horizontal map at the top is a closed immersion (\ref{.3993}). It follows that $\pi_0^\m{ZAR}\Mnfndg$ is algebraic, from which the theorem follows as in the free case.
\end{unmarked}

\begin{thm} \label{.50149}
Let $w$ denote the width of $\fk{g}$ and suppose that $n \le w+1$. Then $\pi_0 ^\m{ZAR} \Mnndg$ is quasi-projective. Thus
$$\mnndg = \pi_0 ^\m{ZAR} \Mnndg \,,$$
and in particular, $\mnndg$ is quasi-projective.
\end{thm}

\begin{proof}
We have
\begin{align*}
\pi_0 ^\m{ZAR} \Mnndg &= X\nnd/_\m{ZAR}\bb{B}_n \\
&= (X\nnd/_\m{ZAR}\bb{U}_n)/_\m{ZAR}\bb{T}_n \\
&= \mnfndg/_\m{ZAR}\bb{T}_n \,.
\end{align*}
Continuing with the notation of \ref{.5012}, consider the action of $\bb{T}_{n-1}$ on $\widetilde{W}$ by 
$$(a_1,\dots,a_{n-1})(\la_{1,2},\dots,\la_{n-1,n})=(a_1\la_{1,2},\dots,a_{n-1}\la_{n-1,n})$$
so that $W = \widetilde{W}/_\m{ZAR}\bb{T}_{n-1}$. Recall that the action of $\bb{T}_n$ on $M^\m{fnd}_n$ factors through $\bb{T}_{n-1} = \bb{T}_n/\bb{G}_m$. The projection $M_n^\m{fnd} \rightarrow \widetilde{W}$ is then $\bb{T}_{n-1}$-equivatiant. Since both actions are functorially free, we get a Cartesian square
\[\xymatrix{
M^\m{fnd}_n \ar[r] \ar[d]	&
\widetilde{W} \ar[d] 	\\
[M^\m{fnd}_n/_\m{ZAR} \bb{T}_{n-1}] \ar@{=}[d] \ar[r]	&
[ \widetilde{W} /_\m{ZAR} \bb{T}_{n-1} ] \ar@{=}[d]	\\
M^\m{fnd}_n /_\m{ZAR} \bb{T}_{n-1} \ar[r]	&
W	}
	\]
Thus, $M^\m{fnd}_n /_\m{ZAR} \bb{T}_{n-1}$ is fppf locally representable by an affine scheme, hence is itself representable by an affine scheme.

The second statement follows from the first as in \ref{.50148}.
\end{proof}

We conclude with a remark concerning the prospects of extending the results of this subsection beyond the bound $w+1$.

\begin{rmk} \label{.5025}
Let
$$J(l) = \{ (i, j) \in [1,n] \; | \; 0 < j-i \le l \} \, .$$
Suppose $n \le 2w+2$. Then the action of $\bb{U}_n/\bb{U}_n^{(w+1)}$ on $X^\m{nd}_n /_\m{fppf} \bb{U}_n^{(w+1)}$ is free. On the other hand, the map
$$f: X^\m{nd}_n \rightarrow \V \fk{g}^{J(l)}$$
is equivariant relative to the action of $\bb{U}_n^{(w+1)}$ ($\bb{U}_n^{(w+1)}$ acts trivially on $\V \fk{g}^{J(l)}$) and the stabilizers are constant along the fibers. So this action is free modulo a family of normal subgroups parametrized by $\V \fk{g}^{J(w)}$. Thus
$$X^\m{nd}_n /_\m{fppf} \bb{U}_n^{(w+1)} = X^\m{nd}_n /_\m{ZAR} \bb{U}_n^{(w+1)}$$
and the situation might be amenable to methods similar to (but much more complicated than) the ones above.
\end{rmk}

\section{Representations of a unipotent group}\label{.52}
\label{Unip}

In this section we put ourselves in the situation indicated by the title of the paper given by a field $k$ of characteristic zero and a unipotent group $G$ over $k$. The problem of moduli of representations of $G$ is equivalent to the problem studied in the previous sections applied to the case $\g := \mathbf{Lie}\, G$. This is largely a matter of reviewing the classical theory. Statements available in the literature, however, focus on representations defined over a field; we explain in detail how to work in families.

\begin{unmarked} \label{.521}
For proofs of the following facts (as well as a discussion of the definition of a unipotent group), we refer the reader to \cite[IV \S2]{dg}. Let $\m{\bld{UG}}$ denote the category of unipotent groups over $k$ and let $\m{\bld{NL}}$ denote the category of nilpotent Lie algebras over $k$. The Lie algebra of a unipotent group is nilpotent. Thus $\m{\bld{Lie}}$ is a functor $\m{\bld{UG}} \rightarrow \m{\bld{NL}}$. On the other hand, if $\fk{g}$ is a nilpotent Lie algebra, its covariant total space may be endowed with a product structure $\star : \V \fk{g}^\lor \times \V \fk{g}^\lor \rightarrow \V \fk{g}^\lor$ given by the Baker-Cambell-Hausdorff formula. This makes $(\V \fk{g} ^\lor, \star)$ into a unipotent group, and defines a functor $\m{\bld{H}}: \m{\bld{NL}} \rightarrow \m{\bld{UG}}$. $\m{\bld{Lie}}$ and $\m{\bld{H}}$ are quasi-inverse (\cite[IV \S2 4.5]{dg}). In particular, there is a natural isomorphism $\exp: \m{\bld{H}} \circ \m{\bld{Lie}} \rightarrow \m{id}_\m{\bld{UG}}$, which is called the \it{exponential map}. 
\end{unmarked}

\begin{unmarked} \label{.5213}
For the remainder of this section, we fix a unipotent group $G$ over $k$ and we let $\fk{g}$ denote its Lie algebra. Recall that formation of the Lie algebra is compatible with flat base-change, so for any $k$-scheme $T$, $\fk{g}_T$ fits into a split short exact sequence of (abstract) groups
\[\xymatrix@1{
1 \ar[r]	&
\Gamma(T, \fk{g}_T) \ar[r]	 &
G(T[\epsilon]) \ar[r] \POS c+(7,1.5) ="t"	&
G(T)	\ar[r] \POS c+(-5,1.5)="s" &
1 \, .
\ar @/_3pt/ "s";"t"	}
\]
Here $T[\epsilon]$ denotes $\m{spec}_T \, \ssf_T [\epsilon]/(\epsilon^2)$. 
\end{unmarked}

We note a few generalities about quasi-coherent representations of a Lie algebra over a general (affine) base (\ref{.5223} -- \ref{.5224}). 

\begin{unmarked} \label{.5223}
Suppose $T = \spec B$ is an affine scheme, $F$ is a $B$-module and $r: B \otimes \fk{g} \rightarrow \m{End}(F)$ is a representation. Then for any $B$-algebra $B'$, $r$ defines a representation $r(B'): B' \otimes \fk{g} \rightarrow \m{End}_{B'}(B' \otimes F)$ determined by the commuting square  
\[\xymatrix{
B \otimes \fk{g} \ar[r]^(.45){r = r(B)} \ar[d] 	&
\m{End}_B(F) \ar[d]	\\
B' \otimes \fk{g} \ar@{.>}[r]_(.35){r(B')} 	&
\m{End}_{B'}(B' \otimes F)	}
\]
and the requirement that $r(B')$ be $B'$-linear. Thus if $\fEnd(F)$ denotes the functor $B' \mapsto \m{End}_{B'}(B' \otimes F)$, then $r$ extends uniquely to a morphism of Lie $o_T$-algebras $\V \fk{g}^\lor_T \rightarrow \fEnd(F)$, which we denote again by $r$.
\end{unmarked}

\begin{unmarked} \label{.5224}
Continuing with the situation of \ref{.5223}, we remark that any vector in the $0$-eigenspace of $r$ is automatically universally in the $0$-eigenspace. That is, if $x \in F$ is such that $r(v)(x) = 0 $ for all $v \in B \otimes \fk{g}$ then for any $B$ algebra $B'$, and any $v' \in B' \otimes \fk{g}$,
$$r(B')(v')(1 \otimes x) = 0\,.$$
Indeed, since $\fk{g}$ is nilpotent and finitely generated, it is finite dimensional. Let $v_1, \dots, v_m$ be a basis and write
$$v' = \sum_i b'_i \otimes v_i$$
with $b'_i \in B'$. Then (identifying $B' \otimes_k \fk{g}$ with $B' \otimes_B B \otimes_k \fk{g}$) we have
\begin{align*}
r(B')(v')(1 \otimes_B x) &= r(B')(\sum b'_i \otimes_k v_i)(1 \otimes_B x) \\
&= \sum b'_i r(B')(1_{B'} \otimes_B 1_B \otimes_k  v_i)(1_{B'} \otimes_B x) \\
&= \sum b'_i \otimes_B (r(B)(1_B \otimes_k v_i)(x)) \\
&= \sum b'_i \otimes_B 0 \\
&= 0 \,.
\end{align*}
\end{unmarked}

\begin{dfn} \label{.523}
Let $B$ be a $k$-algebra and let $r:\fk{g}_B \to \m{End}(F)$ be a representation on a $B$-module $F$. Then $r$ is \textbf{locally nilpotent} if $\m{Fil}^r$ (\ref{.111}) is exhaustive (\ref{.112}).
\end{dfn}

\begin{dfn} \label{.524}
Let $B$ be a $k$-algebra and let $F$ be a $B$-module. We denote by $\fAut F$ the group-valued functor
$$B' \mapsto \m{Aut}_{B'}(B' \otimes F) \,.$$
A representation $\rho$ of $G$ on $F$ (over T) is a morphism of group-valued functors
$$G_T \rightarrow \fAut F \,.$$
The submodule $F^{G_T} \subset F$ of \textit{invariants} is then defined to be the set of universally fixed elements of $F$, that is, those $x \in F$ such that for any $B$-algebra $B'$ and any $u \in G(B')$,
$$\rho(B')(u)(1_{B'} \otimes_B x) = 1_{B'} \otimes_B x \,.$$

We associate to $\rho$ a filtration $\m{Fil}^\rho$ by submodules $F_0 \subset F_1 \subset F_2 \subset \cdots$ of $F$ as for a representation of $\fk{g}$ by setting $F_0 = 0$ and defining $F_{i+1}$ to be the preimage in $F$ of $(F/F_i)^{G_T}$.
\end{dfn}

\Remark{.5243}{Let $\rho:G_T \to \fAut{\cl{F}}$ be a representation of $G$ on a quasi-coherent sheaf $\cl{F}$ over an affine $k$-scheme $T=\spec B$ with structure morphism $f:\spec B \to \spec k$. We denote the $B$-module associated to $\cl{F}$ by $F$ as usual. Then we can define an associated representation
$$f_*\rho:G \to \fAut f_* \cl{F}$$
of $G$ on $f_* \cl{F}$ by forgetting the $B$-linearity of the coaction
$$F \to (B \otimes_k A) \otimes_B F = B \otimes_k F \, .$$
This defines a functor
$$f_*:\mathbf{Rep}\,G_T \to \mathbf{Rep}\,G$$
from the category of quasi-coherent representation of $G_T$ to the category of quasi-coherent representations of $G$ which is exact and satisfies
$$f_*(\cl{F}^{G_T}) = (f_*\cl{F})^G \,.$$
Indeed, both are equal to the kernel of
$$\alpha - \pi: F \to B\otimes_k F$$
where $\alpha$ is the coaction and $\pi$ is the projection $x \mapsto 1_B \otimes_k x$. Consequently,
$$\m{Fil}^\rho = \m{Fil}^{f_*\rho} \,.$$  }

\begin{prop} \label{.525}
Let $T = \spec B$ be an affine $k$-scheme, $F$ a $B$-module, $\rho: G_T \rightarrow \fAut F$ a representation and $\m{Fil}^\rho = (F_0 \subset F_1 \subset \cdots)$ the filtration associated to $\rho$ as in \ref{.524}. Then the filtration $\m{Fil}^\rho$ is exhaustive.
\end{prop}

\begin{proof}
Formation of the associated filtration is compatible with taking subrepresentations. By \cite[II 2.2.2.2]{sr}, every element of $f_*\cl{F}$ is contained in a finite dimensional subrepresentation; by \cite[IV 2.5]{dg}, the filtration associated to a finite dimensional representation over a field is strictly increasing, hence exhaustive.
\end{proof}

We recall the definition and a first property of the derivative of a representation:

\begin{unmarked} \label{.5253}
Let $T = \spec B$ be an affine $k$-scheme and let $\rho: G_T \rightarrow \fAut F$ be a representation of a unipotent group $G$ on a $B$-module $F$. Then $\m{\bld{Lie}}(\rho)$ is the representation $B \otimes \fk{g} \rightarrow \m{End}(F)$ of $\fk{g}$ induced by $\rho(B)$ and $\rho(B[\epsilon])$, forming a morphism of split short exact sequences of abstract groups as in the following diagram:
\[\xymatrix{
1 \ar [r]	&
B \otimes \fk{g} \ar[r] \ar@{.>}[d]_{\m{\bld{Lie}}(\rho)}	 &
G(B[\epsilon]/(\epsilon^2) \ar[r] \ar[d]_{\rho(B[\epsilon]/(\epsilon^2))} \POS c+(11, 1.5)="t1" &
G(B) \ar[r] \ar[d]_{\rho(B)} \POS c+(-6,1.6)="s1" 	 &
1	\\
1 \ar[r]	&
\m{End}(F) \ar[r]	&
\m{Aut}(B[\epsilon]/(\epsilon^2) \otimes F) \ar[r]	\POS c+(18,-1.5)="t2" &
\m{Aut}(F) \ar[r] \POS c+(-8,-1.6)="s2"	&
1	
\ar @/_3.5pt/ "s1";"t1"
\ar @/^3pt/ "s2";"t2"	}
\]
Formation of $\m{\bld{Lie}}(\rho)$ is compatible with arbitrary base-change; that is, given any $B$-algebra $B'$, $\m{\bld{Lie}}(\rho)(B')$ fits into a morphism of split short exact sequences of abstract groups:
\[\xymatrix @ C - 1ex{
1 \ar [r]	&
B' \otimes \fk{g} \ar[r] \ar[d]_{\m{\bld{Lie}}(\rho)(B')}	 &
G(B'[\epsilon]/(\epsilon^2) \ar[r] \ar[d]_{\rho(B'[\epsilon]/(\epsilon^2))} \POS c+(11.5, 1.5)="t1" &
G(B') \ar[r] \ar[d]_{\rho(B')}  \POS c+(-6.5,1.6)="s1" 	 &
1	\\
1 \ar[r]	&
\m{End}(B' \otimes F) \ar[r]	&
\m{Aut}(B' [\epsilon]/(\epsilon^2) \otimes F) \ar[r] \POS c+(18.5,-1.5)="t2"	&
\m{Aut}(B' \otimes F) \ar[r] \POS c+(-12.5,-1.6)="s2"	&
1
\ar @/_4pt/ "s1";"t1"
\ar @/^3pt/ "s2";"t2"	}
\]
\end{unmarked}

\begin{prop} \label{.54}
Let $T$ be an affine $k$-scheme, $\rho: G_T \rightarrow \fAut F$ a quasi-coherent representation of a unipotent group $G$, $r = \m{\bld{Lie}}(\rho)$ the associated representation of the Lie algebra $\fk{g}$ of $G$. As explained in \ref{.5223}, $r$ extends uniquely to a morphism of Lie-algebra-valued functors
$$\V \fk{g}_T^\lor \rightarrow \fEnd(F) \,,$$
hence corresponds to a point $\phi$ of
$$\fEnd(F)(S^\bullet \fk{g}_T^\lor) = \m{End}_{S^\bullet \fk{g}_T^\lor} (S^\bullet \fk{g}_T^\lor \otimes F) \,.$$
On the other hand,
$$\rho \circ \exp: \V \fk{g}_T^\lor \rightarrow G_T \rightarrow \fAut F$$
corresponds to a point 
$$\psi \in \m{Aut}_{S^\bullet \fk{g}_T^\lor} (S^\bullet \fk{g}_T^\lor \otimes F) \,.$$
Then $\phi$ is locally nilpotent and 
\[ \psi = 1 + \phi + \frac{\phi^2}{2} + \frac{\phi^3}{3!} + \cdots \]
\end{prop}

This is standard when $F$ is a vector sheaf. The present situation requires a more careful argument since the functors $\fAut F$ and $\bb{U}_\m{Fil}F$ may not be representable. The proof follows (\ref{.5402} -- \ref{.5404}). We avoid any mention of $\bb{U}_\m{Fil}F$.

\begin{lm} \label{.5402}
Let $B$ be a ring, $F$ a module, $n \in \bb{N}$. Let $B_n = B[T]/T^{n+1}$ and $C_n = B[\epsilon_1, \dots, \epsilon_n]/(\epsilon_1^2, \dots, \epsilon_n^2)$. Then the map 
\[\alpha: \m{Aut}_{B_n}(B_n \otimes F) \rightarrow \m{Aut}_{C_n}(C_n \otimes F)\]
induced by
\[ T \mapsto \epsilon_1 + \cdots + \epsilon_n \]
is injective.
\end{lm}

\begin{proof}
The map $B_n \rightarrow C_n$ is injective with image the subring of invariants for the action of $S_n$ which permutes the variables. The Reynolds operator for this action provides a splitting of the injection regarded as a map of $B$-modules. It is thus universally injective. Now given an automorphism $\phi$ of $B_n \otimes F$, $\phi, \alpha(\phi)$ form a commuting square
\[\xymatrix{
B_n \otimes F \ar[d]_\phi \ar@{^{(}->}[r] &
C_n \otimes F \ar[d]^{\alpha(\phi)} \\
B_n \otimes F \ar@{^{(}->}[r] &
C_n \otimes F }
\]
in which the horizontal maps are injective, from which it follows that $\phi$ is uniquely determined by $\alpha(\phi)$.
\end{proof}

\begin{lm} \label{.5403}
Let $B$ be a ring containing $\bb{Q}$, $G = \spec A$ an algebraic group over $B$, $\fk{g} = \m{\bld{Lie}}(G)$, $\rho: G \rightarrow \fAut F$ a quasi-coherent representation over $T = \spec B$, $r = \m{\bld{Lie}}(\rho): \fk{g} \rightarrow \m{End}(F)$. Denote by $exp$ the formal exponential map
$$\fk{g} \rightarrow G(B[[ T ]])$$
as defined in \cite[II \S6 no. 3]{dg}. Following \cite{dg}, we denote $exp(v)$ by $e^{Tv}$ and given a map $B[[T]] \rightarrow B'$ sending $T \mapsto t \in B'$, we denote the image of $e^{Tv}$ in $G(B')$ by $e^{tv}$. Fix a vector $v \in \fk{g}$ and write
$$\phi := r(v) \,,$$
$$\psi := \rho(B[[T]])(e^{Tv}) \,.$$
Then
$$\psi = \sum \frac {T^i \phi^i}{i!} \,.$$
\end{lm}

\begin{proof}
According to its characterization in \cite{dg}, $exp$ satisfies the following two properties:
\begin{itemize}
\item[(1)] the element $e^{\epsilon v}$ of $G(B[\epsilon]/(\epsilon^2))$ determined by the map $T \mapsto \epsilon$ is also the image of $v$ under
$$\fk{g} \rightarrow G(B[\epsilon]/(\epsilon^2)) \,;$$ and
\item[(2)] $e^{(T+T')v} = e^{Tv}e^{T'v}$ in $G(B[[T, T']])$.
\end{itemize}
Meanwhile, the map
$$\m{End}(F) \rightarrow \m{Aut}(B[\epsilon]/(\epsilon^2) \otimes F)$$
defined by
$$\sigma \mapsto 1 + \epsilon \sigma$$
is injective with cokernel equal to $\m{Aut}(F)$. Since $r$ is defined by the map of short exact sequences
\[\xymatrix{
0 \ar[r] 	&
\fk{g} \ar[r] \ar[d]^r \POS p+(-4,5)*+{v}="v"  	&
G(B[\epsilon]/(\epsilon^2)) \ar[r] \ar[d]^{\rho(B[\epsilon]/(\epsilon^2))} \POS p+(0, 7)*+{e^{\epsilon v}}="e"	&
G(B) \ar[r] \ar[d]^{\rho(B)}	&
0	\\
0 \ar[r]	&
\m{End}( F) \ar[r] 	\POS p+(-10,-6)*+{\phi}="phi"&
\m{Aut}(B[\epsilon]/(\epsilon^2) \otimes  F) \ar[r] \POS p+(0,-7)*+{1+ \epsilon \phi}="1"	&
\m{Aut}( F) \ar[r]	&
0		
\ar@{|->} @/_10pt/  "v";"phi"
\ar@{|->} @/_3pt/  "phi";"1"
\ar@{|->} @/^3pt/ "v";"e"
}
\]
induced by $\rho$, it follows that
$$\rho(B[\epsilon]/(\epsilon^2))(e^{\epsilon v}) = 1 + \epsilon \phi \,.$$
Notation as in \ref{.5402}, the map
$$B[\epsilon]/(\epsilon^2) \rightarrow C_n \text{ defined by } \epsilon \mapsto \epsilon_i$$
gives rise to a commuting square
\[\xymatrix{
G(B[\epsilon]/(\epsilon^2)) \ar[d] \ar[r] \POS p+(-12,7)*+{e^{\epsilon v}}="e"	&
\m{Aut}(B[\epsilon]/(\epsilon^2) \otimes  F) \ar[d] \POS p+(17,7)*+{1+\epsilon \phi}="1"	\\
G(C_n) \ar[r] \POS p+(-12,0)*+{e^{\epsilon_i v}}="i"	&
\m{Aut}(C_n \otimes  F) \POS p+(21,0)*+{1+\epsilon_i \phi}="1i"	
\ar@{|->}@/_5pt/ "e";"i"
\ar@{|->}@/^5pt/ "e";"1"
\ar@{|->}@/^5pt/ "1"+(5,-3);"1i" }
\]
from which it now follows that $\rho(C_n)(e^{\epsilon_i v}) = 1+\epsilon_i \phi$.

Property (2) of $exp$ implies that given nilpotent elements $t,t'$ in a $B$-algebra $B'$, 
$$e^{(t+t')v} = e^{tv}e^{t'v} \,.$$
So the map
$$T \mapsto \epsilon_1 + \cdots + \epsilon_n$$
gives rise to a commuting square 
\[\xymatrix{
G(B[[T]]) \ar[r] \ar[d]	\POS p+(-10,7)*+{e^{Tv}}="e" &
\m{Aut}(B[[T]] \otimes  F) \ar[d] \POS p+(0,7)*+{\psi} ="psi"	\\
G(C_n) \ar[r] \POS p+(-7,-7)*+{e^{\epsilon_1 v} \cdots e^{\epsilon_n v}} = "es"	&
\m{Aut}(C_n \otimes  F) \POS p+(9,-7)*+{(1 + \epsilon_1 \phi)\cdots(1+ \epsilon_n \phi)} = "1s"
\ar@{|->}@/^5pt/"e";"psi"
\ar@{|->}@/_5pt/"e";"es"+(-5,3)	
\ar@{|->}@/_3pt/"es";"1s"}
\]
in which $e^{Tv}$ maps to $\psi$ on the upper right and to $(1 + \epsilon_1 \phi)\cdots(1+ \epsilon_n \phi)$ on the lower right as shown. On the other hand, in the notation of \ref{.5402}, the vertical map on the right factors through
$$\m{Aut}(B_n \otimes  F) \rightarrow \m{Aut}(C_n \otimes  F) \,.$$
By \ref{.5402}, this map is injective. Since in $\m{Aut}(C_n \otimes  F)$, 
\begin{align}
(1+\epsilon_1 \phi) \cdots  & (1+ \epsilon_n \phi) \nonumber \\
&= \sum_i (\m{sum \; of \;} i \m{-fold \; products \; of \; distinct \;} \epsilon_j \m{'s}) \phi^i \nonumber \\
 &= \sum_i \frac{(\epsilon_1 \cdots \epsilon_n)^i}{i!} \phi^i
\end{align}
this map sends
$$\sum \frac {T^i \phi^i}{i!} \text{ to } (1 + \epsilon_1 \phi)\cdots(1+ \epsilon_n \phi) \,.$$
It follows that $\psi$ maps to $\sum \frac {T^i \phi^i}{i!}$ in $\m{Aut}(B_n \otimes  F)$ which concludes the proof of the lemma.
\end{proof}

\begin{unmarked} \label{.5404}
Returning to the situation of the proposition, let
$$B' := B \otimes S^\bullet \fk{g}^\lor$$
and let $v \in \fk{g}_{B'}$ be the universal section. Then $r(B')(v) = \phi$ as defined in the proposition. By \ref{.5403},
$$\rho(B'[[T]])(e^{Tv}) = \sum \frac {T^i\phi^i}{i!} \,.$$
By \cite[IV \S2 4.1]{dg},
$$exp : \fk{g}_{B'} \rightarrow G(B'[[T]])$$
factors through $G(B'[T])$. The situation is summarized in the following diagram.
\[\xymatrix{
\fk{g}_{B'} \ar[r] \ar[d]	 \POS p+(-5,5)*+{v}="v" &
\m{End} (B' \otimes  F) \POS p+(3,6)*+{\phi}="phi"	\\
G(B'[T]) \ar[r] \ar[d]	&
\m{Aut}(B'[T] \otimes  F) \ar[d]	\\
G(B'[[t]]) \ar[r] \POS p+(-9,-6)*+{e^{Tv}}="T"	&
\m{Aut}(B'[[T]]\otimes  F) \POS p+(5,-7)*+{\sum \frac {T^i \phi^i} {i!}}="sum"
\ar@{|->}@/^5pt/"v";"phi"
\ar@{|->}@/_10pt/"v";"T"
\ar@{|->}@/_5pt/"T";"sum"		}
\]
This implies that $\sum \frac {T^i \phi^i}{i!}$ is in $\m{Aut}(B'[T] \otimes  F)$ and in particular that $\phi$ is locally nilpotent. Let $\exp$ denote the global exponential map
$$\V \fk{g}^\lor_B \rightarrow G_B \,.$$
Then by definition, $\exp(B')$ is the composite 
\[\xymatrix @R=0pt{
\fk{g}_{B'}  \ar[r] 	&
G(B'[T]) \ar[r]	&
G(B') 	\\
	&
T \ar@{|->}[r]	&
1 	}
\] 
Finally, $\psi$, as defined in the proposition, equals $\rho(B')(\exp(B')(v))$. So we consider the commuting square 
\[\xymatrix{
G(B'[T]) \ar[r] \ar[d] \POS p+(-8,6)*+{e^{Tv}}="et"	&
\m{Aut}(B'[T] \otimes  F) \ar[d] \POS p+(7, 7)*+{\sum \frac {T^i \phi^i}{i!}}="sum" \POS p+(19,0)*+{T}="T"		\\
G(B') \ar[r] \POS p+(-7,-5)*+{e^v}="ev"	&
\m{Aut}(B' \otimes  F) \POS p+(3,-6)*+{\psi}="psi"	\POS p+(16,0)*+{1}="1"
\ar@{|->}@/^5pt/"et";"sum"
\ar@{|->}@/_10pt/"et";"ev"
\ar@{|->}@/_5pt/"ev";"psi"
\ar@{|->}@/^5pt/"T";"1"	}
\]
from which it follows that
$$\psi = \sum \frac{\phi^i}{i!}$$
as claimed. 
\end{unmarked}

\begin{cor} \label{.541}
Let $T = \spec B$ be an affine $k$-scheme, $\rho : G_T \rightarrow \fAut F$ a quasi-coherent representation, $r = \m{\bld{Lie}}(\rho)$. Let $F^{G_T}$ denote the module of invariants of $\rho$ and let $F^0$ denote the $0$-eigenspace of $r$. Then
$$F^{G_T} = F^0 \,.$$
\end{cor}

The proof follows (\ref{.542} -- \ref{.543}).

\begin{lm} \label{.542}
Let $B$ be a ring, $F$ a $B$-module, $\phi \in \m{End}(F)$ and suppose $\phi$ is locally nilpotent. Let $\psi = \sum \frac {\phi^i}{i!}$ and let $ x\in F$. Then $\phi x = 0 $ if and only if $\psi x = x$.
\end{lm}

\begin{proof}
If $\phi x =0$ then
\begin{align*}
\psi x &= x + \phi(x) + \frac {\phi^2(x)}{2} + \cdots \\
&= x \,.
\end{align*}
If $\psi(x) = x$ then
\begin{align*}
\phi x &= (\log \psi)x \\
&= (\log(1 +(\psi -1)))x \\
&= ((\psi-1) - \frac {(\psi-1)^2}{2} + \frac {(\psi-1)^3}{3} - \cdots)x \\
&= 0 \,.
\end{align*}
\end{proof}

\begin{unmarked} \label{.543}
Returning to the proof of the corollary, suppose $x \in F^{G_T}$, let $v \in B \otimes \fk{g}$ and let $u = \exp (B)(v) \in G(B)$. Then
$$\rho (B)(u)(x) = x$$
and
$$\rho(B)(u) = \sum \frac {r(v)^i}{i!}$$
so
$$r(v)(x) = 0$$
by the lemma.

Conversely, suppose $x \in F^0$, let $B'$ be a $B$-algebra, let $u \in G(B')$ and let $v = \log (B')(u) \in B' \otimes \fk{g}$. As explained in \ref{.525}, $r(B')(v)(x) = 0$. Thus 
\begin{align*}
\rho(B')(u)(x) &= \sum \frac {r(B')(v)^i}{i!}(x) \\
&= x \,.
\end{align*}
\end{unmarked}

\begin{cor} \label{.55}
Let $T$ be an affine $k$-scheme, $\rho : G_T \rightarrow \fAut F$ a quasi-coherent representation, $r = \m{\bld{Lie}}(\rho)$. Then $r$ is locally nilpotent.
\end{cor}

\begin{proof}
It follows from the fact that $\m{\bld{Lie}}$ is exact, from \ref{.525} and from \ref{.541} that the canonical filtration associated to $\rho$ witnesses the local nilpotence of $r$.
\end{proof}

\begin{dfn} \label{.556}
Let $\bb{REP}(G)$ denote the full stack of quasi-coherent representations over the category of affine $k$-schemes. Thus an object is a triple $(T,F,\rho)$, $T = \spec B$ an affine $k$-scheme, $F$ a $B$-module, $\rho: G_T \rightarrow \fAut F$ a representation, and a morphism
$$(T', F',\rho') \rightarrow (T,F,\rho)$$
is a pair $(f,\phi)$, $f$ a morphism $T' \rightarrow T$ and $\phi$ a morphism of representations $f^*\rho \rightarrow \rho'$.

Let $\bb{REP}^\m{nil}(\fk{g})$ denote the fibered category of locally nilpotent representations of $\fk{g}$: an object is a triple $(T,F,r)$ with $T = \spec B$ an affine $k$-scheme, $F$ a $B$-module, $r : B \otimes \fk{g} \rightarrow \m{End}(F)$ a locally nilpotent representation, and a morphism
$$(T',F',r') \rightarrow (T,F,r)$$
is a pair $(f,\phi)$, $f: T' \rightarrow T$ a morphism of $k$-schemes and $\phi: f^*r \rightarrow r'$ a morphism of representations.
\end{dfn}

\begin{thm} \label{.557}
The functor 
$$\m{\bld{Lie}}:\bb{REP}(G) \rightarrow \bb{REP}^\m{nil}(\fk{g})$$
sending a representation to its derivative at the identity is an isomorphism of fibered categories.
\end{thm}

The proof is in paragraphs \ref{.5573}-\ref{.5577}.

\begin{unmarked} \label{.5573}
Compatibility with Cartesian morphisms is clear; so it is enough to fix an affine $k$-scheme $T = \spec B$ and show that
$$\m{\bld{Lie}}(T): \bb{REP}(G)(T) \rightarrow \bb{REP}^\m{nil}(\fk{g})(T)$$
is an isomorphism. We begin by constructing an inverse
$$\Psi: \m{Ob}(\bb{REP}^\m{nil}(\fk{g})(T)) \rightarrow \m{Ob}(\bb{REP}(G)(T))$$
to $\m{Ob}(\m{\bld{Lie}}(T))$.

Let $\m{Fil}$ be an exhaustive increasing filtration indexed by $\bb{N}$ on a $B$-module $F$ and let $\fk{n}_\m{Fil}(F)$ denote the submodule of $\m{End}(F)$ consisting of those endomorphisms which preserve the filtration and induce $0$ on the associated graded. (Note, however, that if F is not finitely presented, formation of $\fk{n}_\m{Fil}(F)$ may not be compatible with flat base-change; so it is better not to think of it as a quasi-coherent sheaf.) Given $v_1, v_2 \in \fk{n}_\m{Fil}(F)$ and $s \in \bb{N}$, all but finitely many terms of $v_1 \star v_2$ are in $(\fk{n}_\m{Fil}(F))^{(s)}$; hence $v_1 \star v_2$ is a locally finite sum. Moreover, 
\begin{equation}
\sum \frac {(v_1 \star v_2)^i}{i!} = ( \sum \frac {v_1^i}{i!} ) (\sum \frac {v_2^i}{i!})
\end{equation}\label{.5575}(\cite[\S6.4]{bourbaki}). 
\end{unmarked}

\begin{unmarked} \label{.5576}
Let $r: \fk{g}_T \rightarrow \m{End}(F)$ be a locally nilpotent representation on a $B$-module $F$ and let $T' = \spec B'$ be an arbitrary affine $T$-scheme. Then by $\ref{.5223}$, $r(T')$ is locally nilpotent. 
Let $\m{Fil}(B')$ denote the associated filtration on $B' \otimes F$. There is thus a factorization of $r(T')$ as
$$B' \otimes \fk{g} \rightarrow \fk{n}_{\m{Fil}(B')}(B' \otimes F) \subset \m{End}_{B'}(B' \otimes F) \,.$$
Now given $u \in G(B')$, define
$$\Psi(r): G_T \rightarrow \fAut F$$
by 
\[ \Psi(r)(B')(u) = \sum_i \frac {(r(B') \log(B') (u))^i} {i!} \]
By \ref{.5575}, $\Psi(r)$ is a morphism of group-valued functors. To check that
$$\m{\bld{Lie}}(T) \circ \Psi = \m{id}_{\m{Ob}(\bb{REP}^\m{nil}(\fk{g})(T))}$$
fix a locally nilpotent representation $r: \fk{g}_T \rightarrow \m{End}(F)$ and consider the following diagram.
\[\xymatrix @R=6ex {
0 \ar[r] \ar[r] 	&
B \otimes \fk{g} \ar[r] \ar[d]^r \POS p+(-6,7)*+{v}="v"	&
G(B[\epsilon]/(\epsilon^2)) \ar[r] \ar[d]^{\Psi(r)(B[\epsilon])}	 \POS p+(12,7)*+{e^{\epsilon v}} = "e" \POS p+(12,1.5)="t1" &
G(B) \ar[r] \ar[d]^{\Psi(r)(B)}  \POS p+(-6,1.6)="s1"	&
0	\\
0 \ar[r]		&
\m{End}(F) \ar[r] \POS p+(-5, -7)*+{r(v)}="r" 	&
\m{Aut}(B[\epsilon]/(\epsilon^2) \otimes F)	\ar[r] \POS p+(-9,-7)*+{1+\epsilon r(v)} = "1"  \POS p+(5,-13)*+{\sum \frac {(r(B[\epsilon])(\log(B[\epsilon])(e^{\epsilon v}))^i} {i!} } = "sum"  \POS p+(18,-1.5)="t2"	&
\m{Aut}(F) \ar[r]	  \POS p+(-7.5,-1.6)="s2"	 &
0
\ar@{|->} @/^5pt/  "v";"e"
\ar@{|->} @/^13pt/   "e" +(2.5,-1.5);"sum"+(14,5)
\ar@{|->} @/_10pt/  "v";"r"+(-4,1.5)
\ar@{|->} @/_2pt/  "r";"1" +(-9.5,-.5)
\ar@/_3pt/ "s1";"t1"
\ar@/^3pt/ "s2";"t2"	}
\]
Our goal being to check that the square on the left commutes, we compute:
\[ \sum \frac {(r(B[\epsilon])(\log(B[\epsilon])(e^{\epsilon v}))^i} {i!} = \sum \frac {(\epsilon r(B)(v))^i}{i!} = 1+ \epsilon r(v) \]

To check that
$$\Psi \circ \m{\bld{Lie}}(T) = \m{id}_{\m{Ob}(\bb{REP}(G)(T))} \,,$$
fix a representation $\rho: G_T \rightarrow \fAut F$, an affine $T$-scheme $T' = \spec B'$ and a point $u \in G(B')$. We then have 
\[ \Psi (\m{\bld{Lie}}(T)(\rho))(T')(u) = \sum \frac {((\m{\bld{Lie}}\; \rho)(B')\log(B')(u))^i} {i!} = \rho(T')(u)\] by \ref{.54}.
\end{unmarked}

\begin{unmarked} \label{.5577}
We've shown that $\m{\bld{Lie}}(T)$ interpolates the invariants and nulspace functors on the nose (\ref{.541}) and that $\m{Ob}(\m{\bld{Lie}}(T))$ is bijective (\ref{.5576}). Finally, note the if one fixes a particular construction of duals and tensor products inside the category of quasi-coherent sheaves on $T$, it makes sense to say that $\m{Ob}(\m{\bld{Lie}})$ respects duals and tensor products.

This gives us, for each $\rho, \rho'$, an equality of sets
\begin{align*}
\m{Hom}(\rho, \rho') &= (\rho^\lor \otimes \rho')^{G_T} \\
&= \m{\bld{Lie}}(\rho^\lor \otimes \rho')^0 \\
&= ((\m{\bld{Lie}} \; \rho)^\lor \otimes \m{\bld{Lie}} \; \rho')^0 \\
&= \m{Hom}(\m{\bld{Lie}} \; \rho, \m{\bld{Lie}} \; \rho')
\end{align*}
compatible with identity elements and composition, hence an isomorphism of categories as claimed.
\end{unmarked}

\Definition{.55773}{We define \textbf{the moduli stack of} $n$\textbf{-dimensional representations of} $G$, denoted $\M_n(G)$, to be the substack of $\bb{REP}(G)$ determined by requiring morphisms of representations to be isomorphisms. }

\Unmarked{.5578}{\textit{Proof of \ref{.38371}.} We have $\Mnflg = \Mnfl(\g/\g^{(n)})$. By \ref{.521}, $\g/\g^{(n)}$ is the Lie algebra of a unipotent group $G$. The full fibered category $\bb{REP}(G)$ of quasi-coherent representations is canonically isomorphic to the fibered category of quasi-coherent sheaves on the classifying stack $BG$ of $G$. The latter is known to obey fppf descent. This implies that $\M_n(G)$ obeys fppf descent. By theorem \ref{.557}, $\Mnfl(\g/\g^{(n)})$ embeds as a fibered subcategory of $\M_n(G)$. Since the flag condition is by its definition local, it follows that $\Mnfl(\g)$ obeys fppf decent. 

Similarly, the nondegeneracy condition is by its definition local, so the same argument shows that $\Mnndg$ obeys fppf descent.\qed}

\section{Compatibility of flag representations}
\label{Com}

The $(n+1)^\m{st}$ moduli space $M\nd_{n+1}$ is naturally fibered over a certain closed subscheme $M_n^\m{cnd}$ of $M\nd_n \times_{p_2, M_{n-1}, p_1} M\nd_n$. In this section we give a construction as well as a modular interpretation of $M_n^\m{cnd}$. This generalizes, for instance, the role played by the diagonal of $\bb{P}\g^\m{ab} \times \bb{P}\g^\m{ab}$ in the construction of $M\nd_3(\g)$ in the case $w(\g) =1$ (\ref{.3901}).

We work over a field $k$ of characteristic zero and work interchangeably with representations of a fixed unipotent group $G$ and with nilpotent representations of its Lie algebra $\fk{g}$. 

\begin{dfn} \label{.56}
There are two natural maps
$$p_1, p_2: \cl{M}_n^\m{fl} \rightrightarrows \cl{M}_{n-1}^\m{fl}$$
given by
$$p_1(r) = r_{n-1}\,, \; p_2(r) = r^1\,.$$
Recall that
$$\cl{M}^\m{fl}_n \underset{p_2, \cl{M}^\m{fl}_{n-1}, p_1} \times \cl{M}^\m{fl}_n$$
may be described as the stack whose objects are $4$-ples $(T, r, r', \phi)$, $T$ an affine $k$-scheme, $r,r'$ flag representations $\fk{g}_T \rightarrow \iEnd \cl{E}$, $\phi$ an isomorphism $r^1 \rightarrow r'_{n-1}$. A morphism
$$(U,s,s',\chi) \rightarrow (T,r,r',\phi)$$
consists of a morphism $f:U\rightarrow T$ and isomorphisms $f^*r \rightarrow s, f^*r' \rightarrow s'$ such that the square
\[\xymatrix{
f^*(r^1) \ar[d]_{f^*\phi} \ar@{=}[r]	&
(f^*r)^1 \ar[r]	&
s^1 \ar[d]^\chi	\\
f^*(r'_{n-1}) \ar@{=}[r]	&
(f^*r')_{n-1} \ar[r]	&
**[r]s'_{n-1}	}
\]
commutes. Then there is a natural map
$$p = (p_1, p_2): \cl{M}^\m{fl}_{n+1} \rightarrow \cl{M}^\m{fl}_n \underset{\cl{M}^\m{fl}_{n-1}} \times \cl{M}^\m{fl}_n$$
sending $r$ over $T$ to $(T, r_n, r^1, \phi)$ where $\phi$ is the canonical isomorphism $(r_n)^1 \rightarrow (r^1)_{n-1}$.
\end{dfn}

\begin{defprop} \label{.57}
The image of $p$ is the same in the indiscrete and fppf topologies, and is a closed substack of $\cl{M}^\m{fl}_n \times_{\cl{M}^\m{fl}_{n-1}} \cl{M}^\m{fl}_n$. Denote it by $\cl{M}_n^\m{cfl}$. We call an object in the essential image of $p$ \textit{a compatible pair of flag representations}.
\end{defprop}

The proof is in paragraphs \ref{.58} -- \ref{.61}.

\begin{unmarked} \label{.58}
Consider first the indiscrete topology in which the image is just the essential image of the functor. An object $(T, r, r', \phi)$ of $\cl{M}^\m{fl}_n \times_{\cl{M}^\m{fl}_{n-1}} \cl{M}^\m{fl}_n$ gives rise to a two step extension
\[ 0 \rightarrow r_1 \rightarrow r \rightarrow r' \rightarrow r'^{n-1} \rightarrow 0\]
hence to a class $o(T,r,r',\phi) \in \m{Ext}^2_{\m{Rep} \; G_T} (r'^{n-1}, r_1)$. Then $(T,r,r',\phi)$ is in the essential image of $p$ if and only if $o(T,r,r', \phi) = 0$. Indeed, $(T,r,r',\phi)$ is in the essential image of $p$ if and only if there exists a quasi-coherent (necessarily $n+1$-dimensional, flag) representation $r''$ over $T$ and isomorphisms $\alpha:r\rightarrow r''_n$, $\beta:r''^1\rightarrow r'$ forming a commuting pentagon:
\begin{displaymath}
\xymatrix{
		& 0 \ar[d]						&			& 0 \ar[d]				&						&			\\
		& r_1 \ar@{=}[rr] \ar[d]			&			& r_1 \ar[d]			&						&			\\
0 \ar[r]	& r \ar[rr]^\alpha \ar[d]			&			& r'' \ar[r] \ar[dd]_\beta	& r'^{n-1} \ar@{=}[dd] \ar[r] 	& 0			\\
		& r^1 \ar[dd] \ar[dr]_\phi^\sim 	&			&					&						&			\\
0 \ar[rr]	&							& r'_{n-1} \ar[r]	& r' \ar[r] \ar[d]			& r'^{n-1} \ar[r]			& 0			\\
		& 0							&			& 0	 				&						&			}
\end{displaymath}
(The rest of the diagram is automatic.) The vertical extension on the left gives rise to a long exact sequence 
\[ \cdots \to
\m{Ext}^1(r'^{n-1}, r) \xrightarrow{\pi_*}	
\m{Ext}^1(r'^{n-1}, r^1) \xrightarrow{\delta}	
\m{Ext}^2(r'^{n-1}, r_1) \to	
\cdots	\]
Under $\pi_*$ followed by $\delta$ the class of
\[0 \rightarrow r \rightarrow r'' \rightarrow r'^{n-1} \rightarrow 0\]
goes to 
\[0 \rightarrow r^1 \rightarrow r' \rightarrow r'^{n-1} \rightarrow 0\]
goes to
\[ 0 \rightarrow r_1 \rightarrow r \rightarrow r' \rightarrow r'^{n-1} \rightarrow 0 \,.\]
Thus $(r'', \alpha, \beta)$ as above exists if and only if the class the two-step extension is zero.
\end{unmarked}

\begin{unmarked} \label{.59}
Suppose now that $(r,r',\phi) \in (\cl{M}^\m{fl}_n \times_{\cl{M}^\m{fl}_{n-1}} \cl{M}^\m{fl}_n)(T)$ is fppf-locally in the essential image of $p$. Then there is an fppf map $f:U\rightarrow T$ such that $f^*(r,r',\phi)$ is compatible. Hence the class of 
\[ 0 \rightarrow (f^*r)_1 \rightarrow f^*r \rightarrow f^*r' \rightarrow (f^*r')^{n-1} \rightarrow 0 \]
in $\m{Ext}^2((f^*r')^{n-1}, (f^*r)_1)$ is zero. This corresponds to the class of 
\[ 0 \rightarrow f^*(r_1) \rightarrow f^*r \rightarrow f^*r' \rightarrow f^*(r'^{n-1}) \rightarrow 0 \]
in $\m{Ext}^2(f^*(r'^{n-1}), f^*(r_1))$ which is the image of $o(r,r', \phi)$ under 
\[ 
\m{Ext}^2(f^*(r'^{n-1}), f^*(r_1)) \xleftarrow{\cong}
f^*\m{Ext}^2(r'^{n-1}, r_1) \leftarrow
\m{Ext}^2(r'^{n-1}, r_1)
\]
It follows that $o(r,r',\phi) = 0$, hence that $(T,r,r',\phi)$ is in the essential image of $p$ by \ref{.58}. This shows that the image is the same in the indiscrete and fppf topologies.
\end{unmarked}

\begin{unmarked} \label{.6}
Suppose $(r,r',\phi) \in (\cl{M}_n^\m{fl} \times_{\cl{M}_{n-1}^\m{fl}} \cl{M}_n^\m{fl})(T)$ and denote by $L'_n, L_1$ the line sheaves corresponding to $r'^{n-1}$, $r_1$ respectively. We claim that
$$\m{Ext}^2(r'^{n-1},r_1) = \m{H}^2(G, k) \otimes L'^\lor_n \otimes L_1.$$
\end{unmarked}

\begin{proof}
We denote by $\m{inv}_*$ the functor which takes a quasi-coherent representation to its sheaf of invariants, and by $\m{inv}^*$ the functor which endows a quasi-coherent sheaf with the trivial group action. We have
\begin{align}
\m{Ext}^2(r'^{n-1}, r_1) & = \m{H}^2(G_T, (r'^{n-1})^\lor \otimes r_1) \nonumber \\
	& = R^2 \m{inv}_* \m{inv}^*(L'^\lor_n \otimes L_1) \nonumber
\end{align} 
since both representations are trivial
\begin{align}
 = R^2\m{inv}_*(1_T) \otimes L'^\lor_n \otimes L_1 \nonumber
\end{align}
by the projection formula (here $1_T$ denotes the trivial representation on $\ssf_T$)
\begin{align}
= \m{H}^2(G,k) \otimes L'^\lor_n \otimes L_1 \nonumber
\end{align}
by compatibility with flat base change. 
\end{proof}

\begin{unmarked} \label{.61}
By \ref{.39853}, there is a closed subscheme $Z \subset T$ representing the vanishing locus of $o(T, r,r',\phi)$. We claim that there is a Cartesian square
\[\entrymodifiers={!! <0pt, .8ex>+}
\xymatrix{
Z \ar[d] \ar[r]	&
\cl{M}_n^\m{cfl} \ar[d]	\\
T \ar[r]	&
\cl{M}^\m{fl}_n \underset{\cl{M}^\m{fl}_{n-1}}{\times} \cl{M}^\m{fl}_n	}\]
Indeed, given $f: T' \rightarrow T$, $f^*(r,r',\phi)$ forms a compatible pair if and only if
$$0 = o(f^*(r,r',\phi)) = f^*(o(r,r',\phi))$$
if and only if $f$ factors through $Z$. This completes the proof of \ref{.57}.
\end{unmarked}

\begin{dfn} \label{.62}
We define $\cl{M}^\m{cnd}_n$, \textbf{the stack of compatible pairs of nondegenerate nilpotent representations of dimension} $n$, by the Cartesian square
\[\entrymodifiers={!! <0pt, .8ex>+}
\xymatrix{
\cl{M}^\m{cnd}_n \ar@{_{(}->}[d] \ar@{_{(}->}[r]	&
\cl{M}^\m{cfl}_n \ar@{_{(}->}[d]	\\
\cl{M}\nd_n \underset{\cl{M}\nd_{n-1}}{\times} \cl{M}\nd_n \ar@{^{(}->}[r] 	&
\cl{M}^\m{fl}_n \underset{\cl{M}^\m{fl}_{n-1}}{\times} \cl{M}^\m{fl}_n	}
\]
We define $M_n^\m{cnd}$, \textbf{the moduli space of compatible pairs of nondegenerate nilpotent representations of dimension} $n$, to be the subfunctor of $M\nd_n \times_{M\nd_{n-1}} M\nd_n$ whose $T$-valued points are those pairs $(r,r')$ such that given any square
\[\entrymodifiers={!! <0pt, .8ex>+}
\xymatrix{
T' \ar[r]^(.3){(s,s',\phi)} \ar[d] 	&
\cl{M}\nd_n \underset{\cl{M}\nd_{n-1}}{\times} \cl{M}\nd_n \ar[d]	\\
T \ar[r]_(.3){(r,r')}	&
M\nd_n \underset{M\nd_{n-1}}{\times} M\nd_n	}
\]
$(s,s',\phi)$ is in (the essential image of) $\cl{M}_n^\m{cnd}$.
\end{dfn}

\begin{prop} \label{.63}
The functor $M_n^\m{cnd}$ is a closed algebraic subspace of $M\nd_n \times_{M\nd_{n-1}} M\nd_n$. The natural map $M\nd_{n+1} \rightarrow M\nd_n \times_{M\nd_{n-1}} M\nd_n$ factors through $M_n^\m{cnd}$:
\[\entrymodifiers={!! <0pt, .8ex>+}
\xymatrix{
M\nd_{n+1} \ar[d]	&
	\\
M^\m{cnd}_n \ar@{^{(}->}[r]	&
M\nd_n \underset{M\nd_{n-1}}{\times} M\nd_n	}
\]
\end{prop}

The proof follows (\ref{.64} -- \ref{.67})

\begin{unmarked} \label{.64}
The morphism
$$\cl{M}\nd_n \times_{\cl{M}\nd_{n-1}} \cl{M}\nd_n \rightarrow \cl{M}\nd_n \times_{M\nd_{n-1}} \cl{M}\nd_n$$
is represented by $\bb{I}\m{som}(r^1,r'_{n-1})$. That is, suppose $T \rightarrow \cl{M}\nd_n \times_{M\nd_{n-1}} \cl{M}\nd_n$  corresponds to the pair of representations $(r,r')$ with $r^1, r'_{n-1}$ fppf-locally isomorphic and consider the fibered product
\[\entrymodifiers={!! <0pt, .8ex>+}
\xymatrix{
Y \ar[r] \ar[d]	&
\cl{M}\nd_n \underset{\cl{M}\nd_{n-1}}{\times} \cl{M}\nd_n \ar[d]	\\
T \ar[r]	&
\cl{M}\nd_n \underset{M\nd_{n-1}}{\times} \cl{M}\nd_n	}
\]
Objects of $Y$ are $6$-tuples $(T', f, s, s', \phi, \psi)$, $T'$ a $k$-scheme, $f$ a map $T' \rightarrow T$, $(s,s',\phi) \in (\cl{M}\nd_n \times_{\cl{M}\nd_{n-1}} \cl{M}\nd_n)(T')$ and $\psi$ an isomorphism $f^*(r,r') \rightarrow (s,s')$. Given two objects
$$a_1 = (T', f_1, s_1, s'_1, \phi_1, \psi_1)\,,\; a_2 = (T', f_2, s_2, s'_2, \phi_2, \psi_2)$$
over $T'$, there is exactly one isomorphism $a_2 \rightarrow a_1$ over $\m{id}_{T'}$ if $f_1 = f_2$ and the induced isomorphism
$$(s_2, s'_2) \rightarrow (s_1,s'_1)$$
respects $\phi_1, \phi_2$ and no isomorphisms otherwise. Then, on the one hand, $Y$ is equivalent to the full subcategory consisting of those objects such that
$$(s,s') = f^*(r,r') \text{ and } \phi = \m{id}_{f^*(r,r')}\,,$$
which, on the other hand, is clearly equivalent to $\bb{I}\m{som}_{T'}(r^1, r'_{n-1})$.

Since $r^1, r'_{n-1}$ are fppf-locally isomorphic, $\bb{I}\m{som}_{T'}(r^1, r'_{n-1})$ is an fppf-torsor under $\fAut r^1$. Hence, in particular, the morphism
$$\cl{M}\nd_n \times_{\cl{M}\nd_{n-1}} \cl{M}\nd_n \rightarrow \cl{M}\nd_n \times_{M\nd_{n-1}} \cl{M}\nd_n$$
is flat and locally of finite presentation.
\end{unmarked}

\begin{unmarked} \label{.65}
The morphism
$$h:\cl{M}\nd_n \times_{M\nd_{n-1}} \cl{M}\nd_n \rightarrow M\nd_n \times_{M\nd_{n-1}} M\nd_n$$
is represented locally by the classifying stack of a smooth group scheme. Indeed, if $(r,r') \in (\cl{M}\nd_n \times_{M\nd_{n-1}} \cl{M}\nd_n)(T)$ then
$$\fAut(r,r') = \fAut r \times_T \fAut r' $$
is smooth. Moreover, 
$$M\nd_n \times_{M\nd_{n-1}} M\nd_n = \pi^\m{fppf}_0(\cl{M}\nd_n \underset{M\nd_{n-1}} \times \cl{M}\nd_n) \,.$$
It follows that $h$ is a rigidification map, hence locally a classifying stack by \cite[Remark 1.5.6]{olsson} as claimed. Hence $h$ is flat and locally of finite presentation.
\end{unmarked}

\begin{claim} \label{.66}
There exists a map $\cl{M}_n^\m{cnd} \to M_n^\m{cnd}$ which forms a Cartesian square
\[\entrymodifiers={!! <0pt, .8ex>+}
\xymatrix{
\cl{M}_n^\m{cnd} \ar@{^{(}->}[d] \ar[r]	&
M_n^\m{cnd} \ar@{^{(}->}[d]	\\
\cl{M}\nd \underset{\cl{M}\nd_{n-1}}{\times} \cl{M}\nd_n \ar[r]	&
M\nd_n \underset {M\nd_{n-1}} {\times}  M\nd_n	}
\]
\end{claim}

\begin{proof}
Let $\cl{Y}$ denote the fibered product
\[\entrymodifiers={!! <0pt, .8ex>+}
\xymatrix{
\cl{Y} \ar@{^{(}->}[d] \ar[r]	&
M_n^\m{cnd} \ar@{^{(}->}[d]	\\
\cl{M}\nd \underset  {\cl{M}\nd_{n-1}} {\times}   \cl{M}\nd_n \ar[r]	&
M\nd_n \underset  {M\nd_{n-1}} {\times}   M\nd_n	}
\]
$\cl{Y}$ is the full subcategory of $\cl{M}\nd_n \times_{\cl{M}\nd_{n-1}} \cl{M}\nd_n$ consisting of those objects whose image in $M\nd_n \times_{M\nd_{n-1}} M\nd_n$ lies in $M_n^\m{cnd}$. Thus $\cl{M}_n^\m{cnd}, \cl{Y}$ are both full subcategories, and we claim that they are in fact equal, the inclusion $\cl{Y} \subset \cl{M}_n^\m{cnd}$ being clear.

For the reverse inclusion we are to consider a square 
\[\entrymodifiers={!! <0pt, .8ex>+}
\xymatrix{
T' \ar[rrr]^(.4){(s,s',\chi)} \ar[d] \ar[dr]^{(r,r',\phi)}	&
	&
	&
\cl{M}\nd_n  \underset {\cl{M}\nd_{n-1}}  {\times}  \cl{M}\nd_n \ar[d]	\\
T \ar[r]	&
\cl{M}_n^\m{cnd} \ar[r]	&
\cl{M}\nd_n  \underset {\cl{M}\nd_{n-1}}  {\times}   \cl{M}\nd_n \ar[r]	&
M\nd_n   \underset {M\nd_{n-1}}  \times  M\nd_n 	}
\]
and to show that $(s,s',\chi) \in \cl{M}_n^\m{cnd}(T')$. After possibly replacing $T'$ by an fppf cover, we may assume $s \cong r, s' \cong r'$. Fixing isomorphisms as in the diagram below
\[\xymatrix{
	&
0 \ar[d]	&
0 \ar[d]	&
	&
	&
	&
	\\
	&
s_1 \ar[r]^\cong \ar[d]	&
r_1 \ar[d]	&
	&
	&
	&
	\\
	&
s \ar[r]^\cong \ar[d]	&
r \ar[d]	&
	&
	&
	&
	\\
	&
s^1 \ar[r]^\cong \ar[ddd]	&
r^1 \ar[ddd] \ar[dr]^\phi	&
	&
	&
	&
	\\
0 \ar[rrr]	&
	&
	&
r'_{n-1} \ar[d]^\cong \ar[r]	&
r'\ar[d]^\cong \ar[r]	&
r'^{n-1} \ar[d]^\cong \ar[r]	&
0	\\
0 \ar[rrr]	&
	&
	&
s'_{n-1} \ar[r]	&
s' \ar[r]	&
s'^{n-1} \ar[r]	&
0	\\
	&
0	&
0	&
	&
	&
	&
	} 
\]
We claim that $(s,s', s^1 \rightarrow r^1 \rightarrow r'_{n-1} \rightarrow s'_{n-1})$ is a compatible pair. Indeed, the square
\[\xymatrix{
\m{Ext}^1(s'^{n-1}, r^1) \ar[r]^\cong \ar[d]_\delta	&
\m{Ext}^1(r'^{n-1}, r^1) \ar[d]^\delta	\\
\m{Ext}^2(s'^{n-1}, r_1) \ar[r]_\cong	&
\m{Ext}^2(r'^{n-1}, r_1)	}
\]
commutes and the arrow at the top sends the class of
\[0 \rightarrow r^1 \rightarrow s' \rightarrow s'^{n-1} \rightarrow 0 \]
to the class of 
\[0 \rightarrow r^1 \rightarrow r' \rightarrow r'_{n-1} \rightarrow 0\]
from which it follows that $(r, s', r^1 \rightarrow r'_{n-1} \rightarrow s'_{n-1})$ is a compatible pair; the next step follows similarly from commutativity of the square 
\[\xymatrix{
\m{Ext}^1(s'^{n-1}, s^1) \ar[r]^\cong \ar[d]	&
\m{Ext}^1(s'^{n-1}, r^1) \ar[d]	\\
\m{Ext}^2(s'^{n-1}, s_1) \ar[r]_\cong	&
\m{Ext}^2(s'^{n-1}, r_1)	}
\]
Finally, $s^1 \rightarrow r^1 \rightarrow r'_{n-1} \rightarrow s'_{s-1}$ differs from $\chi$ by an automorphism of $s^1$ from which it follows that $(s,s',\chi)$ is compatible, again by a similar argument.
\end{proof}

\begin{unmarked} \label{.67}
We've shown that the inclusion
$$M_n^\m{cnd} \hookrightarrow M\nd_n \times_{M\nd_{n-1}} M\nd_n$$
is a closed immersion by checking fppf locally on the target. Finally, the factorization follows from the universal mapping property of the fppf sheaf associated to a stack.
\end{unmarked}

\appendix
\section{A direct proof of separatedness}
\label{Sep}

Let $\g$ be a finitely generated Lie algebra over a field $k$ and let $w$ denote its width (\ref{.26933}). It follows from Theorem \ref{.50149} that $M\nnd$, for $n \le w+1$, is separated. The assumption $n \le w+1$ also enables us to give a direct proof of separatedness. Of interest here is the reappearance of the width. The proof is in four paragraphs (\ref{.502} -- \ref{.5024}) staring with concrete calculations over a discrete valuation ring followed by generalities concerning the reduction of separatedness of a rigidification to a certain condition involving only isomorphism classes of objects of the stack.

\begin{lm} \label{.502}
Let $V$ be a valuation ring over $k$ with fraction field $K$, and let $r,r':\fk{g}\rightarrow \m{End}(E), \m{End}(E')$ be two nondegenerate nilpotent representations. If $r_K\cong r'_K$ then $r\cong r'$. 
\end{lm}

\begin{proof}
Fix bases compatible with the flags. Then there are elements $t\in\bb{T}_n(K)$ and $u\in\bb{U}_n(K)$ such that $ut$ is an isomorphism $r_K\rightarrow r'_K$. Our goal is to show that after possibly modifying $t$ and $u$ by precomposing with any automorphism of $r_K$, we have $t \in \bb{T}_n(V)$ and $u \in \bb{U}_n(V)$.

Denote by $\la_{i,j},\la'_{i,j}$ the full matrix entries $\g_V \to V$ of $r,r'$, respectively. Then we have, for $i=1,\dots,n-1$,
$$t_{i,i}\la_{i,i+1} = t_{i+1,i+1}\la'_{i,i+1} \,.$$
Since $\la_{i,i+1}$ is surjective, there is a $v \in \fk{g}_V$ such that $\la_{i,i+1}v=1$. Plugging in and taking valuations, we get
$$\m{val}(t_{i,i}) = \m{val}(t_{i+1,i+1}) + \m{val}(\la'_{i,i+1}(v))$$
from which, $\m{val}(t_{i,i}) \ge \m{val}(t_{i+1,i+1})$; by symmetry, $\m{val}(t_{i,i}) = \m{val}(t_{i+1,i+1})$. Hence, after possibly multiplying by a suitable constant, we get $t \in \bb{T}_n(V)$.

This completes the first half of our task, and provides us, moreover, with an intermediate representation $r''$ over $V$, the conjugate of $r$ by $t$, such that $t$ is an isomorphism $r \rightarrow r''$ and $u$ is an isomorphism $r''_K \rightarrow r'_K$. 

We claim that $u \in \bb{U}_n(V)$, without requiring further modification. Indeed, writing $\la''_{i,j}$ for the component linear functionals of $r''$, equivariance of $u$ reads 
\begin{align*}
\la_{i,j} + u_{i,i+1}\la_{i+1,j} + \cdots &+ u_{i,j-1}\la_{j-1,j} \\
&= u_{i+1,j}\la''_{i,i+1} + \cdots + u_{j-1,j}\la''_{i,j-1} + \la''_{i,j} \,.
\end{align*}
Note, in particular, that for each $i$, 
$$\la_{i,i+1} = \la''_{i,i+1} \,.$$
Assume for an induction on $s$ that for each $j-i < s-1$, $u_{i,j} \in V$. Applying the above equation with $j-i=s$ to any $v \in \fk{g}_V$, we see that 
\begin{equation} \label{.50201}
u_{i,j-1}\la_{j-1,j}(v) - u_{i+1,j}\la_{i,i+1}(v) \in V \,.
\end{equation}
But since $r$ is wide,
$$\la_{j-1,j} \times \la_{i,i+1}: \fk{g}_V \rightarrow V \times V$$
is surjective. We apply \eqref{.50201} to elements $v_1, v_2 \in \fk{g}_V$ mapping to $(1,0), (0,1)$, respectively, to conclude the proof.
\end{proof}

\begin{unmarked} \label{.5021}
Note that $M\nd_n(\fk{g})$ is locally of finite type over $\spec k$. Indeed, by \cite[I.4.11]{knutson} this may be checked fppf locally on the source. The map $X_n^\m{nd} \rightarrow \cl{M}\nd_n$ is a $\bb{B}_n$-torsor, hence fppf, and the map $\cl{M}\nd_n \rightarrow M\nd_n$ was noted to be fppf in \ref{.3982}. So in checking that $\mnndg$ is locally of finite type, we may first replace $\mnndg$ by $X\nd_n$ and we then note that the latter is quasi-projective over $k$, hence, in particular, locally of finite type over $k$.
\end{unmarked}

\begin{unmarked} \label{.5022}
Abstracting a bit, we have an algebraic stack $\cl{X}$ and a map $f: \cl{X} \rightarrow X$ to an algebraic space which is surjective in the fppf topology. Moreover, $X$ is locally of finite type over a locally Noetherian base $S$. We've shown that if $V$ is a valuation ring over $S$ and $K$ is its fraction field, then given $r,r' : \spec V \rightarrow \cl{X}$ over $S$, $r_K \cong r'_K$ implies $r \cong r'$. By \cite[A3]{lm}, the valuative criterion for separatedness for $X$ may be checked against complete discrete valuation rings. So suppose given $r, r': T \rightarrow X$, $T = \spec V$ a complete discrete valuation ring with generic point $\tau = \spec K$ and closed point $t = \spec k$. Then there exists an fppf $T$-scheme $Y$ and a lifting of $r, r'$ to $s, s': Y \rightarrow \cl{X}$.
\end{unmarked}

In order to be able to apply \ref{.502} to $s,s'$, we will need to replace $Y$ by an appropriate valuation ring.

\begin{claim} \label{.5023}
If $Y$ is a flat, finite type $T$-scheme, $T = \spec V$, $V$ a complete discrete valuation ring, then there exists a complete discrete valuation ring $V'$, flat and finite over $V$ and a map
$$T' = \spec V' \rightarrow Y$$
over $V$. 
\end{claim}

A proof based on global geometric techniques was outlined by Hassett in oral communication. Below we follow an alternative approach.

\begin{proof}
Since the special fiber $Y_t$ is of finite type over $t$, it contains a point $y$, finite over $t$. The stalk $\ssf_{Y,y}$ of $\ssf_Y$ at $y$ is flat over $V$, hence injects into $K \otimes_V \ssf_{Y,y}$ ($K$ the fraction field of $V$). Hence in particular $K \otimes_V \ssf_{Y,y}$ is nonzero. An arbitrary point $\zeta$ of its spectrum gives rise to a point of $Y_\tau$ which specializes to $y$. By \cite[Corollary 2.4]{osserman} the specialization
$$\zeta \rightsquigarrow y$$
admits a factorization
$$\zeta \rightsquigarrow \xi \rightsquigarrow y$$
through a closed point $\xi$ of $Y_\tau$. Since $Y_\tau$ is of finite type over $\tau$, $\xi$ is finite over $\tau$. Let $Z$ be the closed subscheme of $Y$ defined by endowing $\bar { \{ \xi \} }$ with the reduced induce structure. Then $\ssf_{Z,y}$ is a Noetherian local domain with generic point $\xi$. Hence by \cite[II Exercise 4.11]{hartshorne}, $K' := k(\xi)$ has a discrete valuation ring $V'$ dominating $\ssf_{Z,y}$. Its spectrum $T' = \spec V'$ surjects onto $T$, hence is torsion free hence flat. On the other hand, since $K'$ is finite over $K$, it has a unique discrete valuation ring dominating $V$ which may be characterized as the integral closure of $V$ in the residue field of $\xi$; moreover this discrete valuation ring is automatically complete (\cite[XII 2.5]{lang}). Thus $T'$ is complete and fppf over $T$ and has a map $T' \rightarrow Y$ as claimed.
\end{proof}

\begin{unmarked} \label{.5024}
We thus have a complete discrete valuation ring $V'$, and sections
$$u, u' : T' := \spec V' \rightarrow \cl{X}$$
lifting $r, r'$. Denoting the function field of $V'$ by $K'$, these have the property that $u_{K'}, u'_{K'}$ agree in $X$ hence are fppf locally isomorphic in $\cl{X}$. Thus there is a finite extension $K''$ of $K'$ such that $u_{K''} \cong u'_{K''}$. There exists a discrete valuation ring $V''$ of $K''$ (complete and) fppf over $V'$. Let $v = u_{V''}, v' = u'_{V''}$ (that is, $v, v'$ are the composites
$$T'' = \spec V'' \rightarrow T' \rightrightarrows \cl{X}\,).$$
Then $v,v'$ are valuation valued points of $\cl{X}$ which are generically isomorphic, so, as we assumed in \ref{.5022}, it follows that they are isomorphic. Hence they agree in $X$. Hence our original $r, r'$ agree fppf locally, hence agree globally and we're done.
\end{unmarked}

\bigskip

\scriptsize\textsc{Leibniz Universit\"at Hannover\\
\indent Institut f\"ur Algebra, Zahlentheorie und Diskrete Mathematik\\
\indent Fakult\"at f\"ur Mathematik und Physik\\
\indent Welfengarten 1\\
\indent 30167 Hannover\\
\indent Germany}

\smallskip

\textit{E-mail address:} \texttt{dan-cohen@math.uni-hannover.de}

\end{document}